\documentclass[final,leqno]{article}
\usepackage[top=1in, bottom=1.25in, left=1.25in, right=1.25in]{geometry}
\usepackage{amssymb}
\usepackage{amsmath}
\usepackage{amsthm}
\usepackage{graphicx}
\usepackage{varwidth}
\usepackage{float}
\usepackage{array}
\usepackage{bbm}
\usepackage{subfig}
\usepackage{grffile}
\newcommand{\R}{\mathbb{R}}
\newcommand{\norm}[1]{\Vert #1 \Vert}

\usepackage{color}
\newtheorem{theorem}{Theorem}
\newtheorem{remark}[theorem]{Remark}
\newtheorem{example}{Example}
\newtheorem{lemma}{Lemma}
\newtheorem{proposition}{Proposition}
\newtheorem{conclusion}{Conclusion}
\newtheorem{definition}{Definition}
\def\p{\partial}
\DeclareMathOperator*{\argmin}{arg\,min}

\title{Spectral Decompositions using One-Homogeneous Functionals}
\author{Martin Burger\footnote{The first three authors contributed equally.}{ }\thanks{Institut f\"ur Numerische und Angewandte Mathematik, Westf\"alische Wilhelms-Universit\"at M\"unster, Einsteinstr. 62, 48149 M\"unster, Germany} \and Guy Gilboa\footnotemark[1]{ }\thanks{Electrical Engineering Department, Technion -– Israel Institute of Technology, Haifa 32000, Israel} \and Michael Moeller\footnotemark[1]{ }\thanks{ Department of Computer Science, Technische Universit\"at M\"unchen, Germany} \and Lina Eckardt\footnotemark[2] \and Daniel Cremers \footnotemark[4]}

\begin{document}
\maketitle

\begin{abstract}
This paper discusses the use of absolutely one-homogeneous regularization functionals in a variational, scale space, and inverse scale space setting to define a nonlinear spectral decomposition of input data. We present several theoretical results that explain the relation between the different definitions. Additionally, results on the orthogonality of the decomposition, a Parseval-type identity and the notion of generalized (nonlinear) eigenvectors closely link our nonlinear multiscale decompositions to the well-known linear filtering theory. Numerical results are used to illustrate our findings.
\end{abstract}

\section{Introduction}
\label{sec:Introduction}
One of the most important and successful concepts in digital image and signal processing are changes of representation of an input signal to analyze and manipulate particular features that are well separated in a certain basis. For instance, soundwaves are often represented as a superposition of sine and cosine. A standard application is illustrated in Figure \ref{fig:classical} below: We are given a noisy discrete input signal $f\in \R^n$ (Fig. b) along with the prior knowledge that the desired noise-free signal (Fig. a) contains rather low frequencies. Thus, we use the discrete cosine transform (DCT), and apply an ideal low pass filter (Fig. d) to the resulting representation (Fig. c). We obtain the filtered coefficients (Fig. e) and, after inverting the DCT, obtain the final result shown in Fig. f.

\newcommand{\dctFigureWidth}{0.32\textwidth}
\begin{figure}[h]
\begin{center}
\begin{minipage}[t]{\dctFigureWidth}
\begin{center}
\includegraphics[width=1\textwidth]{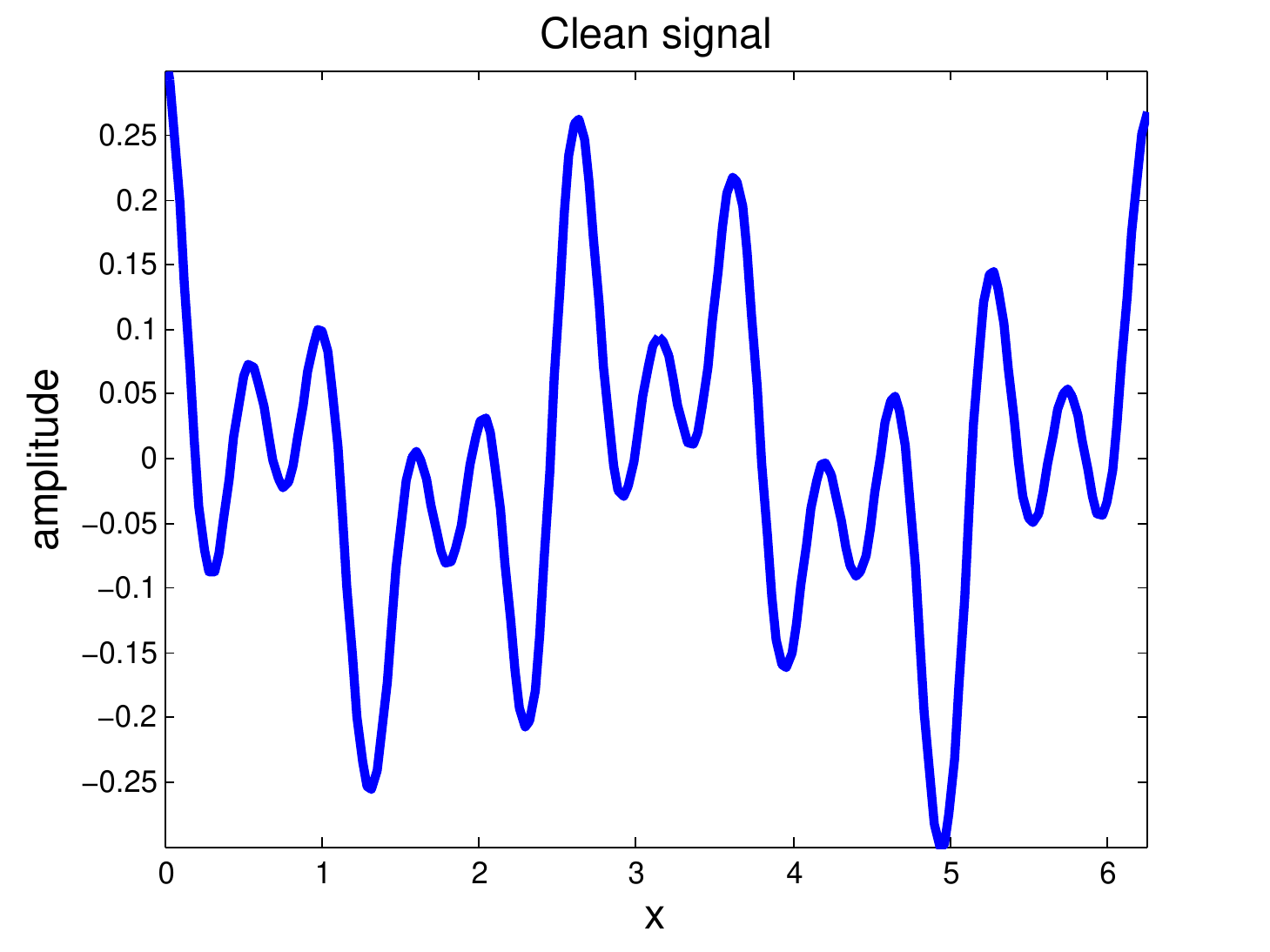}
\end{center}
\end{minipage}
\begin{minipage}[t]{\dctFigureWidth}
\begin{center}
\includegraphics[width=1\textwidth]{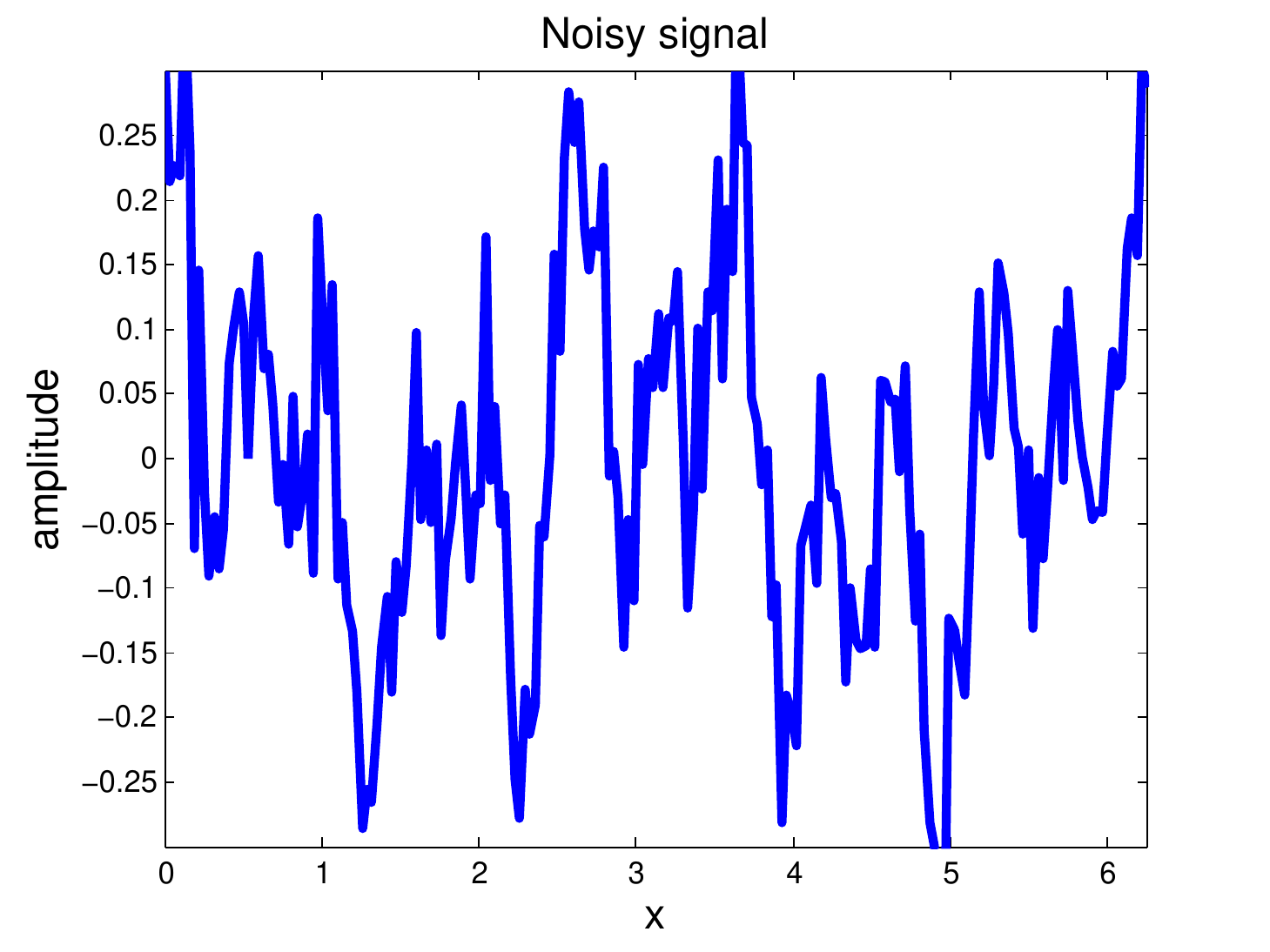}
\end{center}
\end{minipage}
\begin{minipage}[t]{\dctFigureWidth}
\begin{center}
\includegraphics[width=1\textwidth]{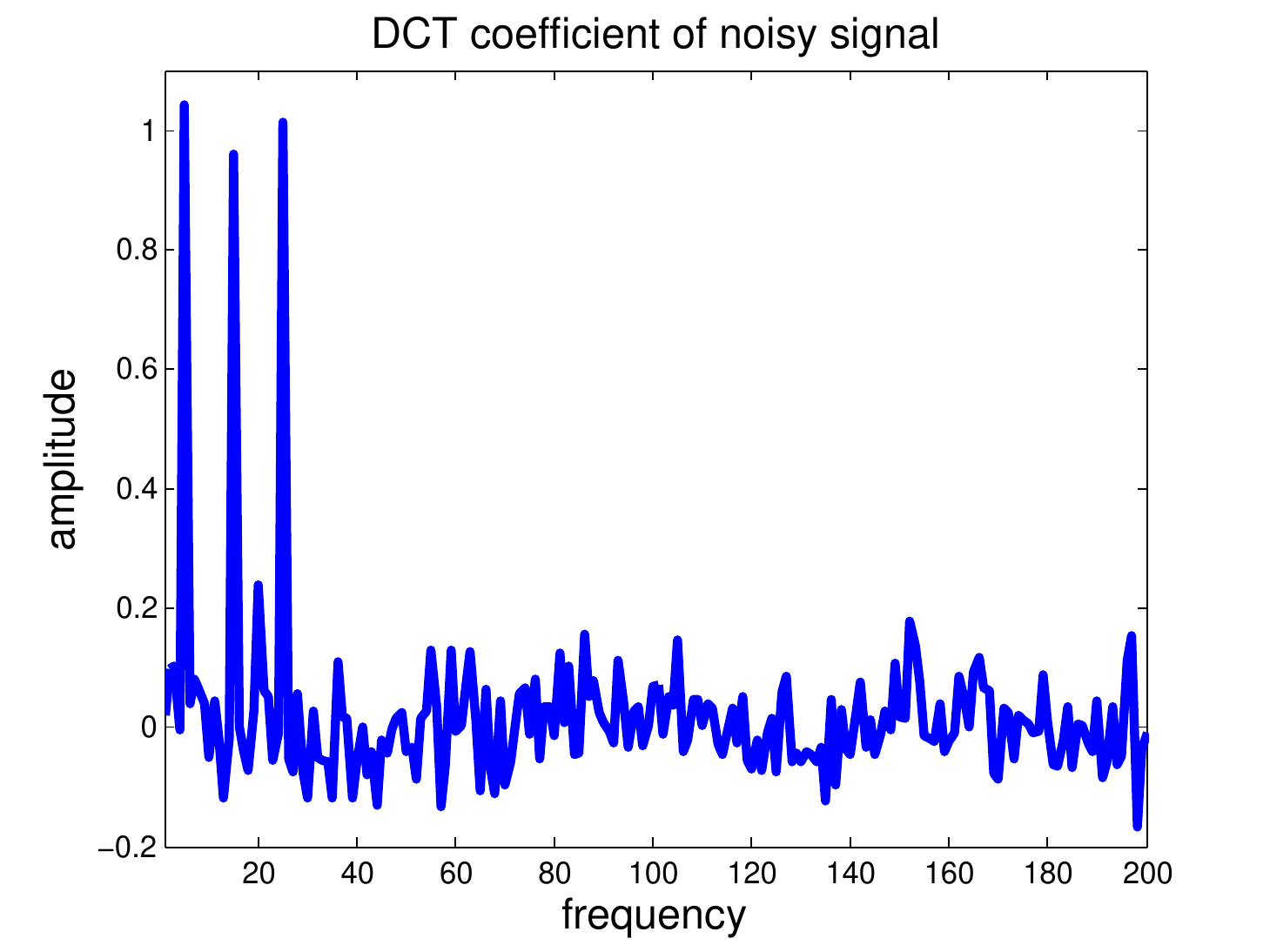}
\end{center}
\end{minipage}\\
\begin{minipage}[t]{\dctFigureWidth}
\vspace{-0.5cm}
\begin{center} a) Clean signal
\end{center}
\end{minipage}
\begin{minipage}[t]{\dctFigureWidth}
\vspace{-0.5cm}
\begin{center} b) Noisy signal
\end{center}
\end{minipage}
\begin{minipage}[t]{\dctFigureWidth}
\vspace{-0.5cm}
\begin{center} c) Noisy DCT coefficients
\end{center}
\end{minipage} \\

\begin{minipage}[t]{\dctFigureWidth}
\begin{center}
\includegraphics[width=1\textwidth]{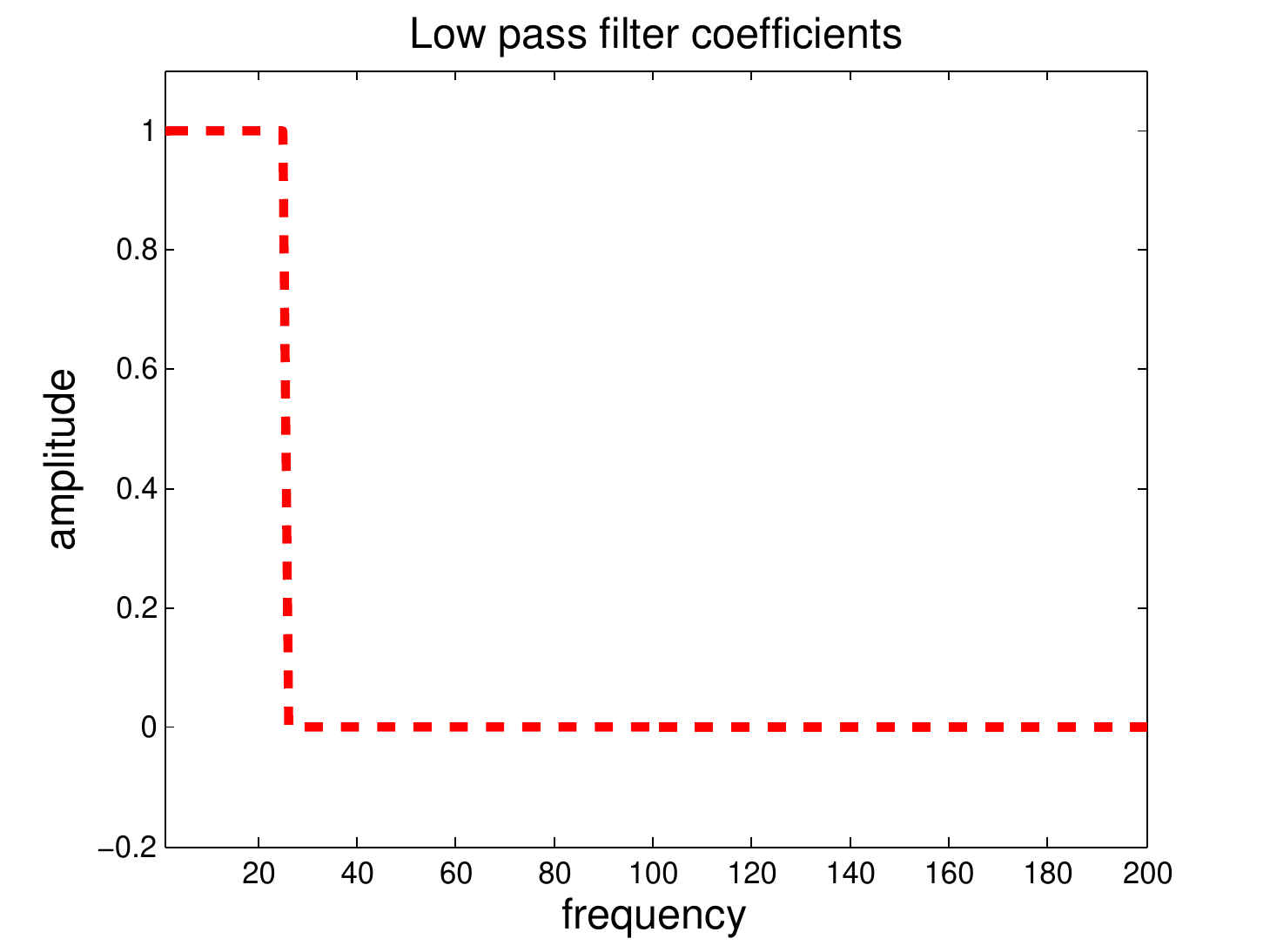}
\end{center}
\end{minipage}
\begin{minipage}[t]{\dctFigureWidth}
\begin{center}
\includegraphics[width=1\textwidth]{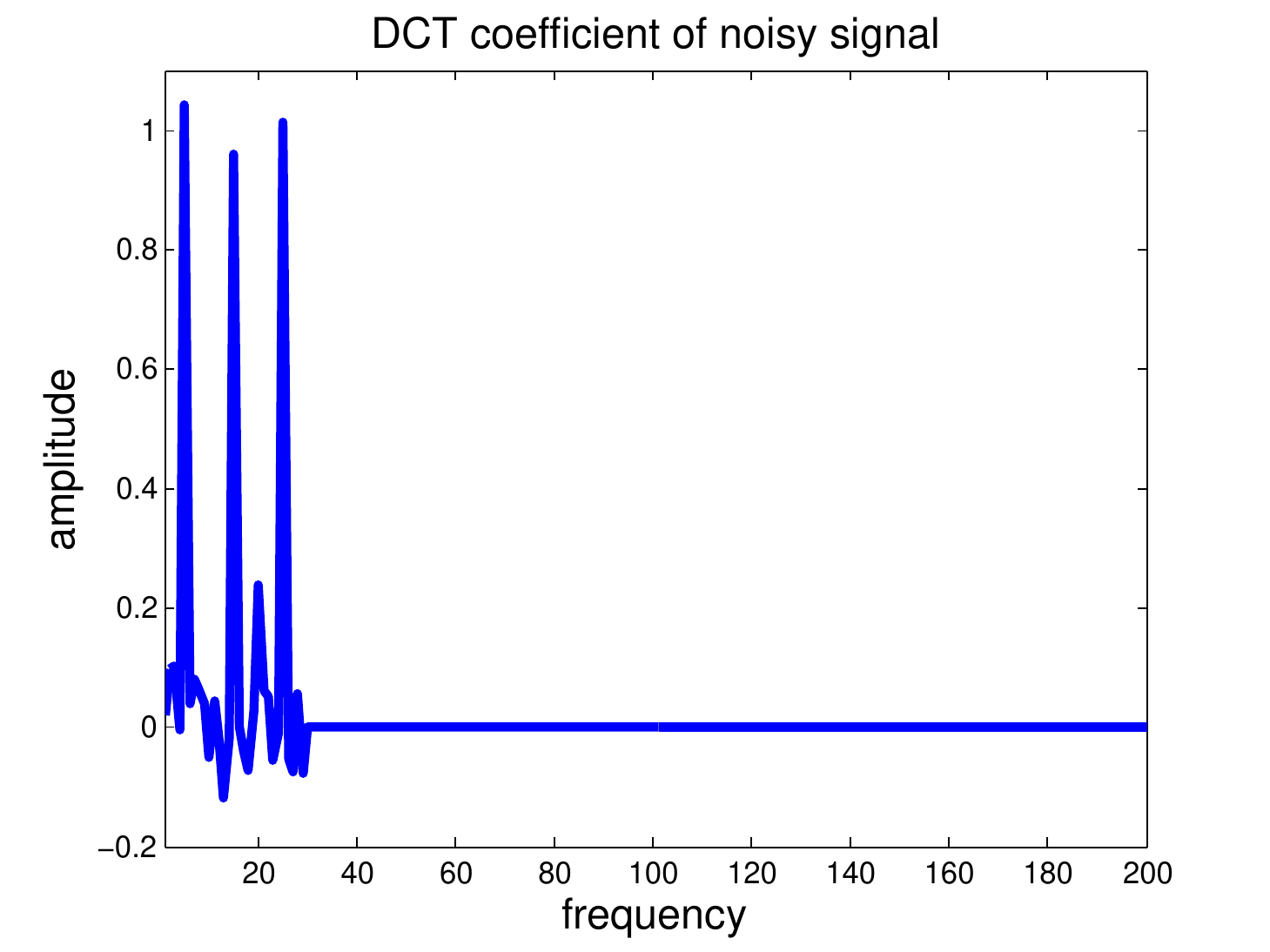}
\end{center}
\end{minipage}
\begin{minipage}[t]{\dctFigureWidth}
\begin{center}
\includegraphics[width=1\textwidth]{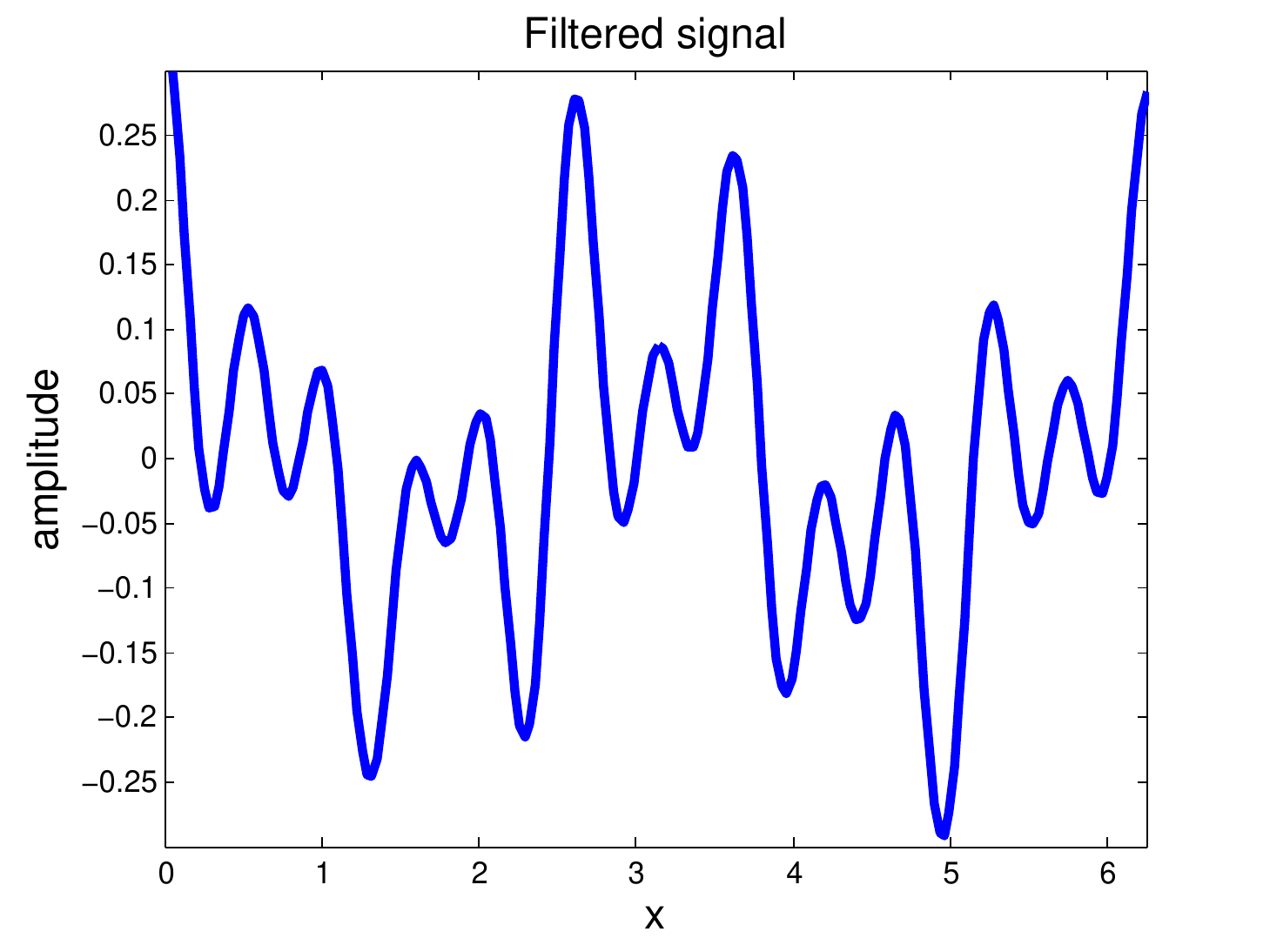}
\end{center}
\end{minipage}\\

\begin{minipage}[t]{\dctFigureWidth}
\vspace{-0.5cm}
\begin{center} d) Low pass filter
\end{center}
\end{minipage}
\begin{minipage}[t]{\dctFigureWidth}
\vspace{-0.5cm}
\begin{center} e) Filtered DCT coefficients
\end{center}
\end{minipage}
\begin{minipage}[t]{\dctFigureWidth}
\vspace{-0.5cm}
\begin{center} f) Denoised signal
\end{center}
\end{minipage}
\caption{Example of classical DCT-based filtering. The noisy input signal is represented as a superposition of cosines. The coefficients are filtered and the inverse transform is applied.}
\label{fig:classical}
\end{center}
\end{figure}
The entire filtering process illustrated in Figure \ref{fig:classical} is a linear operation on the input data. Denoting the orthonormal DCT transformation matrix by $V^T$, we can represent the filtered data $\hat{u}$ in Fig. f as
\begin{align}
\label{eq:linearFilter}
 \hat{u} = V D V^T f =: R(f),
\end{align}
where $D$ is a diagonal matrix containing the filter as its diagonal. In our above example, the diagonal entries of $D$ consist of $j$ ones followed by $n-j$ zeros with $j$ being the cutoff frequency of the ideal lowpass filter. Note that the diagonal elements of $D$ are exactly the eigenvalues of the linear operator $R$. Natural extensions of this theory to the continuous setting can be established for compact linear operators $R$ in Hilbert spaces.

In the last few decades, nonlinear filters based on \textit{variational methods} have become increasingly popular and have replaced linear filtering approaches in many applications. The most widely used variational filtering strategy is to compute
\begin{align}
\label{eq:variationalMethod}
u_{VM}(t) = \arg \min_u \frac{1}{2}\|u-f\|_2^2 + t J(u),
\end{align}
for a suitable proper, convex, lower semi-continuous regularization functional $J: {\cal X} \rightarrow \mathbb{R}^+ \cup \{\infty\}$ defined on Banach space ${\cal X}$ embedded into $L^2(\Omega)$. Examples include the total variation (TV) \cite{rof92}, or sparsity based regularization, e.g. \cite{Mairal08}.

A variant of the above approach is to fix the regularization parameter $t$ and iteratively use the previous denoising result as the input data $f$. Considering the continuous limit for small times $t$ leads to the \textit{scale space} or \textit{gradient flow}
\begin{align}
\label{eq:scaleSpace}
\partial_t u_{GF}(t) = -p_{GF}(t), \qquad p_{GF}(t) \in \partial J(u_{GF}(t)), \ u(0)=f.
\end{align}
We refer the reader to \cite{wk_book,ak_book02} for an overview of scale space techniques.

Due to the systematic bias or loss of contrast introduced by the techniques \eqref{eq:variationalMethod} and \eqref{eq:scaleSpace}, Osher et al. proposed the \textit{Bregman iteration} \cite{Bregman_obgxy} which -- in the continuous limit -- leads to the \textit{inverse scale space flow }

\begin{align}
\label{eq:inverseScaleSpace}
\partial_s q_{IS}(s) = f-v_{IS}(s), \qquad q_{IS}(s) \in \partial J(v_{IS}(s)), \ q_{IS}(0)=0,
\end{align}
and avoids the systematic error.

Until very recently, nonlinear variational methods such as \eqref{eq:variationalMethod}, \eqref{eq:scaleSpace}, and \eqref{eq:inverseScaleSpace} have been treated independent of the classical linear point of view of changing the representation of the input data, filtering the resulting representation and inverting the transform. In \cite{Gilboa_SSVM_2013_SpecTV,Gilboa_spectv_SIAM_2014} the use of \eqref{eq:scaleSpace} in the case of $J$ being the total variation (TV) to define a TV spectral representation of images was proposed, which allows to extend the idea of filtering approaches from the linear to the nonlinear case. We investigated further extensions of his framework by considering all three possible approaches \eqref{eq:variationalMethod}, \eqref{eq:scaleSpace}, and \eqref{eq:inverseScaleSpace} and general convex absolutely one-homogeneous regularization functionals in the conference proceedings \cite{spec_one_homog} and gave an extended overview over the related ideas in \cite{jmivPreprint}.

\subsection{What Is a Spectral Representation?}
\label{sec:WhatIsSpectral}


We would like to define more precisely what we refer to as a \emph{spectral representation}, in which linearity is not assumed, and hence it
can be generalized to the nonlinear (in our case convex) setting.
One notion being particularly important is the concept of nonlinear eigenfunctions induced by convex functionals.
Given a convex functional $J(u)$ and its subdifferential  $\partial J(u)$, we refer to $u$ with  $\|u\|_2=1$ as an \emph{eigenfunction} if it admits the following eigenvalue problem:
\begin{equation}
\label{eq:ef_problem}
\lambda u  \in \p J(u),
\end{equation}
where $\lambda \in \mathbb{R}$  is the corresponding eigenvalue. In some cases, depending on context, we will use also the term \emph{eigenvector}.

{
Let us detail the idea of nonlinear spectral representations of a function $f$, with respect to a convex functional $J$, in an appropriate Banach space $\mathcal{X}$. Our main motivation is related to the spectral representation of a positive self-adjoint operator $A$ on Hilbert spaces (cf. \cite{dunford1963linear}), i.e. the derivative of a convex quadratic functional, given by vector-valued measure $E_\lambda$ from $\mathbb{R}_+$ to the space of bounded linear operators on $\mathcal{X}$. The latter generalizes the eigenvalue decomposition of compact linear operators and will be illustrated with an example below. The vector-valued measure $E_\lambda$ allows to define functions of the operator $A$ via
\begin{equation}
	\varphi(A) = \int_{\mathbb{R}_+} \varphi(\lambda)~dE_\lambda,
\end{equation}
where $\varphi: \mathbb{R}_+ \rightarrow \mathbb{R}$ is a filter function. One can apply the filter $\varphi(A)$ to elements $f \in \mathcal{X}$ via
\begin{equation}
		\varphi(A)f = \int_{\mathbb{R}_+} \varphi(\lambda)~dE_\lambda \cdot f,
\end{equation}
here denoting by $\cdot$ the scalar product in $\mathcal{X}$. Note that if we are only interested in filtering $f$, we can directly reduce the expression to a vector-valued measure on $\mathcal{X}$
\begin{equation}
	\tilde{\Phi}_s = E_s \cdot f,
\end{equation}
such that
\begin{equation}
		\varphi(A)f = \int_{\mathbb{R}_+} \varphi(t)~d\tilde{\Phi}_s.
\end{equation}
Note that there is a trivial reconstruction of $f$ from $\tilde{\Phi}_s$ by choosing $\varphi\equiv 1$. Moreover, $\tilde{\Phi}_s$ is related to eigenvectors of the operator $A$ in the following way: If $f$ is an eigenvector for the eigenvalue $\lambda$, then $\tilde{\Phi}_s$ is the concentrated vector-valued measure $f ~\delta_\lambda$ ($\Vert \tilde{\Phi}_s\Vert$ is the scalar measure $\Vert f \Vert \delta_\lambda$), with $\delta_\lambda(s)$  corresponding to the Dirac delta distribution centered at $s=\lambda$.

In a nonlinear setting with $\mathcal{X}$ being a Banach space there is no appropriate way to generalize the spectral decomposition $E_s$, but we can try to generalize the data-dependent decomposition $\tilde{\Phi}_s$. Obviously we need to give up the linear dependence of $\tilde{\Phi}_s$ upon $f$, but we can still hope for a simple reconstruction and the relation to eigenvalues.

\begin{definition}\label{def:spectralDecompositionFrequency}
A map from $f \in \mathcal{X}$ to a vector-valued Radon measure $\tilde{\Phi}_s$ on $\mathcal{X}$ is called a {\em spectral (frequency) representation with respect to the convex functional $J$} if the following properties are satisfied:
\begin{itemize}
\item {\bf Eigenvectors as atoms:}
For $f$ satisfying $\norm{f}=1$ and $\lambda f \in \partial J(f)$ the spectral representation is given by $\tilde{\Phi}_s = f ~\delta_\lambda(s)$.

\item {\bf Reconstruction:} The input data $f$, for any $f \in \mathcal{X}$, can be reconstructed by
\begin{equation}
f = \int_0^\infty d\tilde{\Phi}_s .	
\end{equation}

\end{itemize}
\end{definition}

A spectral representation naturally carries a notion of scale $s$, features arising at small $s$ will be referred to as large scale details of $f$, while those at larger $s$ related to small scales.

As an example let $J(u)=\frac{1}{2}\int_\Omega |\nabla u(x)|^2 dx$, then $\p J(u) =\{-\Delta u \} $ and one can see that the decomposition into Laplacian eigenfunctions ($-\Delta u_j = \lambda_j u_j$)
$$ f = \sum_{j=1}^\infty c_j u_j $$
yields a spectral decomposition with
$$ \tilde{\Phi}_s = \sum_{j=1}^\infty c_j u_j \delta(s-\lambda_j). $$

To see the inverse relation between $s$ and the scale of features, it is instructive to consider the polar decomposition of the measure $\tilde{\Phi}_s$ into
\begin{equation} \label{eq:polarDecomposition}
	d\tilde{\Phi}_s = \tilde\psi(s) ~d\Vert \tilde{\Phi}_s \Vert
\end{equation}
 with $\Vert  \tilde\psi(s) \Vert =1$ for $\Vert \tilde{\Phi}_s \Vert$-almost every $s > 0$. To avoid ambiguity we shall use a representation of $\tilde\psi$ that vanishes outside the support of $\tilde{\Phi}_s$ and think of $\psi(s)$ as the representatives of scale $s$. The scalar measure $d\Vert \tilde{\Phi}_s \Vert$ can be considered as the {\em spectrum} of $f$ with respect to the functional $J$.

In our above example, the polar decomposition can simply be identified with the normalized eigenfunctions,
$$\tilde \psi(\lambda_j) = \text{sign}(c_j) u_j, \qquad \Vert \tilde{\Phi}_s \Vert = \sum_{j=1}^\infty \vert c_j \vert \delta(s-\lambda_j). $$
In the simple setting of $\Omega=[0,2 \pi]$ and zero Neumann boundary conditions, we obtain
$$u_j(x) = \frac{1}{\sqrt{\pi}} cos(j x)$$
and can see that the spectral representation is related to the frequency via $\lambda_j = j^2$. Hence, we will refer to the spectral representation of definition \ref{def:spectralDecompositionFrequency} as a \emph{generalized frequency decompositions}.

In further analogy to Fourier methods, we can consider a change of variables from $s$ to $t=\frac{1}s$, which leads to large $t$ corresponding to large scales and small $t$ corresponding to small scales. More precisely, we define
$$ d{\Phi_t} = \frac{1}{s^2} \tilde d\Phi_{1/s} $$
which means that
$$ \int_0^\infty \varphi(s) ~ d \tilde{\Phi}_s = \int_0^\infty \varphi(1/t) ~ d\Phi_t $$
holds for all filter functions $\varphi$. The new vector-valued Radon measure $\Phi$ is an equally meaningful candidate for a spectral decomposition, to which we refer as the \emph{generalized wavelength decomposition}. More formally and in full analogy to definition \ref{def:spectralDecompositionFrequency}, we define:
\begin{definition}\label{def:spectralDecompositionWavelength}
A map from $f \in \mathcal{X}$ to a vector-valued Radon measure $\Phi_t$ on $\mathcal{X}$ is called a {\em spectral (wavelength) representation with respect to the convex functional $J$} if the following properties are satisfied:
\begin{itemize}
\item {\bf Eigenvectors as atoms:}
For $f$ satisfying $\norm{f}=1$ and $\lambda f \in \partial J(f)$ the spectral representation is given by $\Phi_t = f ~\delta_{\frac{1}{\lambda}}(t)$.
\item {\bf Reconstruction:} The input data $f$, for any $f \in \mathcal{X}$, can be reconstructed by
\begin{equation}
f = \int_0^\infty d \Phi_t .	
\end{equation}
\end{itemize}
\end{definition}
Similar to spectral frequency decompositions, the polar decomposition of $\Phi$ in a wavelength decomposition gives rise to normalized functions $\psi(t)$.}

The example of Laplacian eigenfunctions (more general eigenfunction expansions in Hilbert space) has a further orthogonality structure that we did not consider so far. Note that for $j \neq k$ we have $\langle u_j, u_k \rangle_{L^2} = 0$ and we can compute the coefficients as $c_k = \langle f, u_k \rangle$.  This also implies the Parsevel identity
\begin{equation}
	\Vert f \Vert^2 = \sum_{j=1}^\infty |c_k|^2 .
\end{equation}
We generalize these properties to define an orthogonal spectral representation:

\begin{definition}[Orthogonal spectral representation]\
A spectral representation on $\mathcal{X}$ is called an {\em orthogonal spectral representation with respect to the convex functional $J$} if the following properties are satisfied:
\begin{itemize}
\item {\bf Orthogonality of scales:} $\tilde \psi(s_1) \cdot \tilde\psi(s_2) = 0$  if $s_1 \neq s_2$ for $\Vert \tilde\Phi_s \Vert$-almost all $s_1$ and $s_2$.

\item {\bf Generalized Parseval identity:} $\tilde \Phi_s \cdot f$ is a nonnegative measure on $\mathbb{R}^+$  (i.e. $\tilde\psi(s) \cdot f \geq 0$) and
\begin{equation}
\Vert f \Vert^2  = \int_0^\infty d(\tilde \Phi_s \cdot f) = \int_0^\infty (\tilde \psi(s) \cdot f) d\Vert \tilde\Phi_s\Vert.
\end{equation}

\end{itemize}
\end{definition}

Note that the above definition can equivalently be made for orthogonal spectral wavelength representations by removing the tilde and replacing $s$ by $t$.

From the above motivations it becomes natural to define nonlinear filtering of the data $f$ with respect to $J$ via certain integrals with respect to the measure $\tilde \Phi_s$. If we are merely interested in the latter with sufficiently regular filters we can extend from $\tilde \Phi_s$ to vector-valued distributions.

\begin{definition}[Weak spectral representation]\label{def:weakSpectralDecomposition}
A map from $f \in \mathcal{X}$ to a vector-valued distribution $\tilde \phi(s)$ on $\mathcal{X}$ is called a {\em weak spectral representation with respect to the convex functional $J$} if the following properties are satisfied:
\begin{itemize}
\item {\bf Eigenvectors as atoms:} For $f$ satisfying $\norm{f}=1$ and $\lambda f \in \partial J(f)$ the weak spectral representation is given by $\tilde \phi(s) = f ~\delta_\lambda(s)$.

\item {\bf Reconstruction:} The input data $f$ can be reconstructed by
\begin{equation}
f = \int_0^\infty \tilde\phi(s)~ds .	
\end{equation}

\end{itemize}
\end{definition}
Again, exactly the same definition of a weak spectral representation can also be made for spectral wavelength decomposition with an inverse relation between $\lambda$ and $t$ for eigenvectors as atoms.

We mention that throughout the paper we will use a rather formal notation for the distribution $\tilde \phi$ and its wavelength counterpart $\phi$ as in the reconstruction formula above. All integrals we write indeed have to be understood as duality products with sufficiently smooth test functions. We will verify  the well-definedness for the specific spectral representations we investigate below, it will turn out that in all cases it suffices to use test functions in $W^{1,1}_{loc}(\mathbb{R}_+;\mathbb{R}^n)$.

\subsection{Paper goal and main results}
\label{sec:Goal}
In this paper, we extend our previous works \cite{spec_one_homog, jmivPreprint} by several novel theoretical results:

\vspace{1em}
\begin{itemize}
\setlength\itemsep{1em}
\item We prove that all three methods \eqref{eq:variationalMethod}, \eqref{eq:scaleSpace}, and \eqref{eq:inverseScaleSpace} yield a well-defined weak spectral representation for arbitrary convex and absolutely homogeneous functionals $J$ on $\mathbb{R}^n$.

\item We prove that polyhedral regularizations $J$ in finite dimensions lead to spectral decompositions that consist of finitely many delta-peaks for all three approaches, \eqref{eq:variationalMethod}, \eqref{eq:scaleSpace}, and \eqref{eq:inverseScaleSpace}.
\item We analyze the relation between the three possible spectral representations. In particular, we show that
\begin{itemize}
\setlength\itemsep{0.5em}
\item The variational method \eqref{eq:variationalMethod} and the gradient flow \eqref{eq:scaleSpace} coincide under the assumption of a polyhedral regularization $J$ that meets a regularity assumption on its subdifferentials we call (MINSUB).
\item All three methods \eqref{eq:variationalMethod}, \eqref{eq:scaleSpace}, and \eqref{eq:inverseScaleSpace} yield exactly the same orthogonal spectral representation for $J(u) = \|Ku\|_1$ and $KK^* \in \mathbb{R}^{m \times m}$ being diagonally dominant.
\end{itemize}
\item We prove the orthogonality of the spectral decomposition arising from variational methods and the gradient flows for polyhedral regularizations meeting (MINSUB).
\item We prove that the spectral decomposition arising from $J(u) = \|Ku\|_1$ with $KK^* \in \mathbb{R}^{m \times m}$ being diagonally dominant represents the input data as a linear combination of generalized eigenfunctions meeting \eqref{eq:ef_problem}.
\end{itemize}
\vspace{1em}

The main goal of this paper is analyzing and understanding the theory of nonlinear multiscale decompositions. Nevertheless, there is a vast number of possible applications in imaging. As an example, consider Figure \ref{fig:filteringTeaser}: A nonlinear bandstop filter with respect to a spectral TV decomposition was applied to the image in (a) to remove wrinkles and obtain the image shown in (b). Due to the great localization of the TV frequency components, a spatial correspondence (registration) between (a) and (c) allows to synthesize the image in (d) by inserting the filtered frequencies in a straightforward manner.
Preliminary applications for denoising through learned filters \cite{learnedSpectralFilters} and for spatially-varying texture separation \cite{HoreshGilboa_submitted,Horesh_thesis} have been recently proposed by the authors and colleagues.
The focus of this paper is mostly theoretical.

\begin{figure}[h]
\begin{center}
\begin{minipage}[t]{0.22\textwidth}
\begin{center}
\includegraphics[width=1\textwidth]{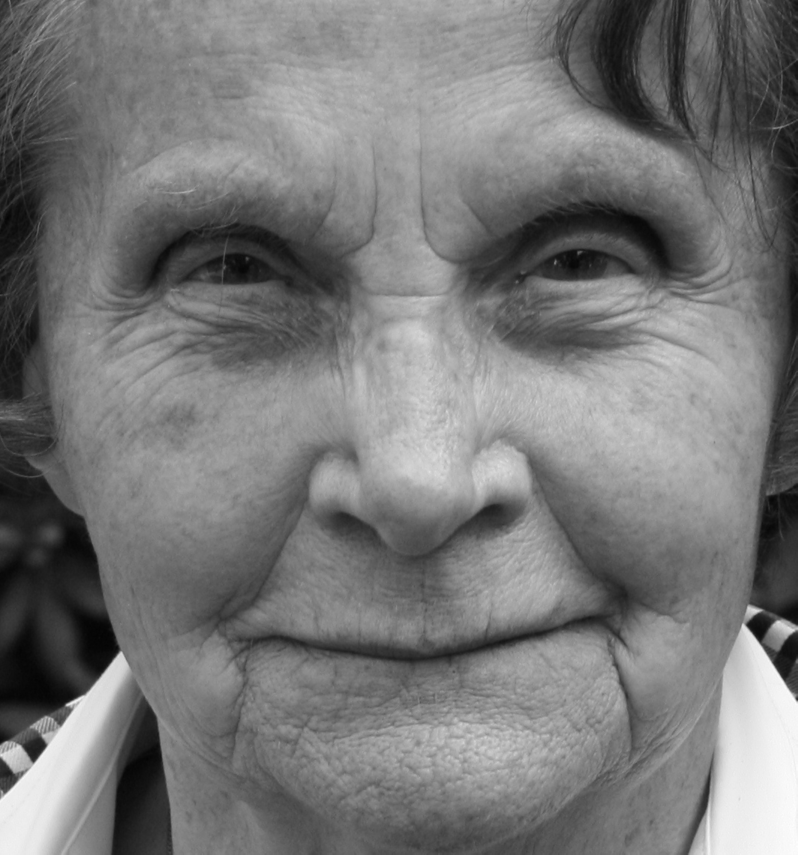}
\end{center}
\end{minipage}
\begin{minipage}[t]{0.22\textwidth}
\begin{center}
\includegraphics[width=1\textwidth]{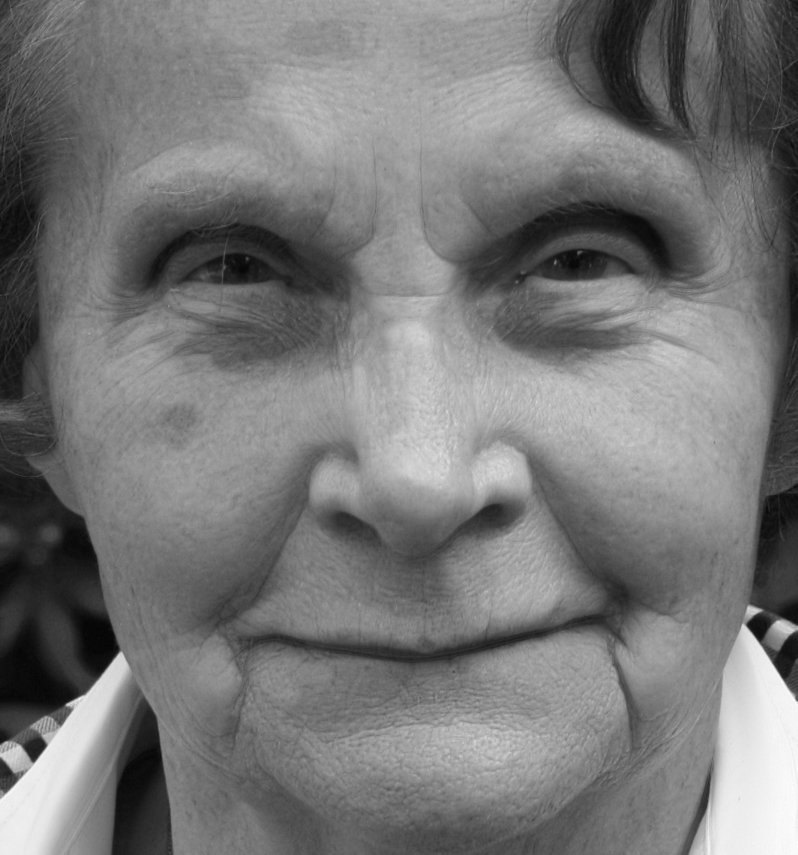}
\end{center}
\end{minipage}
\begin{minipage}[t]{0.26\textwidth}
\begin{center}
\includegraphics[width=1\textwidth]{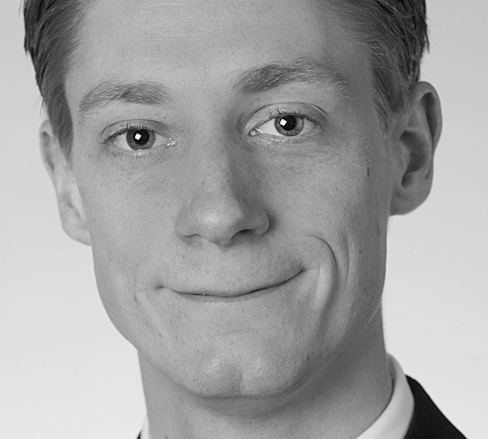}
\end{center}
\end{minipage}
\begin{minipage}[t]{0.26\textwidth}
\begin{center}
\includegraphics[width=1\textwidth]{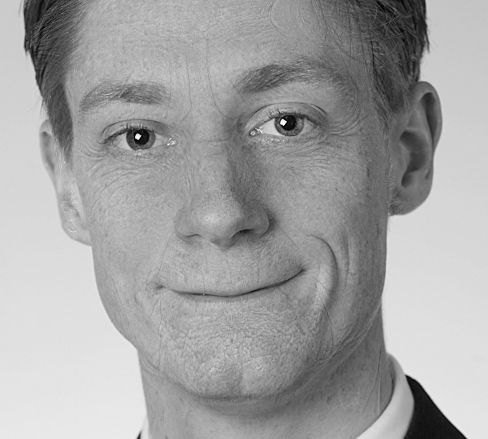}
\end{center}
\end{minipage}\\

\begin{minipage}[t]{0.22\textwidth}
\begin{center} (a) Original image
\end{center}
\end{minipage}
\begin{minipage}[t]{0.22\textwidth}
\begin{center} (b) TV bandstop filtered
\end{center}
\end{minipage}
\begin{minipage}[t]{0.26\textwidth}
\begin{center} (c) Image with known correspondence to (a)
\end{center}
\end{minipage}
\begin{minipage}[t]{0.26\textwidth}
\begin{center} Inserting high frequency components of (a) in (c)
\end{center}
\end{minipage}
\caption{Example application of nonlinear spectral filtering for wrinkle removal and image style transformation: The proposed spectral decomposition framework not only allows to selectively remove certain frequencies so as to reduce the effects of aging (left two images) - it also allows to inject respective frequencies into other images (with known registration) to simulate the effect of aging again.}
\label{fig:filteringTeaser}
\end{center}
\end{figure}

The rest of this paper is organized as follows. First we recall the idea of nonlinear eigenfunctions of convex regularizers in section \ref{sec:eigenfunctions}. In section \ref{sec:Spectral} we recall the different definitions of spectral representations based on \eqref{eq:variationalMethod}, \eqref{eq:scaleSpace}, and \eqref{eq:inverseScaleSpace} and discuss possible definitions of the spectrum of a decomposition. We analyze the similarities between the three approaches in section \ref{sec:analysis} and establish sufficient conditions for their equivalence. We prove that our spectral approaches yield a decomposition of $f$ into generalized eigenvectors for a particular class of regularization functionals. Finally, in section \ref{sec:NumericalResults} we illustrate our theoretical findings with numerical experiments before we draw conclusions and point out directions of future research in section \ref{sec:Conclusions}.

\section{Nonlinear Eigenfunctions of Convex Regularizations}\label{sec:eigenfunctions}
As explained in the previous section, the motivation and interpretation of classical linear filtering strategies is closely linked to the spectral decomposition of positive semidefinite linear operators (derivatives of quadratic functionals). To closely link the proposed nonlinear spectral decomposition framework to the linear one, let us summarize earlier studies concerning the use of nonlinear eigenfunctions in the context of variational methods.
The description and definition of the eigenvalue problem is given in section \ref{sec:WhatIsSpectral}.

Previous studies have mainly examined nonlinear eigenfunctions of regularizers in the context of total-variation (where it was referred to as \emph{calibrable sets}).
Meyer \cite{Meyer[1]} gave an explicit solution to the ROF model \cite{rof92} in the case of a disk (an eigenfunction of TV), explained the loss of contrast which motivated
the use of the $TV-G$ model. Extensive analysis was performed for the TV-flow \cite{tvFlowAndrea2001,andreu2002some,bellettini2002total,Steidl,burger2007inverse_tvflow,discrete_tvflow_2012,tvf_giga2010} where explicit solutions were given in several cases and for various spatial settings. In \cite{iss} an explicit solution of a disk for the inverse-scale-space flow is presented, showing its instantaneous appearance at a precise time point related to its radius and height.

In \cite{bellettini2002total} a precise geometric characterization of TV eigenfunctions is given in the 2D case, let us briefly recall it.
Let $\chi_C \in \mathbb{R}^2$ be a characteristics functions, then it admits \eqref{eq:ef_problem}, with $J$ the TV functional, if
\begin{equation}
\label{eq:eigen_xi}
\textrm{ess} \sup_{q \in \p C}\kappa(q) \le \frac{P(C)}{|C|},
\end{equation}
where $C$ is convex, $\p C \in C^{1,1}$, $P(C)$ is the perimeter of $C$, $|C|$ is the area of $C$
and $\kappa$ is the curvature. In this case the eigenvalue is $\lambda = \frac{P(C)}{|C|}$.

In \cite{Muller_thesis,Benning_highOrderTV_2013} eigenfunctions related to the total-generalized-variation (TGV) \cite{bredies_tgv_2010} and to infimal convolution total variation (ICTV) \cite{Chambolle[2]} are analyzed and their different reconstruction properties on particular eigenfunctions of the TGV are demonstrated theoretically as well as numerically.
Examples of certain eigenfunctions for different extensions of the TV to color images are given in \cite{collaborativeTV}

The work in \cite{Benning_Burger_2013} considers more general variational reconstruction problems involving a linear operator in the data fidelity term, i.e.,
$$\min_u \frac{1}{2}\|Au-f\|_2^2 + t \ J(u), $$
and generalizes equation \eqref{eq:ef_problem} to
$$ \lambda A^*A u \in \partial J(u), \qquad \|Au\|_2=1,$$
in which case $u$ is called a \textit{singular vector}. Particular emphasis is put on the \textit{ground states}
$$ u^0 = \argmin\limits_{\substack{u \in \text{kern}(J)^\perp,\\ \|Au\|_2=1}} J(u) $$
for semi-norms $J$, which were proven to be singular vectors with the smallest possible singular value. Although the existence of a ground state (and hence the existence of singular vectors) is guaranteed for all non-trivial $J$, it was shown in \cite{Benning_Burger_2013}  that the Rayleigh principle for higher singular values fails. In \cite{jmivPreprint} we showed that the Rayleigh principle for higher eigenvalues even fails in the case where $A$ is the identity. As a consequence, determining or even showing the existence of a basis of singular vectors remains an open problem for general semi-norms $J$.

In \cite{Steidl} Steidl et al. have shown the close relations, and equivalence in a 1D discrete
setting, of the Haar wavelets to both TV regularization and TV flow.
This was later developed for a 2D setting in \cite{Steidl_2D}.
In the field of morphological signal processing, nonlinear transforms were introduced
in \cite{Dorst_Boomgaard_Slope_trans_1994,Kothe_Morph_SS_ECCV1996}.

\section{Spectral Representation}
\label{sec:Spectral}
Throughout this section we will assume the regularization $J$ to be a proper, convex, lower semi-continuous function on $\R^n$.  In this finite-dimensional setting the well-posedness of the variational problems \eqref{eq:variationalMethod}, and the differential inclusions \eqref{eq:scaleSpace}, and \eqref{eq:inverseScaleSpace} follows immediately from the above assumptions. Furthermore, we will assume $J$ to be absolutely one-homogeneous, i.e.
\begin{equation}
	J(s u ) = |s| J(u) \qquad \forall~s \in \mathbb{R}, \forall~u \in {\cal X}.
\end{equation}

In the next subsection we summarize some important properties of absolutely one-homogeneous functionals. In the following three subsections we will discuss the spectral representation for the three approaches \eqref{eq:variationalMethod}, \eqref{eq:scaleSpace}, and \eqref{eq:inverseScaleSpace} separately, before analyzing their relation in the section thereafter.

\subsection{Properties of Absolutely One-homogeneous Regularizations}

Under our above assumptions we can interpret $J$ as a seminorm, respectively a norm on an appropriate subspace.
\begin{lemma}\
A functional $J$ as above is a seminorm and its nullspace
$$ {\cal N}(J) = \{ u \in \mathbb{R}^n ~|~J(u) = 0 \} $$ is a linear subspace.
In particular there exist constants $0 < c_0 \leq C_0$ such that
\begin{equation} \label{normequivalence}
	c_0 \Vert u \Vert \leq J(u) \leq C_0 \Vert u \Vert, \qquad \forall~u \in {\cal N}(J)^\perp.
\end{equation}
\end{lemma}
\begin{proof}
First of all we observe that $J$ is nonnegative and absolutely one-homogeneous due to our above definitions, so that it suffices to verify the triangle inequality. From the convexity and absolute one-homogeneity we have for all $u, v \in {\cal X}$
$$ J(u+v) = 2 \: J\left(\frac{1}2 u + \frac{1}2 v\right) \leq 2\left(\frac{1}2J(u)+ \frac{1}2 J(v)\right) = J(u) + J(v). $$
The fact that the nullspace is a linear subspace is a direct consequence, and the estimate \eqref{normequivalence} follows from the norm equivalence in finite dimensional space.
\end{proof}

\begin{lemma} \label{lem:nullLemma}
Let $J$ be as above, then for each $u \in \mathbb{R}^n$ and $v \in {\cal N}(J)$ the identity
\begin{equation}
	J(u+v) = J(u)
\end{equation}
holds.
\end{lemma}
\begin{proof}
Using the triangle inequality we find
\begin{align*}
J(u+v) &\leq  J(u) + J(v) = J(u), \\
J(u) &= J(u+v-v) \leq J(u+v) + J(-v) = J(u+v),
\end{align*}
which yields the assertion.
\end{proof}

We continue with some properties of subgradients:

\begin{lemma} \label{subgradientlemma1}
Let  $u \in \mathbb{R}^n$, then $p \in \partial J(u)$ if and only if
\begin{equation}
	J^*(p) = 0 \quad \text{and} \quad \langle p, u \rangle = J(u).
\end{equation}
\end{lemma}

A common reformulation of Lemma \ref{subgradientlemma1} is the characterization of the subdifferential of an absolutely one-homogeneous $J$ as
\begin{align}
\label{eq:subdifferential}
\partial J(u) = \{ p \in {\cal X}^* ~|~ J(v) \geq \langle p, v \rangle \ \forall v \in {\cal X}, ~ J(u) = \langle p, u \rangle \}.
\end{align}
\begin{remark}
\label{rem:subdiffRemark}
A simple consequence of the characterization of the subdifferential of absolutely one-homogeneous functionals is that any $p \in \partial J(0)$ with $p \notin \partial J(u)$ meets $J(u) - \langle p, u \rangle >0$.
\end{remark}

Additionally, we can state a property of subgradients relative to the nullspace of $J$ with a straightforward proof:
\begin{lemma} \label{subgradientlemma3}
Let $p \in \partial J(0)$ and $J(u)  = 0$, then $\langle p, u \rangle = 0$. Consequently
$\partial J(u) \subset \partial J(0) \subset {\cal N}(J)^\bot$ for all $u \in \mathbb{R}^n$.
\end{lemma}

The nullspace of $J$ and its orthogonal complement will be of further importance in the sequel of the paper. In the following we will denote the projection operator onto ${\cal N}(J)$ by $P_0$ and to ${\cal N}(J)^\bot$ by $Q_0 = Id - P_0$. Note that as a consequence of Lemma \ref{lem:nullLemma} we have $J(u)=J(Q_0 u)$ for all $u \in \mathbb{R}^n$.

\begin{lemma} \label{subgradientlemma2}
For absolutely one-homogeneous $J$ the following identity holds
\begin{equation}
	\bigcup_{u \in \mathbb{R}^n} \partial J(u) = \partial J(0) = \{p \in \mathbb{R}^n~|~J^*(p) = 0 \}.
\end{equation}
Moreover $\partial J(0)$ has nonempty relative interior in ${\cal N}(J)^\bot$ and for
any $p$ in the relative interior of $\partial J(0)$ we have $p \in \partial J(u)$ if and only if $J(u)=0$.
\end{lemma}
\begin{proof}
We have $p \in \partial J(0)$ if and only if
$ \langle p, u \rangle \leq J(u) $
for all $u$. Since equality holds for $u=0$ this is obviously equivalent to $J^*(p) = 0$. Since we know from Lemma \ref{subgradientlemma1} that $\partial J(u)$ is contained in $\{p \in \mathbb{R}^n~|~J^*(p) = 0 \}$ and the union also includes $u=0$ we obtain the first identity.
Let $p \in {\cal N}(J)^\bot$ with $\Vert p \Vert < c_0$ sufficiently small. Then we know by the Cauchy-Schwarz inequality and \eqref{normequivalence} that
$$ \langle p, u \rangle = \langle p, Q_0 u \rangle \leq \Vert p \Vert ~\Vert Q_0 u \Vert \leq \frac{\Vert p \Vert}{c_0} J(Q_0 u) < J(Q_0 u) = J(u)$$
for all $u$ with $Q_0 u\neq 0$.
Finally, let $p$ be in the relative interior of $\partial J(0)$ and $p \in \partial J(u)$. Since there exists a constant $c>1$ such that $cp \in \partial J(0)$, we find $c \langle p,u \rangle = c J(u) \leq J(u)$ and consequently $J(u) =0$.
\end{proof}
\begin{conclusion}\label{conclu:boundedSubdiff}
Using \eqref{normequivalence} as well as Lemma \ref{subgradientlemma3}, we can conclude that for any $p \in \partial J(0)$ we have
$$ \|p\|^2 \leq J(p) \leq C\|p\|, $$
such that $\|p\|\leq C$ holds for all possible subgradients $p$.
\end{conclusion}

As usual, e.g. in the Fenchel-Young Inequality, one can directly relate the characterization of the subdifferential to the convex conjugate of the functional $J$.
\begin{lemma}
\label{lem:Jstar}
The convex conjugate of an absolutely one-homogeneous functional $J$ is the characteristic function of the convex set $\partial J(0)$.
\end{lemma}
\begin{proof}
Due to \eqref{eq:subdifferential} we have
$$\partial J(0) = \{ p \in {\cal X}^* ~|~ J(v) - \langle p, v \rangle \geq 0  \ \forall v \in {\cal X}\}.$$
The definition
$$J^*(p) = \sup_u (\langle p, u \rangle - J(u)) $$
tells us that if $p \in \partial J(0)$ the above supremum is less or equal to zero and the choice $u=0$ shows that $J^*(p)=0$. If $p \notin \partial J(0)$ then there exist a $u$ such that $\langle p, u \rangle - J(u)>0$ and the fact that
$$ \langle p, \alpha u \rangle - J( \alpha u) = \alpha (\langle p, u \rangle - J(u)) $$
holds for $\alpha \geq 0$ yields $J^*(p)=\infty$.
\end{proof}

\subsection{Variational Representation}
\label{sec:SpectralVariational}
In this section we would like to describe how to define a spectral representation based on the variational method \eqref{eq:variationalMethod}. As discussed in Section \ref{sec:WhatIsSpectral} we would like to establish the following analogy to the linear spectral analysis: Eigenfunctions meeting \eqref{eq:ef_problem} should be fundamental atoms of the spectral decomposition. It is easy to verify (cf. \cite{Benning_Burger_2013}) that for $f$ being an eigenfunction the solution to \eqref{eq:variationalMethod} is given by
\begin{align*}
u_{VM}(t) = \left\{
\begin{array}{cc}
(1 -  t \lambda ) f & \text{ for } t\leq \frac{1}{\lambda}, \\
0 & \text{ else. }
\end{array}
\right.
\end{align*}
Since $u_{VM}(t)$ behaves piecewise linear in time, and our goal is to obtain a single peak for $f$ being an eigenfunction, it is natural to consider the second derivative of $u_{VM}(t)$ (in a distributional sense). The latter will yield the (desired) delta distribution $\partial_{tt} u_{VM}(t) = \lambda \delta_{\frac{1}{\lambda}}(t)f$. Consider the desired property of reconstructing the data by integrating the spectral decomposition as discussed in Section \ref{sec:WhatIsSpectral}. We find that in the case of $f$ being an eigenfunction
$$ \int_0^\infty \partial_{tt} u_{VM}(t) ~dt = \lambda f.$$
Therefore, a normalization by multiplying $\partial_{tt} u_{VM}(t)$ with $t$ is required to fully meet the integration criterion. Motivated by the behavior of the variational method on eigenfunctions, we make the following definition:
\begin{definition}[Spectral Representation based on \eqref{eq:variationalMethod}]\
 Let $u_{VM}(t)$ be the solution to \eqref{eq:variationalMethod}. We define
\begin{align}
\label{eq:phiVm}
\phi_{VM}(t) = t \partial_{tt} u_{VM}(t)
\end{align}
to be the \textit{wavelength decomposition} of $f$.
\end{definition}

In the following we will show that the above $\phi_{VM}$ is a weak spectral wavelength decomposition in the sense of definition \ref{def:weakSpectralDecomposition} for any convex absolutely one-homogeneous regularization functional $J$.

As discussed in section \ref{sec:WhatIsSpectral} the term \textit{wavelength} is due to the fact that in the case of $f$ being an eigenvector, we can see that the peak in $\phi_{VM}(t)$ appears at a later time, the smaller $\lambda$ is. The eigenvalue $\lambda$ reflects our understanding of generalized frequencies which is nicely underlined by the fact that $\lambda = J(f)$ holds due to the absolute one-homogenity of $J$. Therefore, we expect contributions of small frequencies at large $t$ and contributions of high frequencies at small $t$, which is the relation typically called wavelength representation in the linear setting.

While we have seen that the definition of \eqref{eq:phiVm} makes sense in the case of $f$ being an eigenfunction, it remains to show this for general $f$. As a first step we state the following property of variational methods with absolutely one-homogeneous regularizations:
\begin{proposition}[Finite time extinction]
\label{prop:finiteTimeExtinction}
Let $J$ be an absolutely one-homogeneous functional, and $f$ be arbitrary. Then there exists a time $T<\infty$ such that $u_{VM}$ determined by \eqref{eq:variationalMethod} meets
$$u_{VM}(T) = P_0(f).$$
\end{proposition}
\begin{proof}
Considering the optimality conditions for \eqref{eq:variationalMethod}, the above statement is the same as
$$\frac{Q_0 f}T  =  \frac{f - P_0(f)}{T} \in \partial J(P_0(f))= \partial J(0).$$
Since $\partial J(0)$ has nonempty relative interior in ${\cal N}(J)$ this is guaranteed for $T$ sufficiently large.
\end{proof}
Note that the above proof yields the extinction as the minimal value $T$ such that $\frac{Q_0 f}T \in \partial J(0)$, which can also be expressed as the minimal $T$ such that $J^*(\frac{Q_0 f}T)=0$, i.e. $\frac{Q_0 f}T$ is in the dual unit ball. This is the generalization of the well-known result by Meyer \cite{Meyer[1]} for total variation denoising.

A second useful property is the Lipschitz continuity of $u_{VM}$, which allows to narrow the class of distributions for $\phi$:
\begin{proposition} \label{prop:VMlipschitz}
The function $u_{VM}: \R^+ \rightarrow \R^n$ is Lipschitz continuous. Moreover, the spectral representation satisfies $\phi_{VM} \in (W_{loc}^{1,1}(\mathbb{R}^+,\mathbb{R}^n))^*$.
\end{proposition}
\begin{proof}
Consider $u_{VM}(t)$ and $u_{VM}(t+\Delta t)$. Subtracting the optimality conditions yields
$$0 = u_{VM}(t) - u_{VM}(t+ \Delta t) + t (p_{VM}(t) - p_{VM}(t + \Delta t)) - \Delta t p_{VM}(t + \Delta t).$$
Taking the inner product with $u_{VM}(t) - u_{VM}(t+ \Delta t)$ yields
\begin{align*}
0 =& \|u_{VM}(t) - u_{VM}(t+ \Delta t)\|^2  - \Delta t  \langle p_{VM}(t + \Delta t), u_{VM}(t) - u_{VM}(t+ \Delta t)\rangle \\
&  + t \underbrace{\langle p_{VM}(t) - p_{VM}(t + \Delta t)),u_{VM}(t) - u_{VM}(t+ \Delta t)\rangle}_{\geq 0}\\
\geq& \|u_{VM}(t) - u_{VM}(t+ \Delta t)\|^2  - \Delta t  \langle p_{VM}(t + \Delta t), u_{VM}(t) - u_{VM}(t+ \Delta t)\rangle \\
\geq &\|u_{VM}(t) - u_{VM}(t+ \Delta t)\|^2  - \Delta t  \| p_{VM}(t + \Delta t)\| \|u_{VM}(t) - u_{VM}(t+ \Delta t)\|.
\end{align*}
Using conclusion \ref{conclu:boundedSubdiff} (and Lemma \ref{subgradientlemma2}), we find
$$\|u_{VM}(t) - u_{VM}(t+ \Delta t)\| \leq \Delta t \; C.$$

Through integration by parts we obtain for regular test functions $v$:
\begin{align}
\label{eq:formalFiltering0}
 \int_0^\infty v(t) \; \cdot \phi_{VM}(t) ~dt  =&  \int_0^\infty t v(t) \cdot \partial_{tt}u_{VM}(t) ~dt,  \nonumber\\
 =&    - \int_0^\infty \partial_t(v(t) \;t) \partial_{t}u_{VM}(t) ~dt,   \nonumber\\
  =&  - \int_0^\infty (t \partial_t v(t) + v(t)) \; \partial_{t}u_{VM}(t) ~dt.
\end{align}
If $u_{VM}(t)$ is Lipschitz-continuous, then $\partial_{t}u_{VM}(t)$ is an $L^\infty$ function. Due to the finite time extinction, the above integrals can be restricted to the interval $(0,T_{ext})$ and the last integral is  well-defined for any $v$ such that $v\in W^{1,1}_{loc}(\mathbb{R}_+,\mathbb{R}^n)$. A standard density argument yields that \eqref{eq:formalFiltering} finally allows to use all such test functions, i.e. defines $\partial_t u_{VM}$ in the dual space.
\end{proof}

Thanks to the finite time extinction we can state the reconstruction of any type of input data $f$ by integration over $\phi_{VM}(t)$ in general.
\begin{theorem}[Reconstruction of the input data]\label{thm:recon_vm}
It holds that
\begin{align}
\label{eq:recon_vm}
 f =  P_0(f) + \int_0^\infty \phi_{VM}(t) ~dt.
 \end{align}
\end{theorem}
\begin{proof}
Since for each vector $g \in \mathbb{R}^n$ the constant function $v=g$ is an element of $ W^{1,1}_{loc}(\mathbb{R}_+,\mathbb{R}^n)$, we can use \eqref{eq:formalFiltering0}

$$
\int_0^\infty g\cdot \phi_{VM}(t) ~dt 
= -  \int_0^\infty g \cdot \partial_{t}u_{VM}(t) ~dt = - g\cdot \int_0^\infty  \partial_{t}u_{VM}(t) ~dt .
$$
Hence, with the well-defined limits of $u_{VM}$ at $t=0$ and $t\rightarrow \infty$ we have
$$ g\cdot \int_0^\infty  \phi_{VM}(t) ~dt = -  g \cdot \int_0^\infty \partial_{t}u_{VM}(t) ~dt
= g \cdot (f - P_0(f)), $$
which yields the assertion since $g$ is arbitrary.
\end{proof}

In analogy to the classical linear setting, we would like to define a filtering of the wavelength representation via
\begin{align}
\label{eq:filtering}
 \hat{u}_{filtered} = w_0 \;P_0(f) + \int_0^\infty w(t) \; \phi_{VM}(t) ~dt
 \end{align}
for $w_0 \in \mathbb{R}$ and a suitable filter function $w(t)$.
While the above formulation is the most intuitive expression for the filtering procedure, we have to take care of the regularity of $\phi_{VM}$ and hence understand the integral on the right-hand side in the sense of \eqref{eq:formalFiltering0}, i.e.
\begin{equation}
\label{eq:formalFiltering}
	\hat{u}_{filtered} = w_0 \;P_0(f) - \int_0^\infty (t w'(t) + w(t)) \; \partial_{t}u_{VM}(t) ~dt.
\end{equation}

\subsection{Scale Space Representation}
\label{sec:SpectralScaleSpace}

A spectral representation based on the gradient flow formulation \eqref{eq:scaleSpace} was the first work towards defining a nonlinear spectral decomposition and has been conducted by Guy Gilboa in \cite{Gilboa_SSVM_2013_SpecTV,Gilboa_spectv_SIAM_2014} for the case of $J$ being the TV regularization. In our conference paper \cite{spec_one_homog}, we extended this notion to general absolutely one-homogeneous functionals by observing that the solution of the gradient flow coincides with the one of the variational method in the case of $f$ being an eigenfunction, i.e., for $\|f\|=1$, $\lambda f \in \partial J(f)$, the solution to \eqref{eq:scaleSpace} is given by
\begin{align*}
u_{GF}(t) = \left\{
\begin{array}{cc}
(1 -  t \lambda ) f & \text{ for } t\leq \frac{1}{\lambda}, \\
0 & \text{ else. }
\end{array}
\right.
\end{align*}
The latter motivates exactly the same definitions as for the variational method. In particular, we define the wavelength decomposition of the input data $f$ by
$$\phi_{GF}(t) = t \partial_{tt} u_{GF}(t). $$

As in the previous section, we will show that $\phi_{GF}$ also is a weak spectral wavelength decomposition in the sense of definition \ref{def:weakSpectralDecomposition} for any convex absolutely one-homogeneous regularization functional $J$.

\begin{proposition}[Finite time extinction]
\label{prop:finiteTimeExtinctionGF}
Let $J$ be an absolutely one-homogeneous functional, and $f$ be arbitrary. Then there exists a time $T<\infty$ such that $u_{VM}$ determined via \eqref{eq:variationalMethod} meets
$$u_{GF}(T) = P_0(f).$$
\end{proposition}
\begin{proof}
Since any subgradient is orthogonal to ${\cal N}(J)$ the same holds for $\partial_t u$, hence
$P_0(u(t))=P_0(u(0))=P_0(f)$ for all $t \geq 0$. On the other hand we see that
\begin{eqnarray*}
\frac{1}2 \frac{d}{dt} \Vert Q_0 u_{GF}\Vert^2 &=& \langle Q_0 u_{GF}, Q_0 \partial_t u_{GF} \rangle = - \langle Q_0 u_{GF}, p_{GF} \rangle = -  \langle   u_{GF}, p_{GF} \rangle \\ &=& - J(u_{GF}) = -
J(Q_0 u_{GF}) \leq - c_0 \Vert Q_0 u_{GF} \Vert.
\end{eqnarray*}
For $t$ such that $\Vert Q_0 u_{GF} \Vert \neq 0$ we conclude
$$  \frac{d}{dt} \Vert Q_0 u_{GF}\Vert \leq -  c_0, $$
thus
$$ \Vert Q_0 u_{GF}(t) \Vert \leq \Vert f \Vert -  c_0 t. $$
Due to the positivity of the norm we conclude that $\Vert Q_0 u_{GF}(T) \Vert = 0$ for $T \geq
\frac{\Vert f \Vert}{c_0}$.
\end{proof}

The regularity $\partial_t u_{GF}(t) \in L^\infty$ is guaranteed by the general theory on gradient flows (cf. \cite[p. 566, Theorem 3]{evans}), in our case it can also be inferred quantitatively from the a-priori bound on the subgradients in Conclusion \ref{conclu:boundedSubdiff}. With the same proof as
in Proposition \ref{prop:VMlipschitz} we can analyze $\phi_{GF}$ as a bounded linear functional:
\begin{proposition}\
The function $u_{GF}: \R^+ \rightarrow \R^n$ is Lipschitz continuous. Moreover, the spectral representation satisfies $\phi_{GF} \in (W_{loc}^{1,1}(\mathbb{R}^+,\mathbb{R}^n))^*$.
\end{proposition}

Thus, we have the same regularity of the distribution as in the case of the variational method and can define filters in the same way.

Naturally, we obtain exactly the same reconstruction result as for the variational method.
\begin{theorem}[Reconstruction of the input data]\label{thm:recon_gf}
It holds that
\begin{align}
\label{eq:recon_gf}
 f =  P_0(f) + \int_0^\infty \phi_{GF}(t) ~dt.
 \end{align}
\end{theorem}
Furthermore, due to the finite time extinction and the same smoothness of $u_{GF}$ as for $u_{VM}$ we can define the formal filtering by
\begin{align}
\label{eq:formalFiltering_GF}
\hat{u}_\text{filtered} = w_0 P_0(f) - \int_0^\infty (t w'(t) + w(t)) \; \partial_{t}u_{GF}(t) ~dt,
\end{align}
for all filter functions $w \in W_\text{loc}^{1,1}$.

\subsection{Inverse Scale Space Representation}
\label{sec:SpectralInverseScaleSpace}
A third way of defining a spectral representation proposed in \cite{spec_one_homog} is via the inverse scale space flow equation \eqref{eq:inverseScaleSpace}. Again, the motivation for the proposed spectral representation is based on the method's behavior on eigenfunctions. For  $\|f\|=1$, $\lambda f \in \partial J(f)$, the solution to \eqref{eq:inverseScaleSpace} is given by
\begin{align*}
v_{IS}(s) = \left\{
\begin{array}{cc}
0 & \text{ for } s\leq \lambda, \\
f & \text{ else. }
\end{array}
\right.
\end{align*}
As we can see, the behavior of \eqref{eq:inverseScaleSpace} is fundamentally different to the one of \eqref{eq:variationalMethod} and \eqref{eq:scaleSpace} in two aspects: Firstly, stationarity is reached in an inverse fashion, i.e. by starting with zero and converging to $f$. Secondly, the primal variable $v_{IS}(s)$ has a piecewise constant behavior in time opposed to the piecewise linear behavior in the previous two cases. Naturally, only one derivative of $v$ is necessary to obtain peaks. We define
$$ \tilde \phi_{IS}(s) = \partial_s v_{IS}(s) = - \partial_{ss} q_{IS}(s) $$
to be the \textit{frequency representation} of $f$ in the inverse scale space setting.

We recall from section \ref{sec:WhatIsSpectral} that we can relate frequency and wavelength representations by a change of variable $s=1/t$ yielding 
\begin{equation}
	\label{eq:representationConversion}
u_{IS}(t) = v_{IS}(\frac{1}t), \quad p_{IS}(t) = q_{IS}(\frac{1}t), \quad \phi_{IS}(t)= - \partial_s v_{IS}(\frac{1}t) = t^2 \partial_t u_{IS}(t).
\end{equation}
Note that with these conversions we have
$$ \int_0^\infty \phi(t) \cdot v(t) ~dt = \int_0^\infty \tilde \phi(s) \cdot v(1/s) ~ds,$$
hence we may equally well consider integrations in the original variable $s$.

In the following we will show that the above $\phi_{IS}$ is a weak spectral wavelength decomposition in the sense of definition \ref{def:weakSpectralDecomposition} for any convex absolutely one-homogeneous regularization functional $J$.

Analogous to the other methods, the inverse scale space methods has finite time extinction, and as we see from the proof even at the same time as the variational method:
\begin{proposition}[Finite time extinction]
\label{prop:finiteTimeExtinctionIS}
Let $J$ be an absolutely one-homogeneous functional, and $f$ be arbitrary. Then there exists a time $T<\infty$ such that $u_{IS}$ determined via \eqref{eq:variationalMethod} meets
$u_{IS}(T) = P_0(f).$
\end{proposition}
\begin{proof}
It is straight-forward to see that $v_{IS}(s) =  P_0(f)$ for $s \leq s_0$, where $s_0$ is the maximal value such that $s_0 (f- P_0(f)) \in \partial J(0)$. The conversion to $u_{IS}$ yields the finite time extinction.
\end{proof}

Furthermore note that the inverse scale space flow \eqref{eq:inverseScaleSpace} can be written as a gradient flow on the dual variable $q_{IS}(s)$ with respect to the convex functional $J^*(q)-\langle f , q\rangle $. This guarantees that $\partial_s q_{IS}(s) \in L^\infty$, thus $v_{IS}(s) \in L^\infty$, and hence also $u_{IS}(t) \in L^\infty$.

\begin{proposition}\
The function $u_{IS}: \R^+ \rightarrow \R^n$ is bounded almost everywhere and  $p_{IS}: [t_0,\infty) \rightarrow \R^n$ is Lipschitz continuous for every $t_0 > 0$. Moreover, the weak spectral representation satisfies $\phi_{IS} \in (W_{loc}^{1,1}(\mathbb{R}^+,\mathbb{R}^n))^*$.
\end{proposition}
\begin{proof}
The standard energy dissipation in the inverse scale space flow yields that $s \mapsto \Vert v_{IS}(s) - f\Vert$ is a non-increasing function, hence it is bounded by its value $\Vert f \Vert$ at $s=0$. Hence, by the triangle inequality
$$ \Vert v_{IS}(s) \Vert \leq 2 \Vert f \Vert, \quad  \Vert u_{IS}(t) \Vert \leq 2 \Vert f \Vert$$
for all $s,t>0$. The first inequality also implies the Lipschitz continuity of $q_{IS}$ on $\mathbb{R^+}$ and hence by concatenation with $t \mapsto \frac{1}t$ Lipschitz continuity of $p_{IS}$ on $[t_0,\infty)$.
\end{proof}


We can now consider filterings of the inverse scale space flow representation, again by formal integration by parts
\begin{align*}
 \hat{u}_{filtered} &= \int_0^\infty w(t) \; \phi_{IS}(t) ~dt
= - \int_0^\infty w(\frac{1}s) \; \partial_s v_{IS}(s) ~ds \\
&= - \int_{\frac{1}T}^\infty w(\frac{1}s) \; \partial_s v_{IS}(s) ~ds    =  w_0 f - w_T P_0(f) - \int_0^T  \frac{1}{s^2} w '(\frac{1}s) \; v_{IS}(s) ~ds \\
 &=  w_0 f - w_T P_0(f) - \int_0^T  w'(t) \; u_{IS}(t) ~dt ,
\end{align*}
where $T$ is the finite extinction time stated by Proposition \ref{prop:finiteTimeExtinctionIS}. The last line can be used to define filterings in the inverse scale space flow setting. Note that $w(t) = 1$ for all $t$ leads to a reconstruction of the $Q_0f$, i.e.
\begin{align}
\label{eq:recon_iss}
 f =  P_0(f) + \int_0^\infty \phi_{IS}(t) ~dt.
\end{align}

\subsection{Definitions of the Power Spectrum}
\label{sec:spectrum}

As in the linear case, it is very useful to measure in some sense the ``activity" at each frequency (scale). 
This can help identify dominant scales and design better the filtering strategies (either manually or automatically).
Moreover, one can obtain a notion of the type of energy which is preserved in the new representation using an analog of
Parseval's identity. While the amount of information on various spatial scales in linear and nonlinear scale spaces has been analyzed using Renyi’s generalized entropies in \cite{Sporring-Weickert-tip99}, we will focus on defining a spectral power spectrum. As we have seen above, at least for an orthogonal spectral definition there is a natural definition of the power spectrum as the measure
\begin{equation}\label{eq:S3}
	S^2(t)  = \Phi_t \cdot f \equiv S^2_3(t),
\end{equation}
which yields a Parseval identity. 
Note that for $S^2$ being absolutely continuous with respect to the Lebesgue measure with density $\rho$ we can define a continuous power spectrum $s(t)  = \sqrt{\rho(t)}$ and have
\begin{equation}
\Vert f \Vert^2 = \int_0^\infty s(t)^2 ~dt.
\end{equation}
On the other hand, if $S^2$ is a sum of concentrated measures, $S^2(t) = \sum_j a_j \delta(t-t_j)$ we can define $s_j = \sqrt{a_j}$ and have the classical Parseval identity
\begin{equation}
\Vert f \Vert^2 = \sum_j s_j^2.
\end{equation}

We will now briefly recall two earlier definitions of the spectrum and propose a third new one. For the sake of simplicity, we omit the subscripts $VM$, $GF$, and $IS$ in the following discussion when all three variants can be used.

In \cite{Gilboa_SSVM_2013_SpecTV,Gilboa_spectv_SIAM_2014} a $L^1$ type spectrum was suggested for the TV spectral framework
(without trying to relate to a Parseval identity),
\begin{equation}
\label{eq:S1}
S_1(t) := \|\phi(t)\|_{1}.
\end{equation}
{Considering the mathematical definition of $\phi$ as having components in $(W^{1,1}_{loc})^*$, we can see that mollification with a $W^{1,1}_{loc}$ function is needed in order to obtain a well-defined version of equation \eqref{eq:S1}. A simple choice would be
\begin{equation}
S^\sigma_1(t) := \left\|\int_0^T g_\sigma(t) \; \phi(t;x)~dt \right\|_{L^1(\Omega)}
\end{equation}
for a Gaussian function $g_\sigma(t)$ with very small $\sigma$. }

In \cite{spec_one_homog} the following definition was suggested for the gradient flow,
\begin{equation} \label{eq:S2}
	S_2^2 (t) = t^2  { \frac{d^2}{dt^2}J(u_{GF}(t)) } = 2 t {\langle \phi_{GF}(t),  p_{GF}(t)} \rangle.
= - t^2 \frac{d}{dt} \Vert p_{GF}(t)\Vert^2\end{equation}
From the dissipation properties of the flow one can deduce that $S_2^2$ is always nonnegative.
With this definition the following analogue of the Parseval identity can be shown :
\begin{eqnarray}
	\Vert f \Vert^2 & = & - \int_0^\infty \frac{d}{dt} \Vert u_{GF}(t) \Vert^2 ~dt = 2 \int_0^\infty \langle p_{GF}(t), u_{GF}(t) \rangle ~dt =
	 2 \int_0^\infty J(u_{GF}(t))~dt \nonumber \\ & = & - 2 \int_0^\infty t \frac{d}{dt} J(u_{GF}(t))~dt =  \int_0^\infty t \frac{d}{dt} J(u_{GF}(t))~dt = \int_0^\infty S_2(t)^2  ~dt.	\end{eqnarray}
	
Below, we will show that under certain conditions indeed $S_3^2$ is an equivalent realization of $S_2^2$.
In Fig. \ref{fig:spectrum} we present numerical examples of the different behavior of those variants.
	
%
%
%

\section{Analysis of the Spectral Decompositions}
\label{sec:analysis}

\subsection{Basic Conditions on the Regularization}

To analyze the behavior and relation of the different spectral decompositions, let us make the following definition.
\begin{definition}[MINSUB]
\label{def:MINSUB}
We say that $J$ meets (MINSUB) if for all $u \in R^n$, the element $\hat{p}$ determined by
\begin{align}
\label{eq:minimizingSubgradient}
\hat{p} = \arg \min_p \|p\|^2 \ \text{subject to } p \in \partial J(u) ,
\end{align}
meets
\begin{align*}
\langle \hat{p}, \hat{p} - q \rangle = 0 \ \ \forall q \in \partial J(u).
\end{align*}
\end{definition}

To give some intuition about what the condition (MINSUB) means, let us give two examples of regularizations that meet (MINSUB).
\begin{example}[$\ell^1$ regularization meets (MINSUB)]\
Consider $J(u) = \|u\|_1$. The characterization of the $\ell^1$ subdifferential yields
\begin{align}
\label{eq:subdiffCharak}
q \in \partial J(u) \qquad \Leftrightarrow \qquad q_l \left\{ \begin{array}{lc} = 1 & \text{ if } u_l >0, \\ =-1 & \text{ if } u_l <0, \\ \in [-1,1] & \text{ if } u_l =0   .
\end{array} \right.
\end{align}
Consequently, the $\hat{p}$ defined by equation \eqref{eq:minimizingSubgradient} meets $\hat{p}_l = 0$ for all $l$ with $u_l = 0$. Consider any other $q \in \partial J(u)$. Then
\begin{align*}
\langle \hat{p}, \hat{p} - q \rangle &= \underbrace{\sum_{l, u_l>0} \hat{p_l}(\underbrace{\hat{p_l}}_{=1} - \underbrace{q_l}_{=1})}_{=0}+\underbrace{\sum_{l, u_l<0} \hat{p_l}(\underbrace{\hat{p_l}}_{=-1} - \underbrace{q_l}_{=-1})}_{=0}+\underbrace{\sum_{l, u_l=0} \underbrace{\hat{p_l}}_{=0}(\hat{p_l} - q_l)}_{=0} =0,
\end{align*}
which shows that $J(u) = \|u\|_1$ meets (MINSUB).
\end{example}

Another interpretation of the (MINSUB) condition is geometric. Since in Hilbert spaces the scalar product being zero expresses orthogonality, one could also think of (MINSUB) requiring the subdifferentials of $J$ being well-behaved polyhedrons (or being single valued). For illustration purposes, consider figure \ref{fig:minsubGeometric}. If the faces of polyhedron representing all possible subgradients are oriented such that they are tangent to the circle with radius $\|\hat{p}\|$, then (MINSUB) is met.

\begin{figure}
\includegraphics[width=0.24\textwidth]{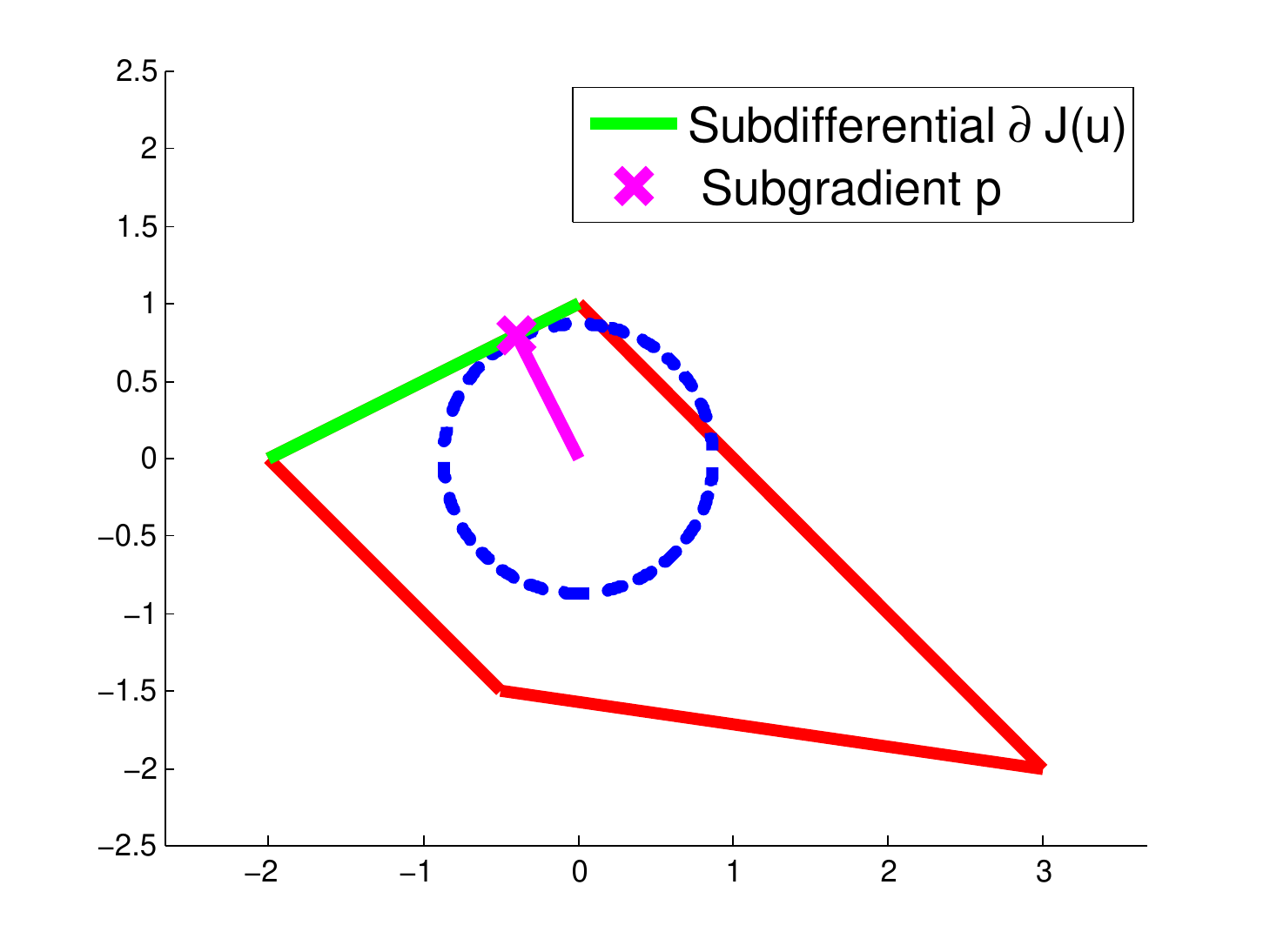}
\includegraphics[width=0.24\textwidth]{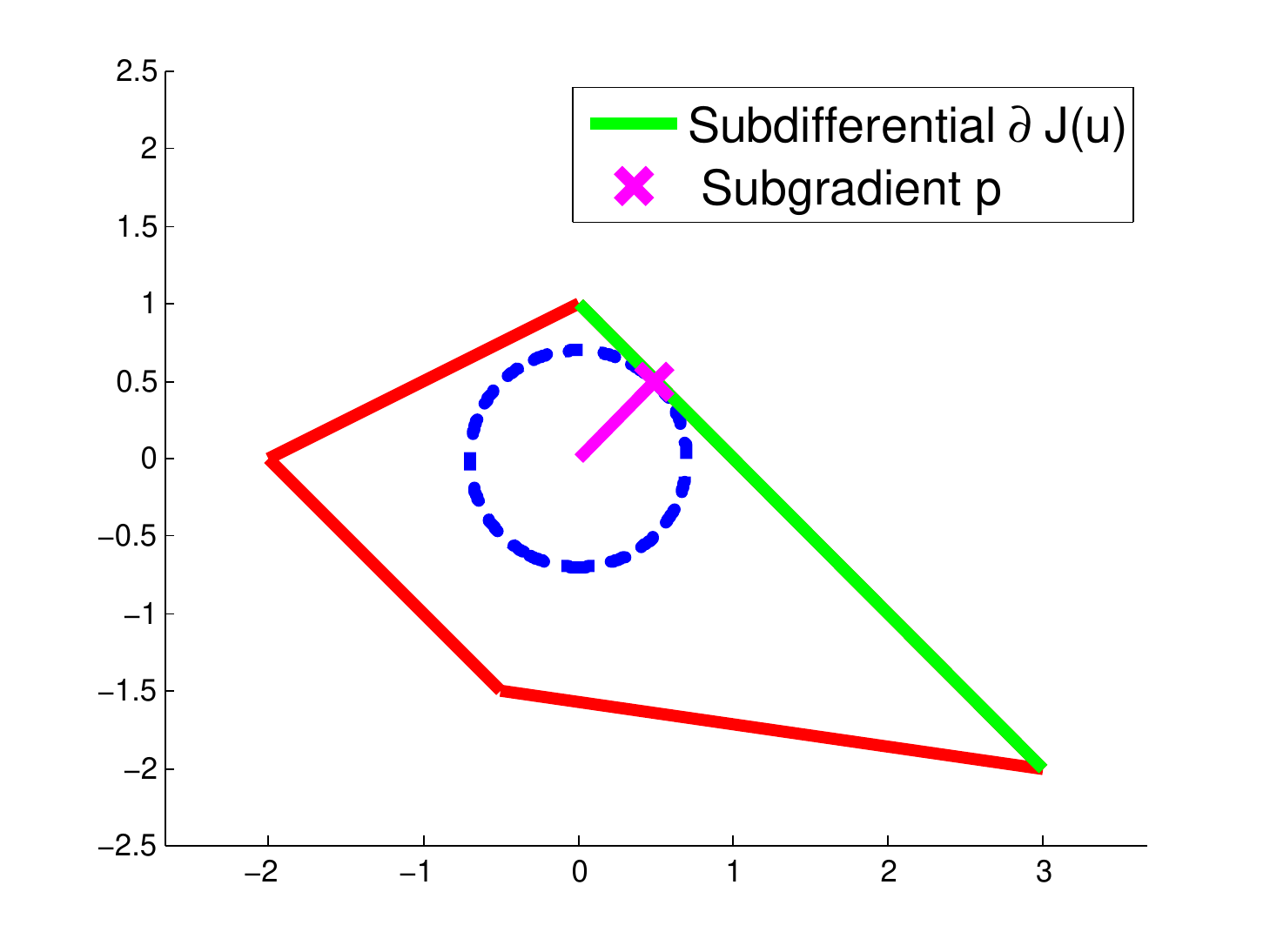}
\includegraphics[width=0.24\textwidth]{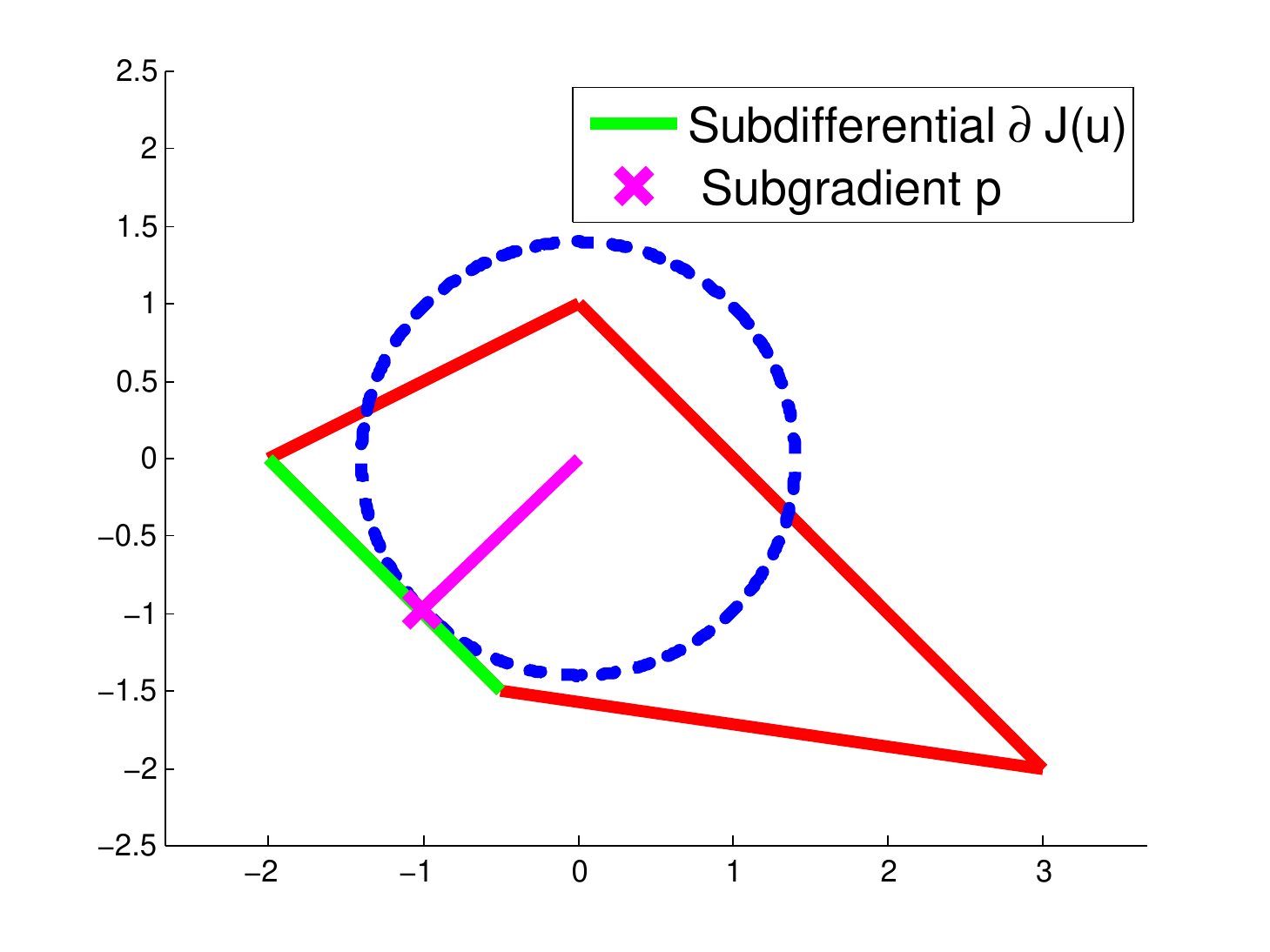}
\includegraphics[width=0.24\textwidth]{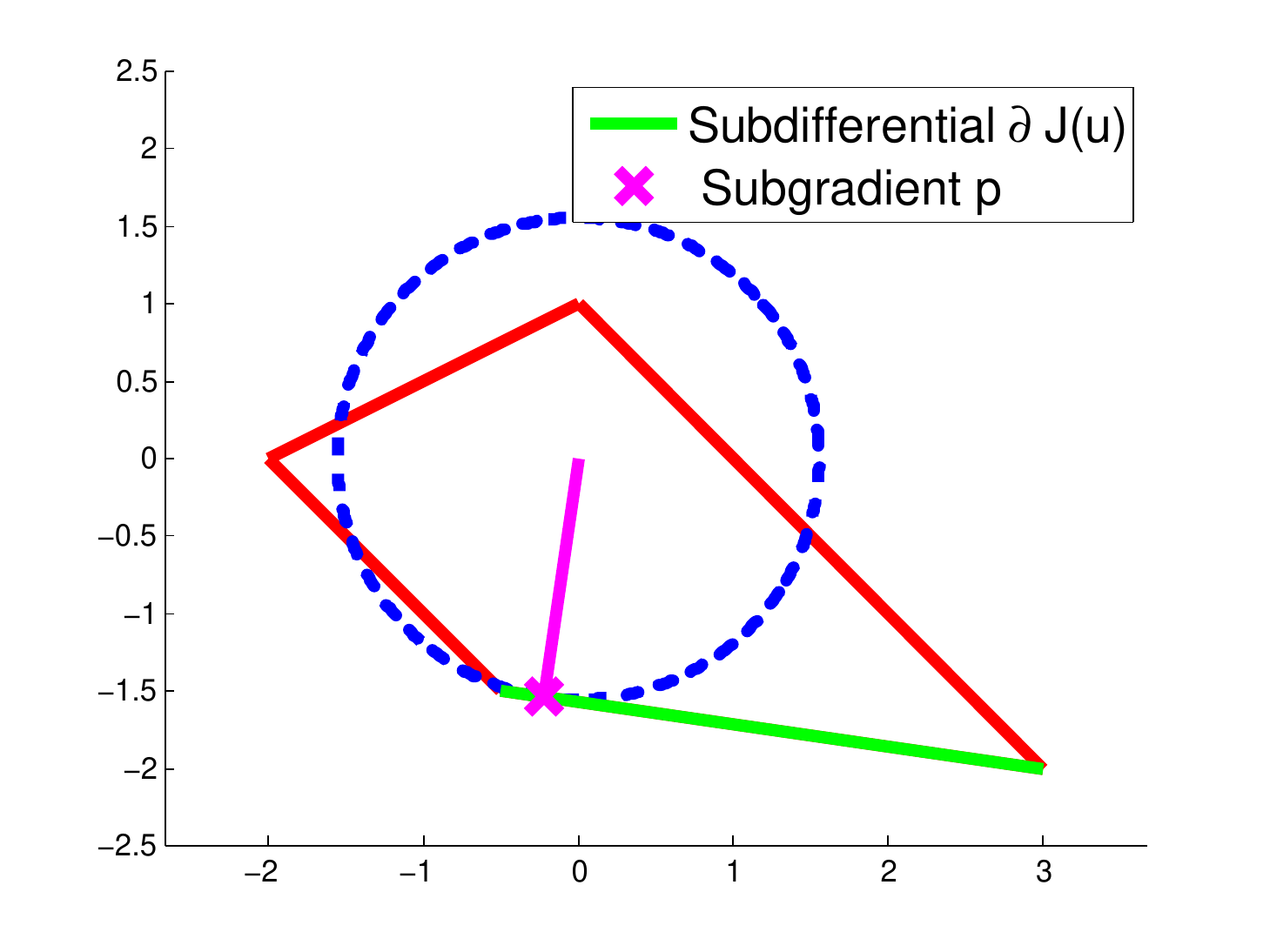}
\caption{Geometric illustration of the (MINSUB) property in 2D with the help of some polyhedron. The green faces illustrate subdifferentials for particular primal variables $u$ and the magenta points mark the subgradient with minimal norm. As we can see all green faces are tangent to the circle such that $q-\hat{p}$ is orthogonal to $\hat{p}$ for any $q\in \partial J(u)$. Hence, (MINSUB) is met.}
\label{fig:minsubGeometric}
\end{figure}

A second property we are interested in, is a particular behavior of the flow that allows to state that the spectral respresentation merely consists of a collection of $\delta$-peaks. For this, we need the definition of a polyhedral seminorm.
\begin{definition}[Polyhedral Seminorm (PS)]
\label{def:PS}
We say that $J$ induces a  \textit{polyhedral seminorm} (PS) if there exists a finite set whose convex hull equals $\partial J(0)$.
\end{definition}

Polyhedral Seminorms have several interesting properties:
\begin{proposition}
\label{prop:finitelyRepresentable}
Let $J$ satisfy (PS), then the set
$$ \{ \arg \min_p \|p\|^2, \ \text{ such that } p \in \partial J(u) ~|~ u \in {\cal X} \}$$
is finite.
\end{proposition}
\begin{proof}
By Lemma \ref{subgradientlemma1} we know that $\partial J(u)$ is the intersection of the polyhedral shape $\partial J(0)$ and the linear manifold defined by $\langle p, u \rangle = J(u)$. Since the intersection cannot contain any interior part of $\partial J(0)$ and due to (PS), the intersection must be a facet of $\partial J(0)$. Hence, the set of minimizers of $\Vert p \Vert$ in the subgradients of $\partial J(u)$ is contained in the set of minimizers of $\Vert p \Vert$ on single facets of $\partial J(0)$. Since the number of facets is finite by (PS), also the set of minimizers (for each facet there is a unique one) is finite.
\end{proof}

As an example, consider the $\ell^1$ norm on $\mathbb{R}^n$ again. The set $\partial \|0\|_1$ coincides with the unit $\ell^\infty$ ball and therefore is polyhedral. The set considered in Proposition \ref{prop:finitelyRepresentable} is the collection of all $p$ such that the $i$-th component of $p$ is in $\{-1,1,0\}$, which is of course finite. Note that all regularizations of the form $J(u) = \|Ku\|_1$, for an arbitrary linear operator $K$, are also polyhedral as $\partial J(0) = K^T \partial \|0\|_1$. Interestingly, the most general form of a $J$ meeting (PS) is in fact $J(u) = \|Ku\|_\infty$:
\begin{proposition}\
Let an absolutely one-homogeneous function $J: \mathbb{R}^n \rightarrow \mathbb{R}$ meet (PS). Then there is a matrix $P$ such that $J(u) = \|Pu\|_\infty$.
\end{proposition}
\begin{proof}
Let $u \in \mathbb{R}^n$ be arbitrary, and let $p\in \partial J(u)$. By (PS) we can write $p$ as a convex combination of finitely many subgradients $p_i \in \partial J(0)$, $p = \sum_i \alpha_i p_i$. We find
\begin{align*}
J(u) = \langle u, p \rangle = \sum_i \alpha_i \langle u, p_i \rangle \leq \sum_i \alpha_i J(u) = J(u).
\end{align*}
The fact that equality holds in the above estimate shows that $p_i \in \partial J(u)$ for all $i$ with $\alpha_i >0$, such that $J(u) = \max_i \langle p_i, u \rangle$. Since $J(u) = J(-u)$ by the  absolute one-homogenity, we may as well consider $J(u) = \max_i |\langle p_i, u \rangle|$, or, after writing the $p_i$ as rows of a matrix $P$, $J(u) = \|Pu\|_\infty$.
\end{proof}


To be able to state our main results compactly, let us introduce a third abbreviation in addition to (PS) and (MINSUB).
\begin{definition}[DDL1]
\label{def:DDL1}
We refer to the specific assumption that we are considering diagonally dominant $\ell^1$ regularization, i.e. the particular case where $J(u) = \|Ku\|_1$ for a matrix $K$ such that $KK^*$ is diagonally dominant, as (DDL1).
\end{definition}

\subsection{Connection between Spectral Decompositions}
Using the three assumptions and definitions (PS), (MINSUB) and (DDL1) introduced in the previous subsection, we can state the following results. 

\begin{theorem}[Piecewise dynamics for $J$ satisfying (PS)]
\label{thm:piecewiseDynamics}
If $J$ meets (PS) then the scale space flow, the inverse scale space flow, and the variational method have piecewise linear respectively piecewise constant dynamics. In more detail:
\begin{itemize}
\item There exist finitely many $0=t_0<t_1<t_2<...<t_K=\infty $ such that the solution of \eqref{eq:scaleSpace} is given by
\begin{align}
\label{eq:u_GF_PS}
u_{GF}(t) = u_{GF}(t_i) - (t-t_i) p_{GF}(t_{i+1}),
\end{align}
for $t \in [t_i,t_{i+1}]$. In other words $p_{GF}(t_{i+1})\in \partial J( u_{GF}(t))$ for $t \in [t_i,t_{i+1}]$ such that the dynamics of the flow is piecewise linear in $u_{GF}(t)$ and piecewise constant in $p_{GF}(t)$.
\item There exist finitely many $0<s_1<s_2<...<s_L=\infty $ such that the solution of \eqref{eq:inverseScaleSpace} is given by
\begin{align}
\label{eq:q_IS_PS}
q_{IS}(s) = q_{IS}(s_i) - (s-s_i) (f-v_{IS}(s_{i+1})),
\end{align}
for $s \in [s_i,s_{i+1}]$. In other words $q_{IS}(s)\in \partial J( v_{IS}(s_{i+1}))$ for $s \in [s_i,s_{i+1}]$ such that the dynamics of the flow is piecewise linear in $q_{IS}(t)$ and piecewise constant in $v_{IS}(t)$.
\item  There exist finitely many $0<t_1<t_2<...<t_L=\infty $ such that the solution of \eqref{eq:variationalMethod} satisfies that $u_{VM}$ is an affinely linear function of $t$ for $t \in [t_i,t_{i+1}]$, and for $p_{VM}(t) \in \partial J(u_{VM}(t))$ denoting the corresponding subgradient, $q_{VM}(s) = p_{VM}(1/s)$ is an affinely linear function of $s$ for $s \in [s_i,s_{i+1}]$, $s_i = \frac{1}{t_i}$.
\end{itemize}
\end{theorem}
\begin{proof}

\textbf{Gradient Flow. }
Let us begin with the proof of the piecewise linear dynamics of the scale space flow. For any $t_{i}\geq 0$ (starting with $t_0=0$ and $u_{GF}(0)=f$) we consider
\begin{align}
\label{eq:minimizationGf}
p_{GF}(t_{i+1}) = \arg \min_p \|p\|^2 \qquad \text{such that } p \in \partial J(u_{GF}(t_i))
\end{align}
and claim that $p_{GF}(t_{i+1}) \in \partial J(u_{GF}(t_i) - (t-t_i) p_{GF}(t_{i+1}))$ for small enough $t>t_i$.
Due the property of $J$ having finitely representable subdifferentials (Proposition \ref{prop:finitelyRepresentable}), we know that for any $t$ there has to exist a $q^j$ in the finite set $M$ such that $q^j \in \partial  J(u_{GF}(t_i) - (t-t_i) p_{GF}(t_{i+1}))$. The characterization of the subdifferential of absolutely one-homogeneous functions \eqref{eq:subdifferential} tells us that $q^j \in \partial  J(u_{GF}(t_i) - (t-t_i) p_{GF}(t_{i+1}))$ is met by those $q^j$ which maximize $\langle q^j, u_{GF}(t_i) - (t-t_i) p_{GF}(t_{i+1})\rangle$. We distinguish two cases:
\begin{itemize}
\item First of all consider $q^j \in \partial J(u_{GF}(t_i))$. In this case the optimality condition to \eqref{eq:minimizationGf} tells us that
\begin{align}
 \langle p_{GF}(t_{i+1}), p_{GF}(t_{i+1})-q \rangle \leq 0 \ \ \forall q \in \partial J(u_{GF}(t_i)),
 \label{eq:optiPgradFlow}
\end{align}
and we find
\begin{align*}
\langle u_{GF}(t_i) - (t-t_i) p_{GF}(t_{i+1}),  q^j \rangle =& J( u_{GF}(t_i)) - (t-t_i) \langle p_{GF}(t_{i+1}),  q^j \rangle, \\
=& \langle u_{GF}(t_i), p_{GF}(t_{i+1})\rangle  - (t-t_i) \langle p_{GF}(t_{i+1}), q^j \rangle, \\
=& \langle u_{GF}(t_i) - (t-t_i) p_{GF}(t_{i+1}),  p_{GF}(t_{i+1}) \rangle \\
&+ (t-t_i) \langle p_{GF}(t_{i+1}),p_{GF}(t_{i+1})- q^j \rangle, \\
\stackrel{\eqref{eq:optiPgradFlow}}{\leq}& \langle u_{GF}(t_i) - (t-t_i) p_{GF}(t_{i+1}),  p_{GF}(t_{i+1}) \rangle .
\end{align*}

\item For $q^j \notin \partial J(u_{GF}(t_i))$ we can use Remark \ref{rem:subdiffRemark}, define $c^j(t_i) := J(u_{GF}(t_i)) - \langle q^j, u_{GF}(t_i) \rangle>0$. We compute
\begin{align*}
\langle u_{GF}(t_i) - (t-t_i) p_{GF}(t_{i+1}),  q^j \rangle =& \langle u_{GF}(t_i), q^j \rangle - (t-t_i)\langle p_{GF}(t_{i+1}),  q^j \rangle \\
=&  J(u_{GF}(t_i)) - c^j(t_i) - (t-t_i) \langle p_{GF}(t_{i+1}),  q^j \rangle \\
=&  J(u_{GF}(t_i)) - (t-t_i)\| p_{GF}(t_{i+1})\|^2 \\
&- \underbrace{\left( c^j(t_i) - (t-t_i)(\| p_{GF}(t_{i+1})\|^2 - \langle p_{GF}(t_{i+1}),  q^j \rangle)\right)}_{\geq 0 \text{ for small }t>t_i} \\
\leq& \langle u_{GF}(t_i) - (t-t_i) p_{GF}(t_{i+1}),  p_{GF}(t_{i+1})  \rangle \ \text{ for small } t.&
\end{align*}
Since the set of all $q^j$ is finite, we can conclude that there exists a smallest time $t>t_i$ over all $j$ up until which $p_{GF}(t_{i+1}) \in \partial J(u_{GF}(t_i) - (t-t_i) p_{GF}(t_{i+1}))$.
\end{itemize}
To see that the number of times $t_k$ at which the flow changes is finite, recall that each $p_{GF}(t_{i+1})$ is determined via equation \eqref{eq:minimizationGf} and that the corresponding minimizer is unique. Since $p_{GF}(t_i)$ is in the feasible set, we find $\|p_{GF}(t_{i+1})\|<\|p_{GF}(t_{i})\|$ (unless $p_{GF}(t_{i}) = p_{GF}(t_{i+1})$ in which case the time step $t_i$ was superfluous to consider). Since the number of possible $p_{GF}(t_i)$ is finite by Proposition \ref{prop:finitelyRepresentable}, and the property $\|p_{GF}(t_{i+1})\|<\|p_{GF}(t_{i})\|$ shows that each $p_{GF}(t_i)$ be attained at most once, the number of $t_i$ at which the flow changes is finite. This concludes the proof for the scale space case.
\vspace{0.3cm}

\textbf{Inverse Scale Space Flow. }Similar to the above proof, for any $s_{i}\geq 0$ (starting with $s_0=0$ and $q_{IS}(0)=0$) we consider
$$v_{IS}(s_{i+1}) = \arg \min_v \|v-f\|^2 \qquad \text{such that } v \in \partial J^*(q_{IS}(s_i))$$
and claim that $q_{IS}(s_{i})+(s-s_i)(f-v_{IS}(s_{i+1}))\in \partial J(v_{IS}(s_{i+1}))$ for small enough $s>s_i$.

The optimality condition to the above minimization problem is
\begin{align}
\label{eq:inequality}
 \langle f-v_{IS}(s_{i+1}), v_{IS}(s_{i+1}) - v \rangle \geq 0  \qquad \forall v \in \partial J^*(q_{IS}(s_i)).
\end{align}
By choosing $v = 0 $ and $v = 2v_{IS}(s_{i+1})$ we readily find
$$ \langle f-v_{IS}(s_{i+1}), v_{IS}(s_{i+1}) \rangle = 0,$$
and therefore
$$\langle q_{IS}(s_{i})+(s-s_i)(f-v_{IS}(s_{i+1})), v_{IS}(s_{i+1}) \rangle = J(v_{IS}(s_{i+1}) )$$
for all $t$. Furthermore \eqref{eq:inequality} reduces to
\begin{align}
\label{eq:inequality2}
 \langle f-v_{IS}(s_{i+1}), v \rangle \leq 0  \qquad \forall v \in \partial J^*(q_{IS}(s_i)).
\end{align}
%
It is hence sufficient to verify that
$$q_{IS}(s_{i})+(s-s_i)(f-v_{IS}(s_{i+1})) \in \partial J(0) $$
for $s>s_i$ sufficiently small. It is well known (cf. \cite{ConvexPolytopes}) that any set that is represented as the convex hull of finitely many points, can also be represented as the intersection of finitely many halfspaces, such that
$$ \partial J^*(0) =  \{q ~|~ Bq \leq c \} $$
for some matrix $B$ and vector $c$. Let $b_l$ denote the rows of the matrix $B$. Since $\langle b_l, q \rangle = J(b_l)$ for $q \in \partial J(b_l)$, it is obvious that $J(b_l)\leq c_l$. We distinguish two cases:
\begin{itemize}
\item If $b_l \in \partial J^*(q_{IS}(s_i))$, then \eqref{eq:inequality2} yields
$$\langle q_{IS}(s_{i+1})+(s-s_i)(f-v_{IS}(s_{i+1})), b_l \rangle = J(b_l) + (s-s_i)\langle f-v_{IS}(s_{i+1}), b_l \rangle \leq J(b_l)$$
\item If $b_l \notin \partial J^*(q_{IS}(s_i))$, then $\langle q_{IS}(s_{i}), b_l \rangle < J(b_l)$, such that
$$\langle q_{IS}(s_{i})+(s-s_i)(f-v_{IS}(s_{i+1})), b_l \rangle \leq J(b_l)$$
for $s \in [s_i, s_{i+1}^l]$ with $s_{i+1}^l>s_i$ small enough.
\end{itemize}
We define $s_{i+1} = \min_l s_{i+1}^l$, which yields the piecewise constant behavior. Similar to the proof for the gradient flow, the $v_{IS}(s_{i+1})$ are unique and $v_{IS}(s_{i})$ is in the feasible set, such that $\|f-v_{IS}(s_{i+1})\|<\|f-v_{IS}(s_{i})\|$. Since the total number of possible constraints $v \in \partial J^*(p)$ is finite, there can only be finitely many times $s_i$ at which the flow changes.

\vspace{0.3cm}

\textbf{Variational Method. }
Since $q_{VM}$ is a Lipschitz continuous function of $s$, the subgradient changes with finite speed on $\partial J(0)$. This means that $p_{VM}$ needs finite time to change from a facet to another facet of the polyhedral set $\partial J(0)$, in other words there exist times $0<\tilde s_1<\tilde s_2<...<\tilde s_L=\infty $ such that $q_{VM}(\tilde s_i) \in \partial J(u_{VM}(s))$ for $s \in [\tilde s_i,\tilde s_{i+1}). $
Comparing the optimality condition for a minimizer at time $s$ and time $\tilde s_i$ we find
$$ s (v_{VM}(s) - \frac{\tilde s_i}s v_{VM}(s_i) -  \frac{s- \tilde s_i}s f) + q_{VM}(s) -q_{VM}(\tilde s_i)=0. $$
Now let $v$ be any element such that $q_{VM}(\tilde s_i) \in \partial J(v)$. Then we have
$$ \langle v_{VM}(s) - \frac{\tilde s_i}s v_{VM}(\tilde s_i) -  \frac{s- \tilde s_i}s f, v - v_{VM}(s) \rangle
= \frac{1}s \langle q_{VM}(s) -q_{VM}(\tilde s_i), v_{VM}(s) - v \rangle \geq 0.  $$
This implies that $v_{VM}(s)$ is the minimizer
$$ v_{VM}(s) = \arg \min_v \Vert v -  \frac{\tilde s_i}s v_{VM}(\tilde s_i) -  \frac{s- \tilde s_i}s f\Vert^2  \qquad \text{ such that  }q_{VM}(\tilde s_i) \in \partial J(v), $$
i.e. it equals the projection of $\frac{\tilde s_i}s v_{VM}(\tilde s_i) +  \frac{s- \tilde s_i}s f$ on the set of $v$ such that $q_{VM}(\tilde s_i) \in \partial J(v)$, which is an intersection of a finite number of half-spaces.
Since  $u_{VM}(\frac{1}s)$ is an affinely linear function for $s \in [\tilde s_i, \tilde s_{i+1}]$, its projection to the intersection of a finite number of half-spaces is piecewise affinely linear for $s \in [\tilde s_i, \tilde s_{i+1}]$. Hence, the overall dynamics of $u_{VM}$ is piecewise affinely linear in $t$. The piecewise affine linearity of $p_{VM}$ in terms of $s=\frac{1}t$ then follows from the optimality condition by a direct computation.
\end{proof}

Note that the results regarding the scale space and inverse scale space flows were to be expected based on the work \cite{polyhedralISS13} on polyhedral functions. The above notation and the proof is, however, much more accessible since it avoids the lengthy notation of polyhedral and finitely generated functions.

From Theorem \ref{thm:piecewiseDynamics} we can draw the following simple but important conclusion.
\begin{conclusion}
If $J$ meets (PS) then the scale space flow, the inverse scale space flow, and the variational method  have a well-defined spectral representation, which consists of finitely many $\delta$-peaks. In other words, the spectral representation is given by
\begin{align}
 \phi_*(t) = \sum_{i=0}^{N_* }\phi_{*}^i \delta(t-t_i), \qquad \text{ for  } * \in \{VM, GF, IS \},
\end{align}
and
the reconstruction formulas \eqref{eq:recon_vm}, \eqref{eq:recon_gf}, or \eqref{eq:recon_iss}, yield a decomposition of $f$ as
\begin{align}
\label{eq:finiteDecomposition}
 f = \sum_{i=0}^{N_* }\phi_{*}^i  , \qquad \text{ for  } * \in \{VM, GF, IS \},
\end{align}
where $t_i$ are the finite number of times where the piecewise behavior of the variational method, the gradient flow, or the inverse scale space flow stated in Theorem \ref{thm:piecewiseDynamics} changes. The corresponding $\phi^i_*$ can be seen as multiples of $\psi_*(t)$ arising from the polar decomposition of the spectral frequency representation. They are given by
\begin{align*}
\phi^i_{GF} =& t_i (p_{GF}(t_{i}) - p_{GF}(t_{i+1})), \\
\phi^i_{VM} =& t_i (u_{VM}(t_{i+1}) - 2u_{VM}(t_{i}) + u_{VM}(t_{i-1}), \\
\phi^i_{IS} =& u_{IS}(t_{i}) - u_{IS}(t_{i+1}),
\end{align*}
with $u_{IS}(t_0) = u_{VM}(t_0) = u_{VM}(t_{-1}) = f$, $p_{GF}(t_0)=0$, $t_0=0$.
\end{conclusion}

Naturally, we should ask what the relation between the different spectral decomposition methods proposed in Section \ref{sec:Spectral} is. We can state the following results:
\begin{theorem}[Equivalence of $GF$ and $VM$ under (MINSUB)]\
 Let $J$ be such that (PS) and (MINSUB) are satisfied. Then
\begin{enumerate}
\item $p_{GF}(s) \in \partial J(u_{GF}(t))$ for all $t\geq s$.
\item The solution $u_{GF}(t)$ meets $u_{GF}(t)=u_{VM}(t)$ for all $t$ where $u_{VM}(t)$ is a solution of \eqref{eq:variationalMethod}. The relation of the corresponding subgradients is given by
\begin{align}
\label{eq:rhoFormula}
p_{VM}(t) = \frac{1}{t}\int_{0}^t p_{GF}(s) ds \in \partial J(u_*(t)),
\end{align}
for $* \in \{VM, GF\}$.
\end{enumerate}
\label{thm:flowVariationalEquivalence}
\end{theorem}
\begin{proof}
Based on Theorem \ref{thm:piecewiseDynamics} we know that there exist times $0<t_1<t_2<...$ in between which $u(t)$ behaves linearly. We proceed inductively. For $0\leq t \leq t_1$ we have
$$ u_{GF}(t) = u_{GF}(0) - t \; p_{GF}(t_1)= f - t \; p_{GF}(t_1),$$
with $p_{GF}(t_1) \in \partial J(u_{GF}(t))$ for all $t \in [0,t_1]$. The latter coincides with the optimality condition for $u_{GF}(t) = \argmin\limits_{u} \frac{1}{2}\|u-f\|_2^2 + t J(u)$, which shows $u_{GF}(t) = u_{VM}(t)$  for $t \in [0,t_1]$. Due to the closedness of subdifferentials, $p_{GF}(t_1) \in \partial J(u(t))$ for $t \in [0,t_1[$, implies that the latter holds for $t=t_1$ too. Thus, we can state that $p_{GF}(t_1) = p_{VM}(t) = \frac{1}{t}\int_0^t p_{GF}(t) ~dt \in \partial J(u_*(t))$  for $t \in [0,t_1]$, $* \in \{GF,\; VM\}$.

Assume the assertion holds for all $t\in [0, t_i]$, We will show that it holds for $t \in [0,t_{i+1}]$, too. Based on the proof of theorem \ref{thm:piecewiseDynamics}, we know that $u_{GF}(t) = u_{GF}(t_i) + (t-t_i) p_{GF}(t_{i+1})$ for $t \in [t_i,t_{i+1}]$, and $p_{GF}(t_{i+1}) = \argmin_{p \in \partial J(u(t_i))} \|p\|_2^2$. Now (MINSUB) implies $\langle p_{GF}(t_{i+1}), p_{GF}(t_{i+1}) - p_{GF}(t_{j}) \rangle = 0$ for all $j\leq i$. We compute
\begin{align*}
 \langle p_{GF}(t_{j}), u_{GF}(t) \rangle =&  \langle p_{GF}(t_{j}), u_{GF}(t_i) + (t-t_i) p_{GF}(t_{i+1}) \rangle \\
 =&  J( u_{GF}(t_i)) + (t-t_i) \langle p_{GF}(t_{j}), p_{GF}(t_{i+1}) \rangle \\
 =&   \langle p_{GF}(t_{i+1}),u(t_i) \rangle + (t-t_i) \langle p_{GF}(t_{i+1}), p_{GF}(t_{i+1}) \rangle \\
 =&   \langle p_{GF}(t_{i+1}),u(t_i) + (t-t_i)p_{GF}(t_{i+1}) \rangle \\
 =&   \langle p_{GF}(t_{i+1}), u(t) \rangle \\
 =&   J(u(t)).
\end{align*}
Therefore, $p_{GF}(t_{j}) \in \partial J(u(t))$ for all $t \in [t_i, t_{i+1}]$ and all $j\leq i+1$. Integrating the scale space flow equation yields for $t \leq t_{i+1}$
$$ 0 = u_{GF}(t) - f  + \int_0^t p_{GF}(t) ~dt= u_{GF}(t) - f  + t \left( \frac{1}{t}\int_0^t p_{GF}(t) ~dt\right).$$
Due to the convexity of the subdifferential, we find $\frac{1}{t}\int_0^t p_{GF}(t) ~dt  \in \partial J(u_{GF}(t))$, which shows that $u_{GF}(t)$ meets the optimality condition for the variational method and concludes the proof.
\end{proof}

The above theorem not only allows us to conclude that $\phi_{VM} = \phi_{GF}$, but also $\phi_{VM}^i = \phi_{GF}^i$ for all $i$ in the light of equation \eqref{eq:finiteDecomposition}. Additionally, the first aspect of the above theorem allows another very interesting conclusion, revealing another striking similarity to linear spectral decompositions.
\begin{theorem}[Orthogonal Decompositions]\label{thm:orthogonality}
Let $J$ be such that (PS) and (MINSUB) are satisfied. Then equation \eqref{eq:finiteDecomposition} is an orthogonal decomposition of the input data $f$ for the variational method as well as for the gradient flow, i.e.,
$$ \langle \phi^i_*, \phi^j_* \rangle = 0, \qquad \forall i\neq j, \ * \in \{ VM,GF\} $$
\end{theorem}
\begin{proof}
Due to the equivalence stated in Theorem \ref{thm:flowVariationalEquivalence} it is sufficient to consider the gradient flow only. For the gradient flow we have $\phi^i_{GF} = t_i (p_{GF}(t_{i}) - p_{GF}(t_{i+1}))$. Let $i>j$, then
\begin{align*}
\langle \phi^i_*, \phi^j_* \rangle =& t_i t_j \langle p_{GF}(t_{i+1}) - p_{GF}(t_{i}), p_{GF}(t_{j+1}) - p_{GF}(t_{j}) \rangle \\
=& t_i t_j \left(\langle p_{GF}(t_{i+1}), p_{GF}(t_{j+1}) - p_{GF}(t_{j}) \rangle  - \langle  p_{GF}(t_{i}), p_{GF}(t_{j+1}) - p_{GF}(t_{j}) \rangle \right).
\end{align*}
Now due to statement 1. in  Theorem \ref{thm:flowVariationalEquivalence}, we know that $p_{GF}(t_{j}) \in \partial J(u(t^i))$ and $p_{GF}(t_{j+1}) \in \partial J(u(t^i))$. Since $p_{GF}(t_{i})$ was determined according to equation \eqref{eq:minimizationGf}, we can use (MINSUB) to state that
\begin{align*}
\langle  p_{GF}(t_{i}), p_{GF}(t_{j+1}) \rangle = \|p_{GF}(t_{i})\|^2 =  \langle  p_{GF}(t_{i}),  p_{GF}(t_{j}) \rangle ,
\end{align*}
which means that $ \langle  p_{GF}(t_{i}), p_{GF}(t_{j+1}) - p_{GF}(t_{j}) \rangle = 0$. With exactly the same argument we find $\langle p_{GF}(t_{i+1}), p_{GF}(t_{j+1}) - p_{GF}(t_{j}) \rangle=0$ which yields the assertion.
\end{proof}
Due to the limited number of possibly orthogonal vectors in $\R^n$ we can conclude:
\begin{conclusion}\
For $J: \R^n \rightarrow \R$ meeting (PS) and (MINSUB) the number of possible times at which the piecewise behavior of the variational and scale space method changes is at most $n$.
\end{conclusion}

As announced before we also obtain the equivalence of definitions of the power spectrum:
\begin{proposition} \label{prop:equivalencespectra}
Let $J$ be such that (PS) and (MINSUB) are satisfied. Then for the spectral representations of the gradient flow and the variational method we have $S^2(t)=S_2^2(t)$ for almost every $t$ (and $S_2^2$ being defined via \eqref{eq:S2}).
\end{proposition}
\begin{proof}
Due to the equivalence of representations it suffices to prove the result for the gradient flow case.
We have due to the orthogonality of the $\phi_{GF}^i$
\begin{eqnarray*}
S^2(t) &=& \sum_{i=0}^ {N_*} (\phi_{GF}^i \cdot f) \delta (t-t_i) = \sum_{i=0}^ {N_*} \Vert \phi_{GF}^i \Vert^2 \delta (t-t_i) \\
  &=&  \sum_{i=0}^ {N_*}   t_i^2 (p_{GF}(t_{i}) - p_{GF}(t_{i+1}))\cdot (p_{GF}(t_{i}) - p_{GF}(t_{i+1}))\delta (t-t_i).
\end{eqnarray*}
With (MINSUB) we conclude
$$ p_{GF}(t_{i}) \cdot   p_{GF}(t_{i+1}) = \Vert p_{GF}(t_{i+1}) \Vert^2. $$
Inserting this relation we have
$$ S^2(t) = \sum_{i=0}^ {N_*}   t_i^2(\Vert p_{GF}(t_{i}) \Vert^2-\Vert p_{GF}(t_{i+1}) \Vert^2) \delta(t-t_i) = t^2 \frac{d}{dt} \Vert p(t) \Vert^2 = S_2^2(t). $$
\end{proof}

Due to the importance of (MINSUB) based on Theorems \ref{thm:flowVariationalEquivalence} and \ref{thm:orthogonality}, the next natural question is if we can give a class of regularization functionals that meet (MINSUB).

\begin{theorem}\
\label{thm:minsubHolds}
Let (DDL1) be met (Def. \ref{def:DDL1}). Then $J$ meets (PS) and (MINSUB).
\end{theorem}
\begin{proof}
First of all, for any $p\in \partial J(u)$ we know from the characterization of the subdifferential \eqref{eq:subdiffCharak} of the $\ell^1$ norm that $p=K^Tq$ with
\begin{align}
q(l) \left\{ \begin{array}{lc} = 1 & \text{ if } Ku(l) >0, \\ =-1 & \text{ if } Ku(l) <0, \\ \in [-1,1] & \text{ if } Ku(l) =0   .
\end{array} \right.
\end{align}
Thus for any given $\hat{u}$ and $\hat{p}$ defined as
\begin{align}
\label{eq:normMinimizing}
 \hat{p} = \arg \min_p \|p\|^2 \ \text{s.t. } p \in \partial J(\hat{u})
\end{align}
we have the optimality condition
\begin{align}
\label{eq:optiCond}
KK^T \hat{q} + \lambda = 0,
\end{align}
for some Lagrange multiplier $\lambda$ to enforce the constraints. For a given $u$ denote
\begin{align}
 I_u = \{l ~|~ Ku(l) \neq 0\}.
\end{align}
For better readability of the proof, let us state an additional lemma:
\begin{lemma}
\label{lem:lagrangeZero}
Let $J(u)=\|Ku\|_1$ for a linear operator $K$ and let $\hat{p}$ be defined by \eqref{eq:normMinimizing} for some arbitrary element $u$. If the $\lambda$ arising from \eqref{eq:optiCond} meets $\lambda(l)=0 \ \forall l \notin I_u$ then (MINSUB) holds.
\end{lemma}
\begin{proof}
Let $z \in \partial J(u)$ be some other subgradient. Based on \eqref{eq:subdiffCharak} there is a $q_z$ such that $z=K^T q_z$. Now
\begin{align*}
\langle \hat{p}, \hat{p} - z \rangle =& \langle K K^T \hat{q}, \hat{q} - q_z \rangle
=\sum_l (KK^T\hat{q})(l) \cdot (\hat{q}(l) - q_z(l)) \\
=& \sum_{l \in I}(KK^T\hat{q})(l) \cdot  (\hat{q}(l) - q_z(l)) + \sum_{l \notin I} (KK^T\hat{q})(l) \cdot  (\hat{q}(l) - q_z(l)) \\
=& \sum_{l \notin I} (KK^T\hat{q})(l) \cdot (\hat{q}(l) - q_z(l))
= - \sum_{l \notin I} \lambda(l) \cdot (\hat{q}(l) - q_z(l))
=0
\end{align*}
\end{proof}

Consider $\hat{p}$ defined by \eqref{eq:normMinimizing} for some arbitrary element $u$. According to Lemma \ref{lem:lagrangeZero} it is sufficient to show that the Lagrange multiplier $\lambda$ in \eqref{eq:optiCond} meets $\lambda(l)=0 \ \forall l \notin I_u$. Assume $\lambda(l) > 0$ for some $l \notin I$. The complementary slackness condition then tells us that $\hat{q}(l) = 1$. Therefore
\begin{align*}
-\lambda(l) =& (K^TK \hat{q})(l),
 = \sum_j (KK^T)(l,j) \hat{q}(j), \\
=& (KK^T)(l,l)  \hat{q}(l) + \sum_{j\neq l } (KK^T)(l,j) \hat{q}(j),
= (KK^T)(l,l)  + \sum_{j\neq l } (KK^T)(l,j)\hat{q}(j) \\
\geq& (KK^T)(l,l)  - \sum_{j\neq l } |(KK^T)(l,j)| \stackrel{KK^T \text{ diag. dom.}}{\geq} 0 ,
\end{align*}
which is a contradiction to $\lambda(l)>0$. A similar computation shows that $\lambda(l)<0$ (which implies $\hat{q}(l)=-1$) is not possible either.
\end{proof}

By  Theorem \ref{thm:flowVariationalEquivalence} the above result implies that all (DDL1) regularizations lead to the equivalence of the spectral representations obtained by the variational and the scale space method. Interestingly, the class of (DDL1) functionals also allows to show the equivalence of the third possible definition of a spectral representation.
\begin{theorem}[Equivalence of $GF$, $VM$, and $IF$]
Let (DDL1) be met. Denote $\tau = \frac{1}{t}$, $v(\tau) = u_{GF}(1/\tau) = u_{GF}(t)$, and $r(\tau) = p_{VM}(1/\tau) = p_{VM}(t)$. It holds that
\begin{align}
\partial_\tau r(\tau ) =& f - \partial_{\tau} \left(\tau \: v(\tau)\right), \ \ r(\tau) \in \partial J\left( \partial_{\tau} \left(\tau \: v(\tau)\right)\right),
\end{align}
in other words, $( r(\tau),  \partial_{\tau} \left(\tau \: v(\tau)\right))$ solve the inverse scale space flow \eqref{eq:inverseScaleSpace}.
\label{thm:issRofSubgradienceEquivalence}
\end{theorem}
\begin{proof}
First of all note that
$$ \tau(u_{VM}(\tau) - f) + p_{VM}(\tau) = 0$$
holds as well as $u_{VM}(\tau)=u_{GF}(\tau)$ based on Theorem \ref{thm:flowVariationalEquivalence}. Differentiating the above equality yields
$$ \partial_\tau r(\tau) = f - \partial_\tau(\tau v(\tau)).$$
We still need to show the subgradient inclusion. It holds due to the piecewise linearity of the flow that
\begin{align}
\partial_{\tau}  \left(\tau \: v(\tau)\right) =& v(\tau) + \tau \: \partial_{\tau} v(\tau)
= u_{GF}(t) + \frac{1}{t} \partial_{\tau} u_{GF}(t(\tau))  \nonumber \\
=& u_{GF}(t) - t \partial_{t} u_{GF}(t)
= u_{GF}(t) + t p_{GF}(t)  \nonumber \\
=& u_{GF}(t_i) - (t-t_i)p_{GF}(t_{i+1}) + t p_{GF}(t_{i+1})  \nonumber \\
=& u_{GF}(t_i) + t_i p_{GF}(t_{i+1}).
\end{align}
Thus, we can continue computing
\begin{align}
\langle r(\tau) , \partial_{\tau} \left(\tau \: v(\tau)\right) \rangle =& \langle p_{VM}(t) ,  u_i + t_i p_{GF}(t_{i+1}) \rangle, \nonumber \\
=& J(u_i) + t_i \langle p_{VM}(t),p_{GF}(t_{i+1}) \rangle.
\label{eq:helper1}
\end{align}
 Due to \eqref{eq:rhoFormula} and the piecewise constant $p_{GF}(t)$ we have
$$  p_{VM}(t) = \frac{1}{t} \left(\sum_{j=1}^{i-1} (t_{j+1}-t_j) p_{GF}(t_{j+1}) + (t-t_i) p_{GF}(t_{i+1})\right). $$
Using the above formula for $p_{VM}(t)$ we can use the (MINSUB) condition to obtain
\begin{align*}
 t_i \langle p_{VM}(t), p_{GF}(t_{i+1}) \rangle
 =& \frac{t_i}{t} \left \langle \left(\sum_{j=1}^{i-1} (t_{j+1}-t_j) p_{GF}(t_{j+1}) + (t-t_i) p_{GF}(t_{i+1})\right), p_{GF}(t_{i+1}) \right\rangle,\\
 =&  \frac{t_i}{t}\left(\sum_{j=1}^{i-1} (t_{j+1}-t_j) \langle p_{GF}(t_{j+1}), p_{GF}(t_{i+1}) \rangle + (t-t_i) \|p_{GF}(t_{i+1})\|^2 \right), \\
  \stackrel{\text{(MINSUB)}}{=}& t_i \|p_{GF}(t_{i+1})\|^2   \stackrel{\text{(Theorem \ref{thm:eigenfunctions})}}{=} t_i J(p_{GF}(t_{i+1})) .
\end{align*}
By combining the above estimate with \eqref{eq:helper1} we obtain
\begin{align*}
 J(\partial_{\tau} \left(\tau \: v(\tau)\right)) =&  J( u(t_i) + t_i p_{GF}(t_{i+1})),  
{\leq} J(u(t_i)) + t_i J(p_{GF}(t_{i+1})),\\
  =& \langle r(\tau) , \partial_{\tau} \left(\tau \: v(\tau)\right) \rangle ,
\end{align*}
which yields $r(\tau) \in \partial J(\partial_\tau (\tau v(\tau)))$ and hence the assertion.
\end{proof}

\begin{conclusion}\label{conclu:sameSpectralRepresentation}
Let (DDL1) be met. Then $GF$, $VM$, and $IF$ all yield the same spectral representation.
\end{conclusion}
\begin{proof}
Theorem \ref{thm:minsubHolds} along with Theorem \ref{thm:flowVariationalEquivalence} show that $u_{GF}(t) = u_{VM}(t)$, which implies $\phi_{VM}(t)=\phi_{GF}(t)$. Theorem \ref{thm:issRofSubgradienceEquivalence} tells us that
\begin{align*}
v_{IS}(s) &=  \partial_{s} \left(s \: u_{GF}(1/s)\right) \\
&=  u_{GF}(1/s) - \frac{1}{s}\partial_t u_{GF}(1/s).
\end{align*}
Thus,
\begin{align*}
\tilde \phi_{IS}(s) &= \partial_s v_{IS}(s) = -\frac{1}{s^2 } \partial_t u_{GF}(1/s) - \left( -\frac{1}{s^2} \partial_t u_{GF}(1/s) - \frac{1}{s^3} \partial_{tt} u_{GF}(1/s)\right).\\
&=\frac{1}{s^3} \partial_{tt} u_{GF}(1/s)
\end{align*}
The relation $\phi_{IS}(t)=\frac{1}{t^2}\psi_{IS}(1/t)$ of \eqref{eq:representationConversion} now yields
$$ \phi_{IS}(t) = t \partial_{tt} u_{GF}(t) = \phi_{GF}(t).$$
\end{proof}

\subsection{Nonlinear Eigendecompositions}
As described in the introduction, the eigendecomposition, or, more generally, the singular value decomposition of a linear operator plays a crucial role in the classical filter analysis (cf. figure \ref{fig:classical}). Furthermore, we discussed that the notion of eigenvectors has been generalized to an element $v_\lambda$ with $\|v_\lambda\|_2=1$ such there exists a $\lambda$ with
$$ \lambda v_\lambda \in \partial J(v_\lambda). $$
The classical notion of singular vectors (up to a square root of $\lambda$) is recovered for a quadratic regularization functional $J$, i.e. $J(u) = \frac{1}{2}\|Ku\|_2^2$, in which case $\partial J(u) =\{K^*Ku\}$. We are particularly interested in the question in which case our generalized notion of a spectral decomposition admits the classical interpretation of filtering the coefficients of a (nonlinear) eigendecomposition of the input data $f$.

It is interesting to note that the use of eigendecompositions of linear operators goes beyond purely linear filtering: Popular nonlinear versions of the classical linear filters \eqref{eq:linearFilter} can be obtained by choosing filters adaptively to the magnitude of the coefficients in a new representation. For example, let $V$ be an orthonormal matrix (corresponding to a change of basis). For input data $f$ one defines $u_{\text{filtered}} = V \; D_{V^Tf} \; V^T f$, where $D_{V^Tf}$ is a data-dependent diagonal matrix, i.e. $\text{diag}(D_{V^Tf}) = g(V^Tf)$ for an appropriate function $g$. Examples for such choices include hard or soft thresholding of the coefficients. In \cite{spec_one_homog} we have shown that these types of soft- and hard thresholdings of representation coefficients can be recovered in our framework by choosing $J(u)=\|V^Tu\|_1$.

The matrix $V$ arising from the eigendecomposition of a linear operator is orthogonal, i.e. $VV^T=V^TV = Id$. The following theorem shows that significantly less restrictive conditions on the regularization, namely $J$ meeting (DDL1), already guarantee the decomposition of $f$ into a linear combination of generalized eigenvectors.

\begin{theorem}[Decomposition into eigenfunctions]
\label{thm:eigenfunctions}
Let (DDL1) be met. Then (up to normalization) the subgradients $p_{GF}(t_{i+1})$ of the gradient flow are eigenvectors of $J$, i.e. $p_{GF}(t_{i+1})  \in \partial J(p_{GF}(t_{i+1}))$. Hence,
\begin{align}
\label{eq:decomp_p}
f = P_0( f) + \sum_{i=0}^{N} (t_{i+1}-t_i) p_{GF}(t_{i+1}),
\end{align}
for $N$ being the index with $u(t_N)=P_0(f)$ is a decomposition of $f$ into eigenvectors of $J$.
\end{theorem}
\begin{proof}
For the sake of simplicity, denote $p_{i}:=p_{GF}(t_{i})$ and $u_{i}:=u_{GF}(t_{i})$. We already know that $p_{i+1}=K^T q_{i+1}$ for some $q_{i+1}$, $\|q_{i+1}\|_{\infty}\leq 1$, with $q_{i+1}(l) = \text{sign}(Ku_{i})(l)$ for $l \in I_{u_{i}}$.
\begin{align*}
\langle p_{i+1}, p_{i+1} \rangle =& \langle KK^T q_{i+1}, q_{i+1} \rangle, \\
=& \sum_{l \notin I_{u_{i+1}}} \underbrace{(KK^T q_{i+1})(l)}_{=0, \text{ Proof of \ref{thm:minsubHolds}}} \cdot q_{i+1}(l) + \sum_{l \in I_{u_{i+1}}} (KK^T q_{i+1})(l) \cdot q_{i+1}(l).
\end{align*}
For the second sum, we have
\begin{align*}
(KK^T q_{i+1})(l) =& (KK^T)(l,l)\cdot \text{sign}(Ku_{i}(l)) + \sum_{j\neq i} (KK^T)(l,j) \cdot q_{i+1}(j), \\
\stackrel{\text{Diag. dom. of $KK^T$}}{\Rightarrow} \text{sign}((KK^T q_{i+1})(l)) =& \text{sign}(Ku_{i}(l)) \ \text{ or } (KK^T q_{i+1})(l)=0.
\end{align*}
Thus, in any case, $l \in I_{u_{i}}$ and  $l \notin I_{u_{i}}$, we have $(KK^T q_{i+1})(l) \cdot q_{i+1}(l) = |(KK^T q_{i+1})(l)|$, such that
\begin{align*}
\langle p_{i+1}, p_{i+1} \rangle =& \langle KK^T q_{i+1}, q_{i+1} \rangle
= \sum_l |(K \underbrace{K^T q_{i+1}}_{=p_{i+1}})(l)|
= \|Kp_{i+1}\|_1,
\end{align*}
which completes the proof.
\end{proof}

Note that conclusion \ref{conclu:sameSpectralRepresentation} immediately generalizes this result from the spectral representation arising from the gradient flow to all three possible spectral representations.

It is interesting to see that the subgradients of the gradient flow are (up to normalization) the eigenvectors $f$ is decomposed into. By the definition of $\phi_{GF}(t) = t\partial_{tt}u_{GF}(t) = -t \partial_t p_{GF}(t)$ and the mathematical definition of $\phi_{GF}(t)$ removing the derivative in front of $p_{GF}(t)$, we can see that any filtering approach in the (DDL1) case indeed simply modifies the coefficients of the representation in Theorem \ref{thm:eigenfunctions}, and therefore establishes a full analogy to linear filtering approaches.

Since any $f$ can be represented as a linear combination of generalized eigenvectors, a conclusion of Theorem \ref{thm:eigenfunctions} is the existence of a basis of eigenfunctions for any regularization meeting (DDL1). Note that there can (and in general will), however, be many more eigenfunctions than dimensions of the space, such that the nonlinear spectral decomposition methods cannot be written in a classical setting.

%

Let us now compare the two representations of $f$ given by Eq. \eqref{eq:decomp_p} and by our usual reconstruction Eq. \eqref{eq:recon_gf}, or -- since (DDL1) implies we have a polyhedral regularization -- the discrete reconstruction given by Eq. \eqref{eq:finiteDecomposition}.
For simplicity we assume $P_0(f)=0$. While one can see that rearranging Eq. \eqref{eq:finiteDecomposition} yields  Eq. \eqref{eq:decomp_p}, our decomposition of an image into its $\phi$ parts correspond to the \emph{change} of the eigenfunctions during the piecewise dynamics. While eigenfunctions of absolutely one-homogeneous regularizations are often highly correlated, their differences can be orthogonal as stated in Theorem \ref{thm:orthogonality} and therefore nicely separate different scales of the input data. Fig. \ref{fig:decomp} shows an example of a spectral decomposition of a piecewise constant input signal with respect to the total variation including the subgradients $p_{GF}(t^i)$ and the corresponding $\phi^i$.

\begin{figure}
\begin{center}
\begin{tabular}{cc}
\includegraphics[height=20mm]{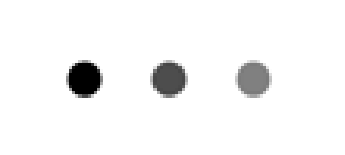}&
\includegraphics[height=30mm]{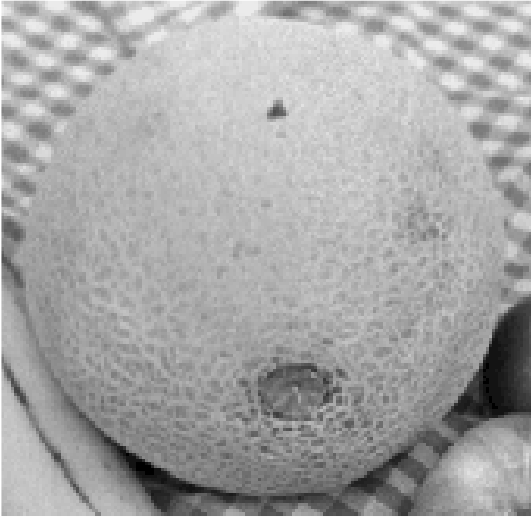}\\
\includegraphics[width=0.45\textwidth]{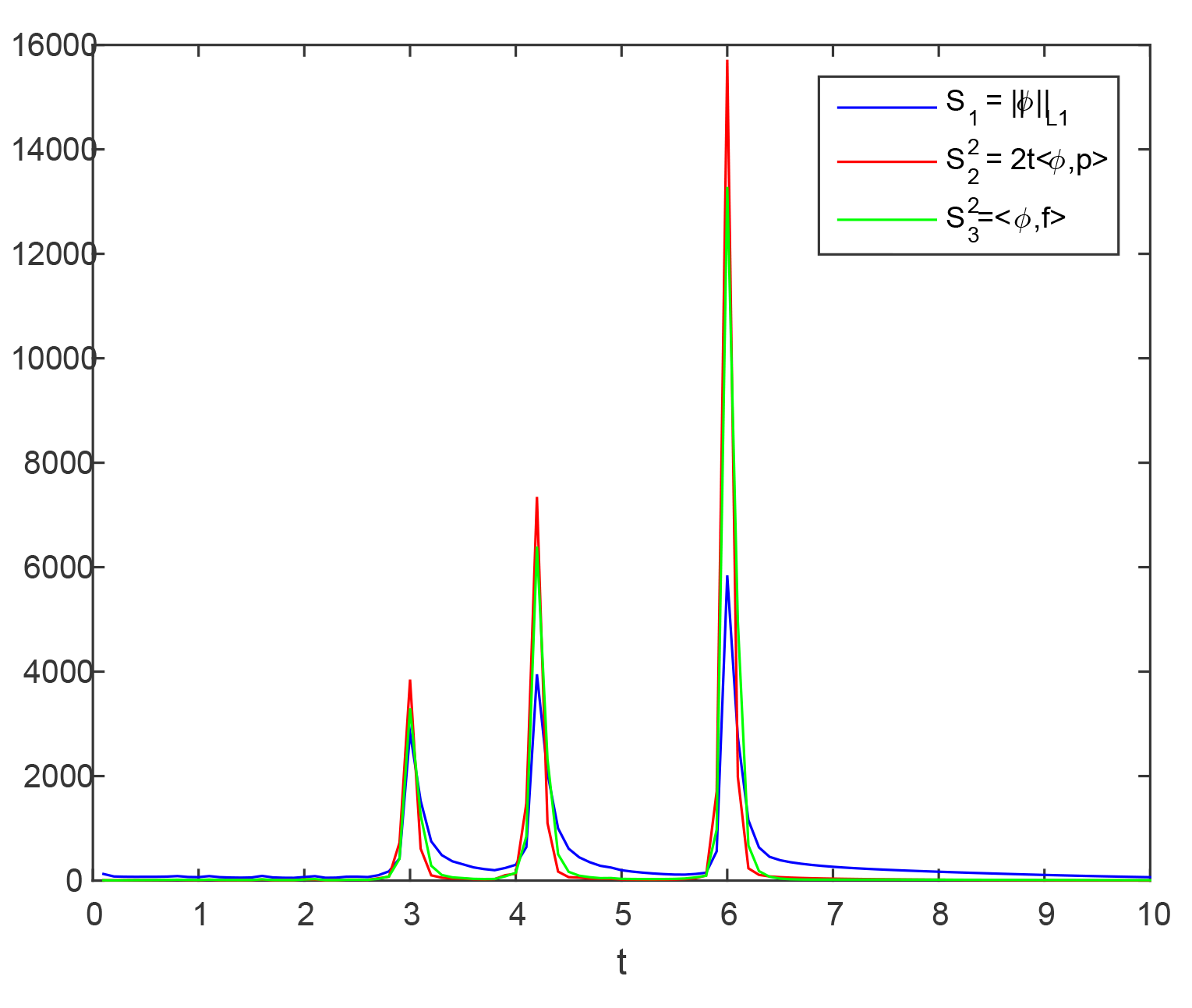}&
\includegraphics[width=0.45\textwidth]{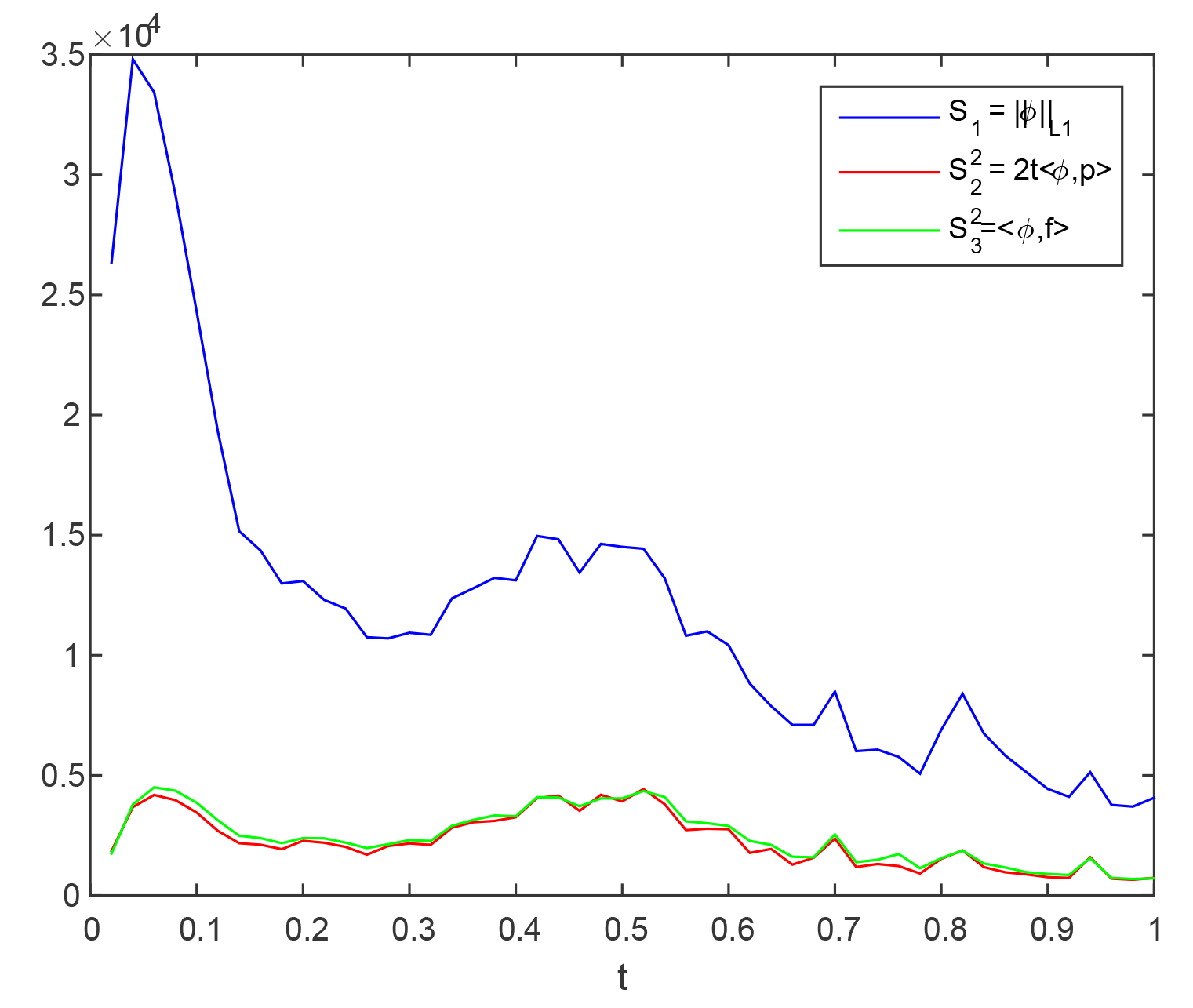}
\end{tabular}
\caption{Comparison of the response of the 3 spectral definitions $S_1$, $S_2$, $S_3$.
Left - response of 3 disks of different contrast, right - response for a natural image.}
\label{fig:spectrum}
\end{center}
\end{figure}

\begin{figure}
\begin{center}
	\tabcolsep0.3mm
	\begin{tabular}{ccc} 	
    $f$ & $S_3^2(t)$ & $S_1(t)$ \\
	\includegraphics[width=0.25\textwidth]{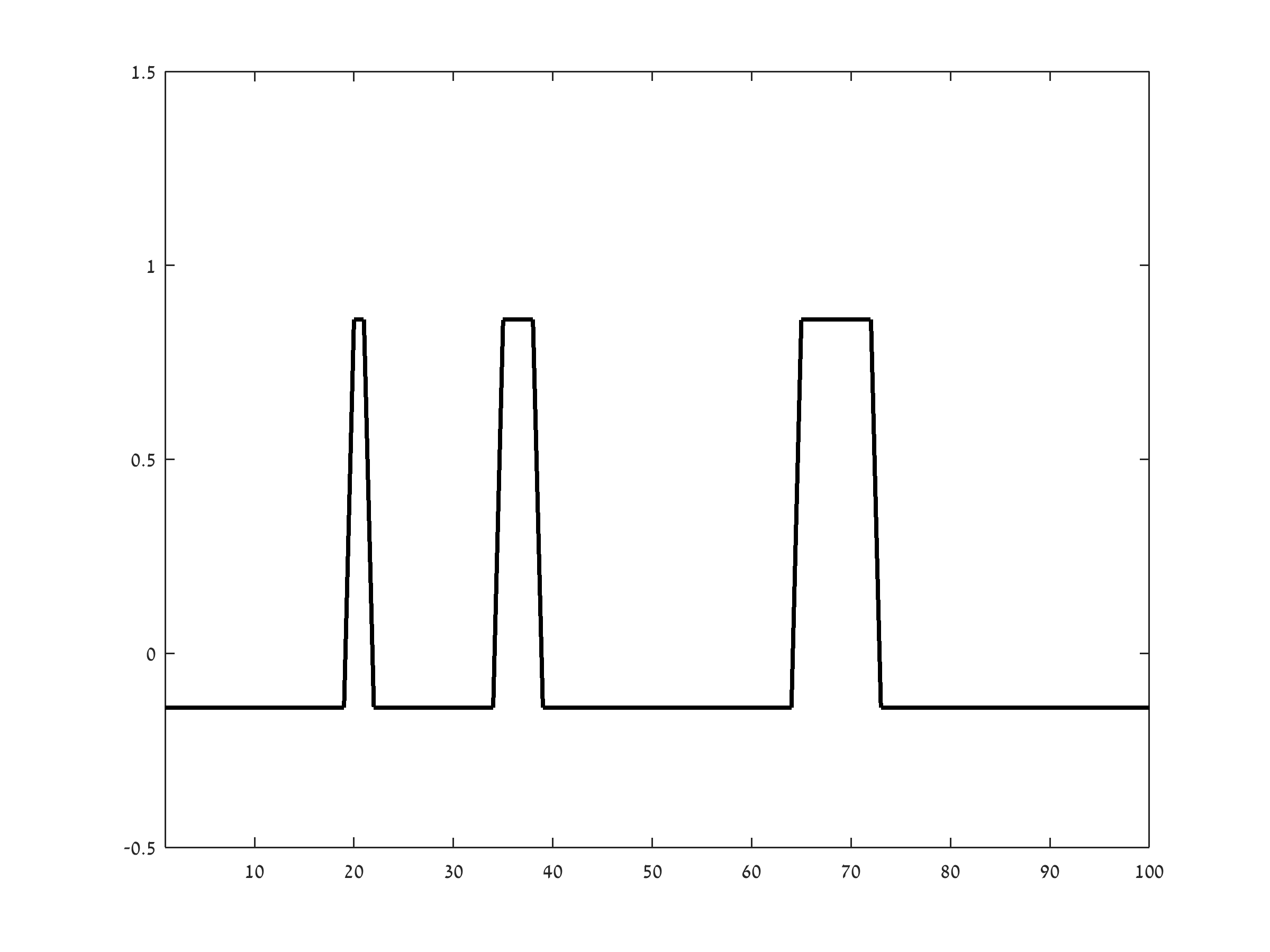} &
	\includegraphics[width=0.25\textwidth]{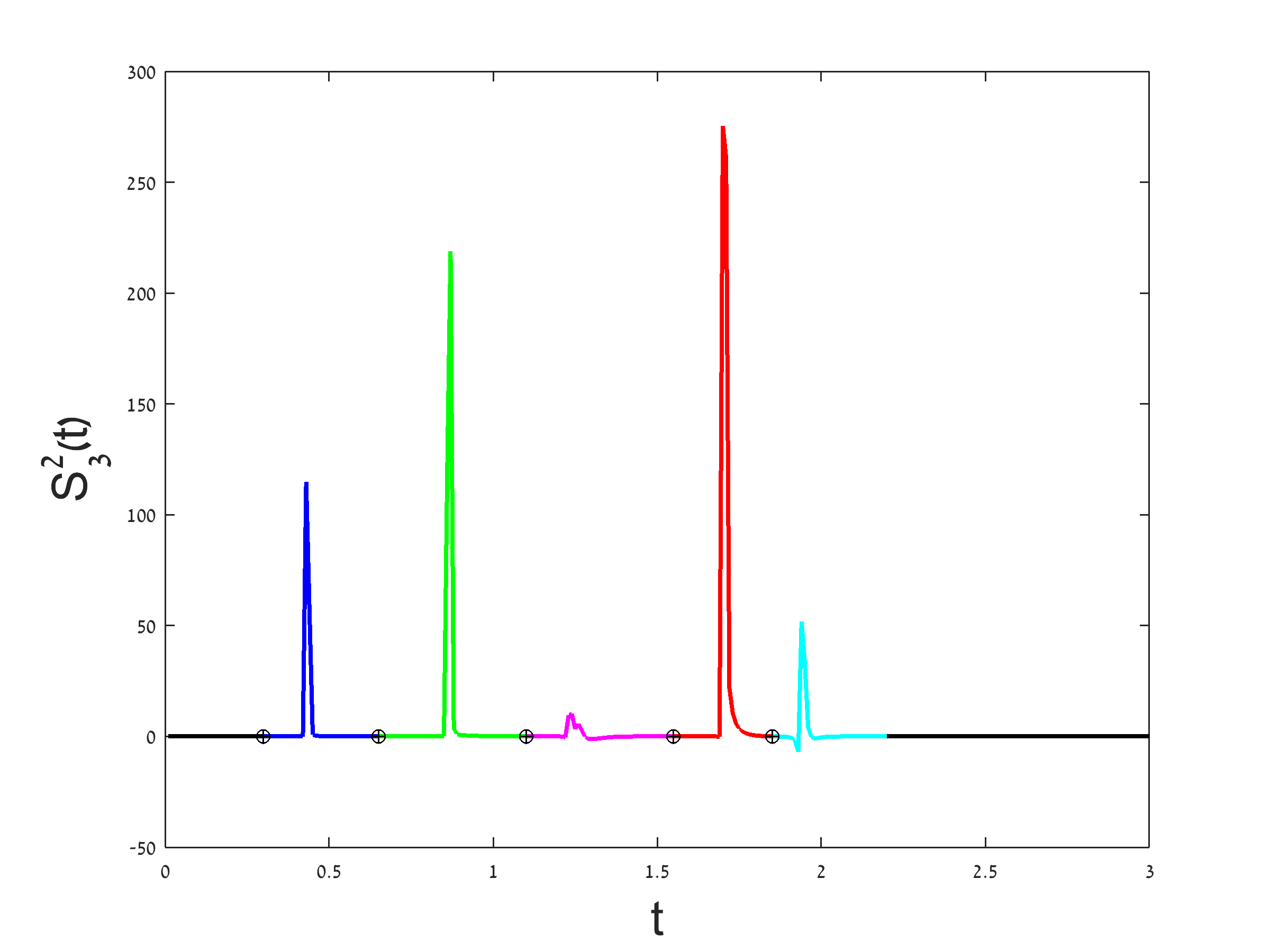} &
	\includegraphics[width=0.25\textwidth]{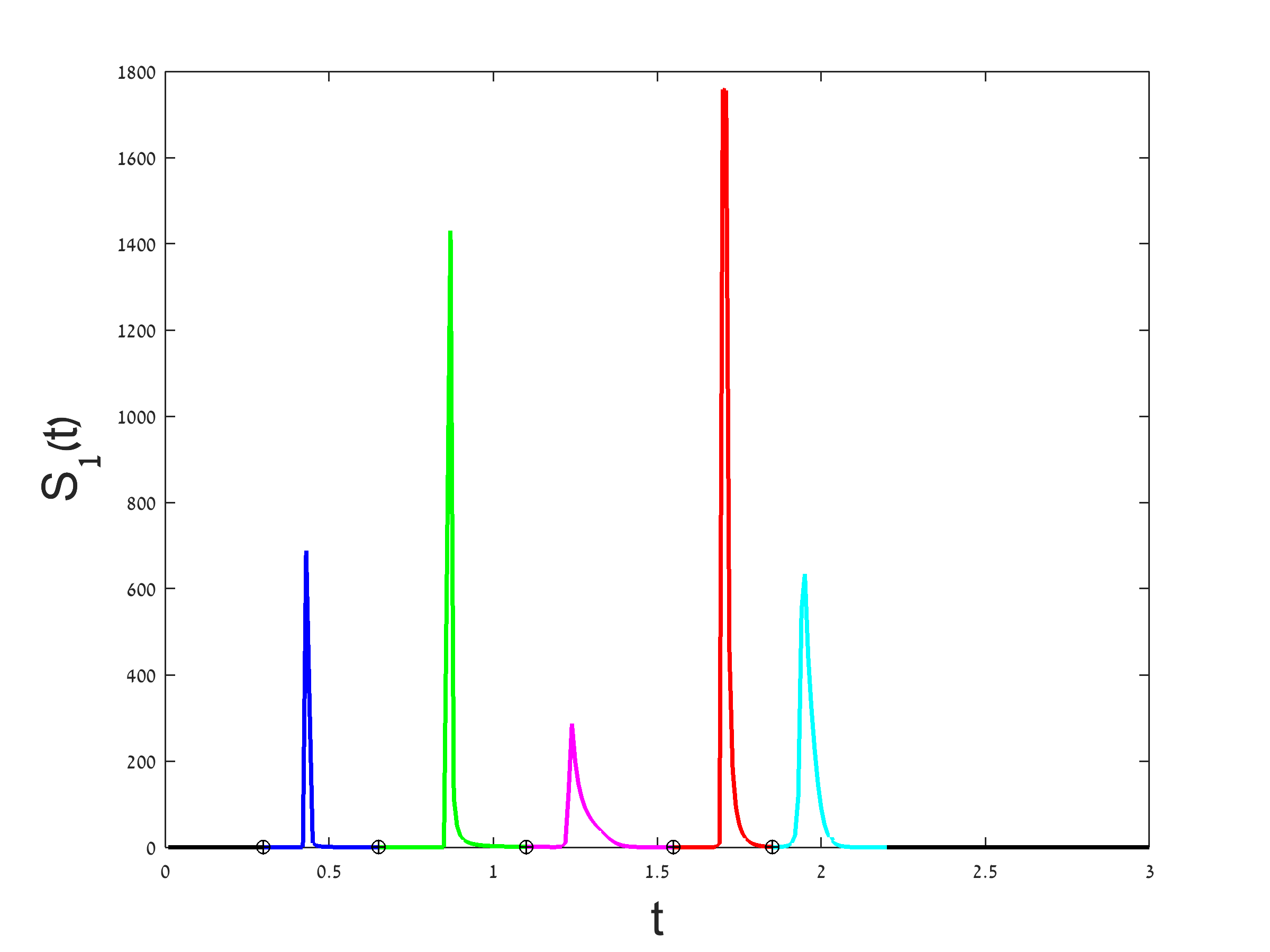} \\
    $u_{GF}(t_i)$ & $p_{GF}(t_i)$ & $\Phi(t_i)$ \\
	\includegraphics[width=0.25\textwidth]{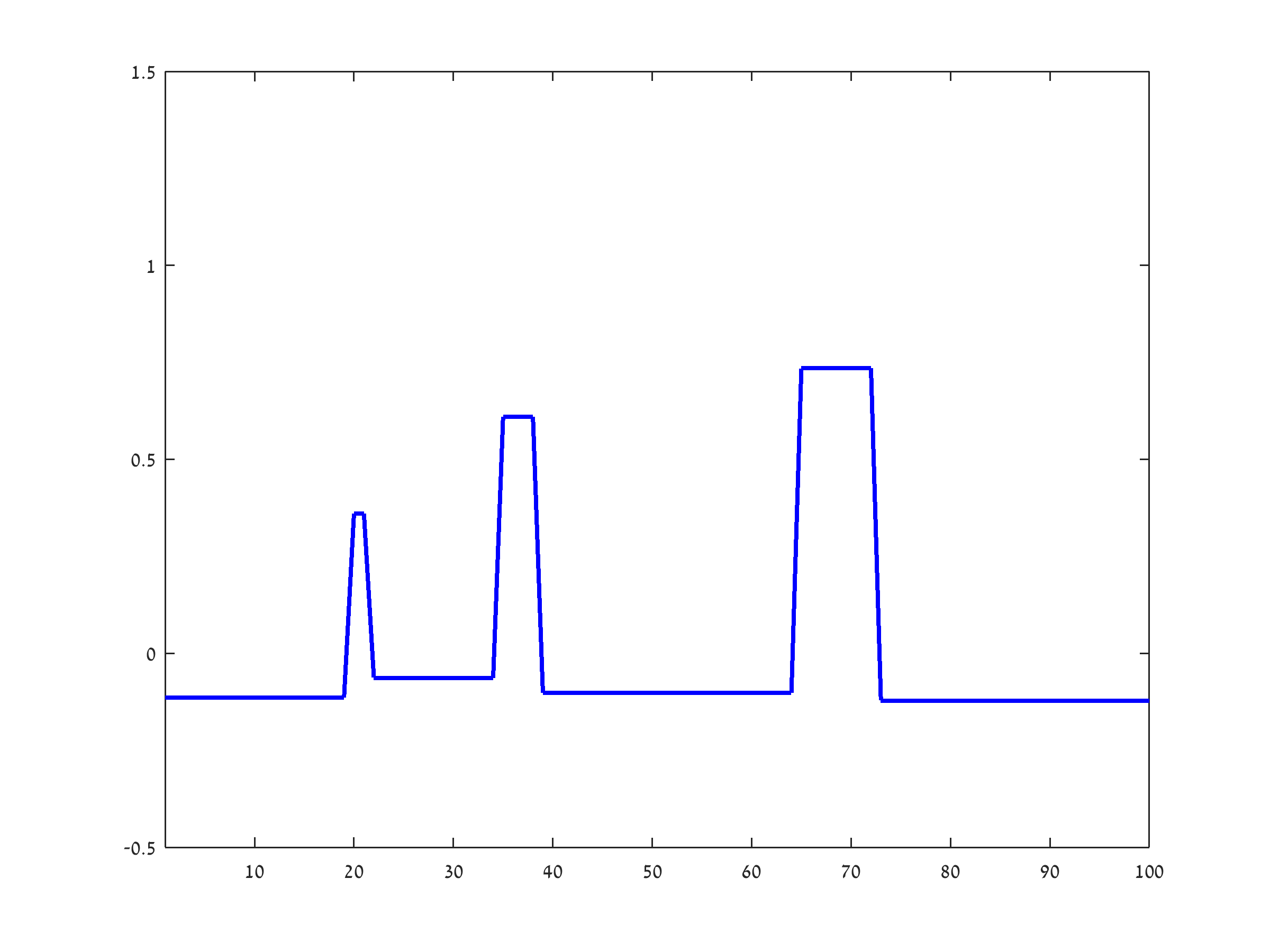} &
    \includegraphics[width=0.25\textwidth]{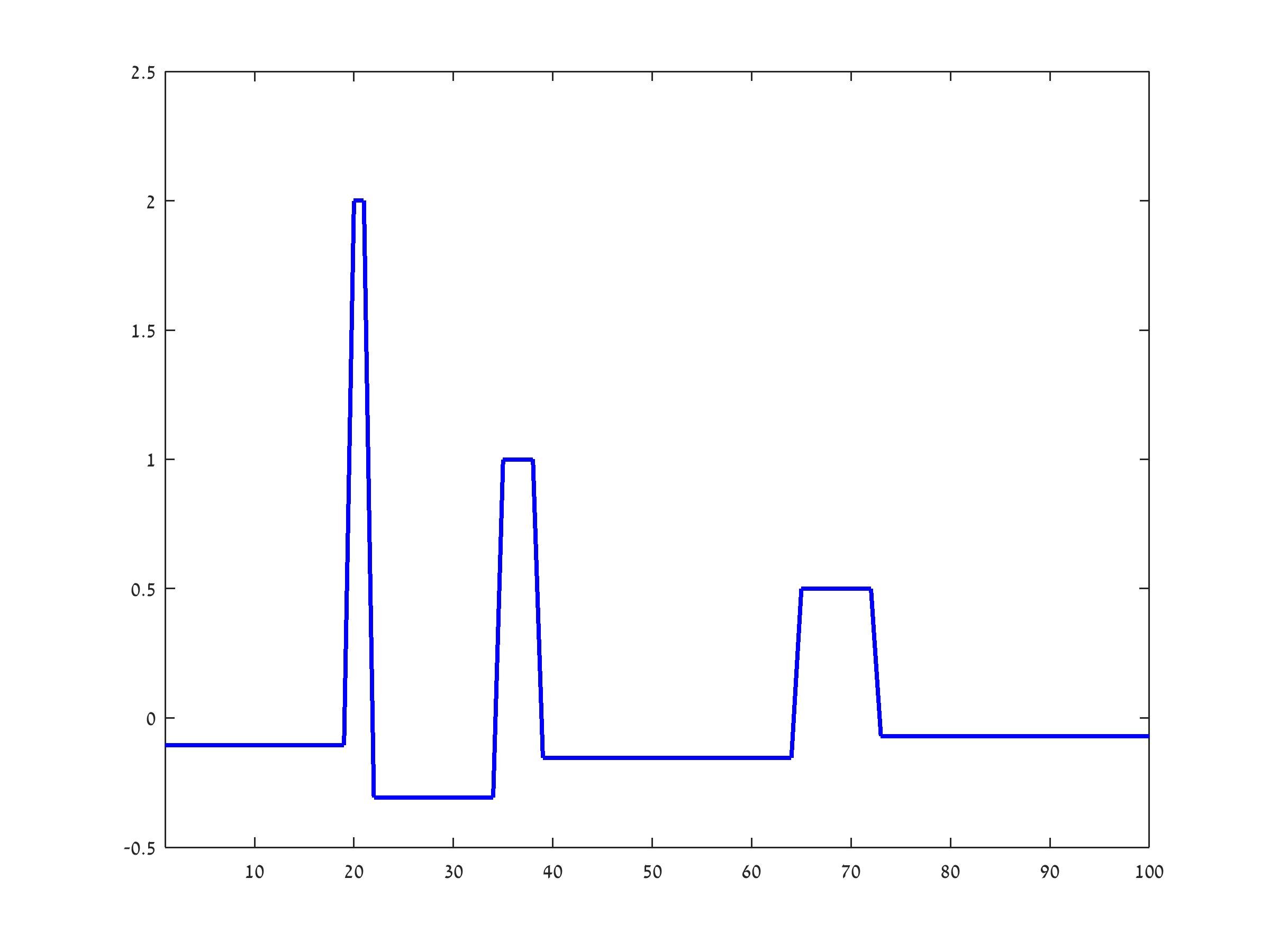} &
    \includegraphics[width=0.25\textwidth]{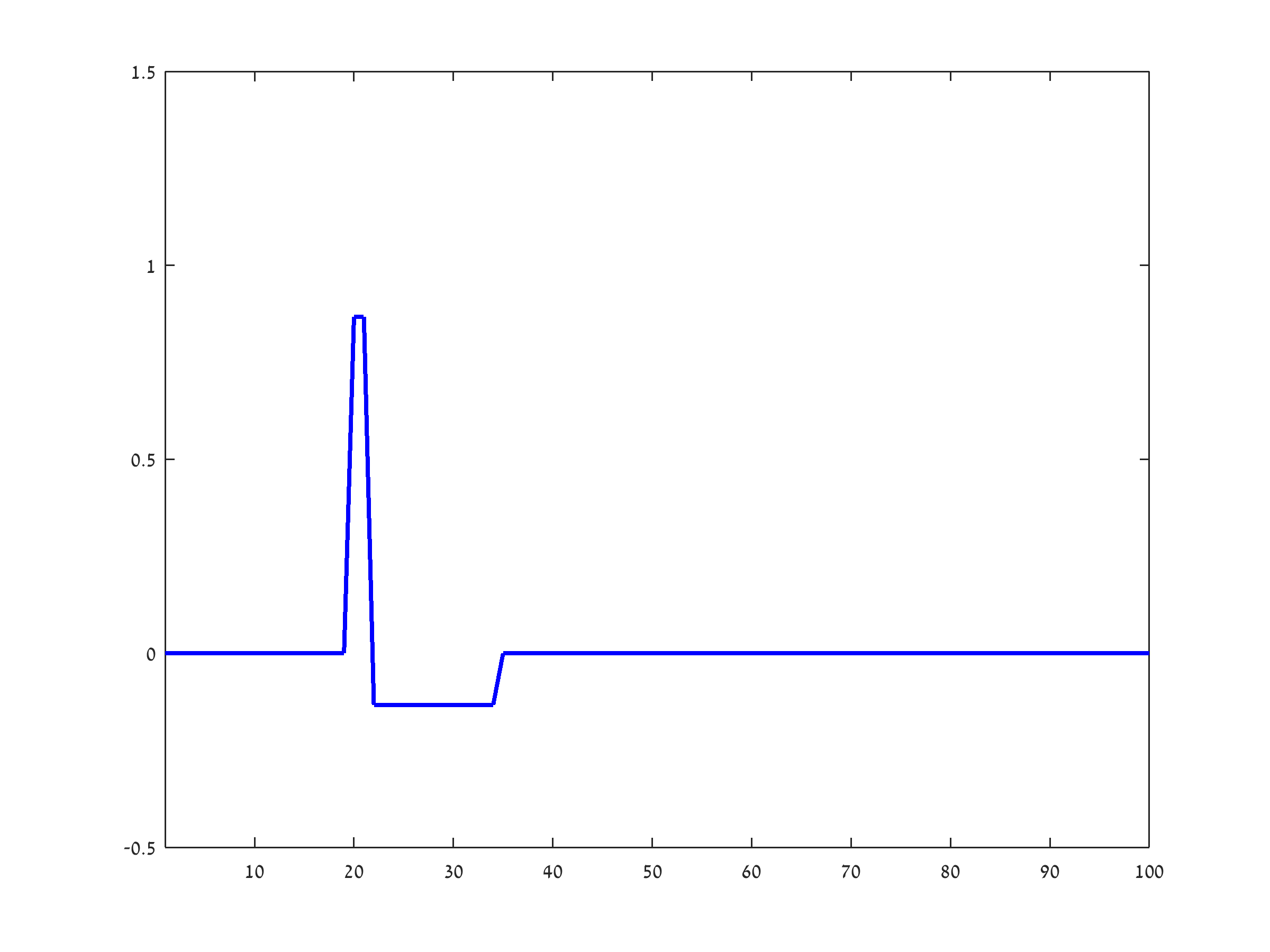} \\
    \includegraphics[width=0.25\textwidth]{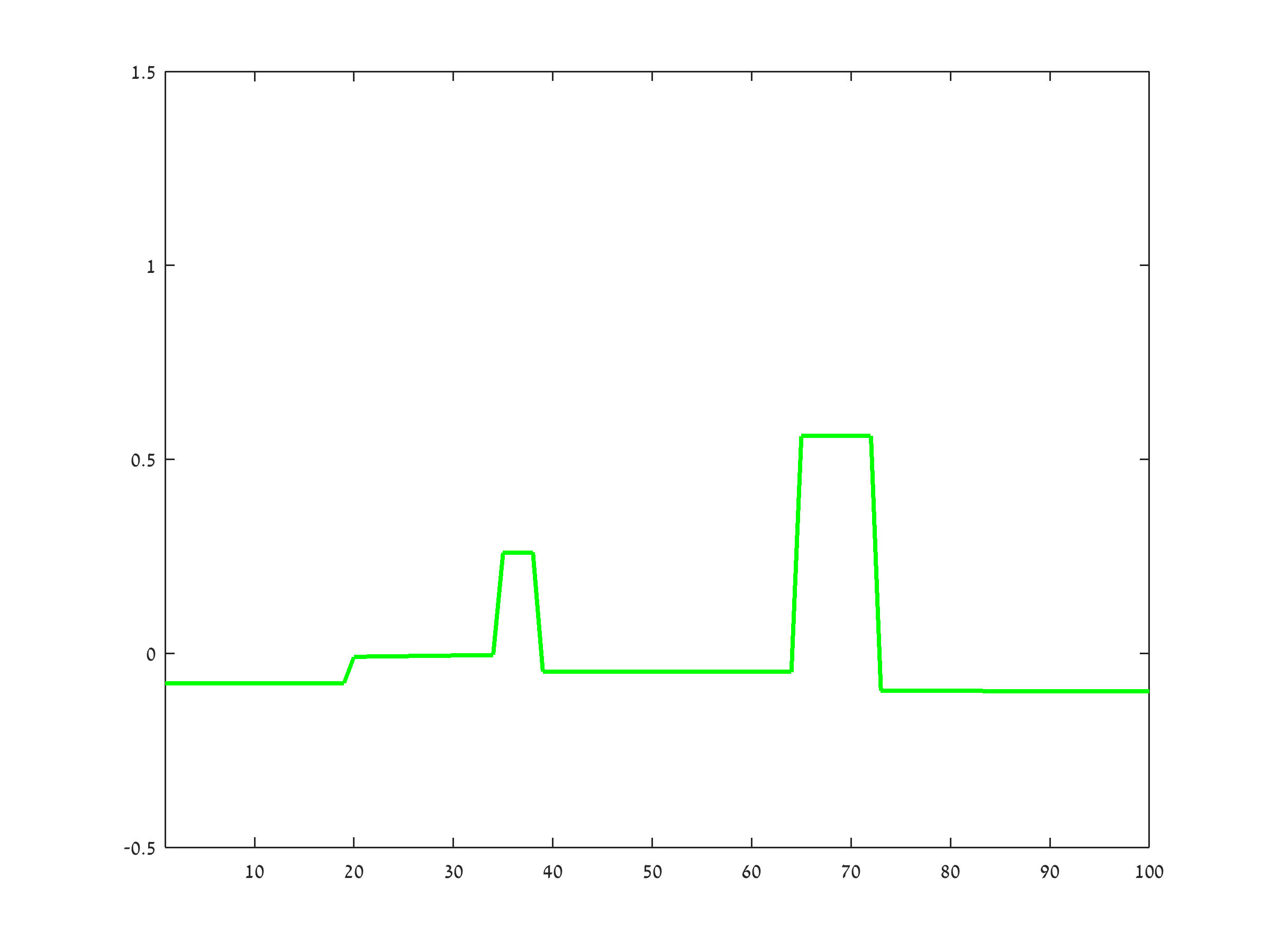} &
    \includegraphics[width=0.25\textwidth]{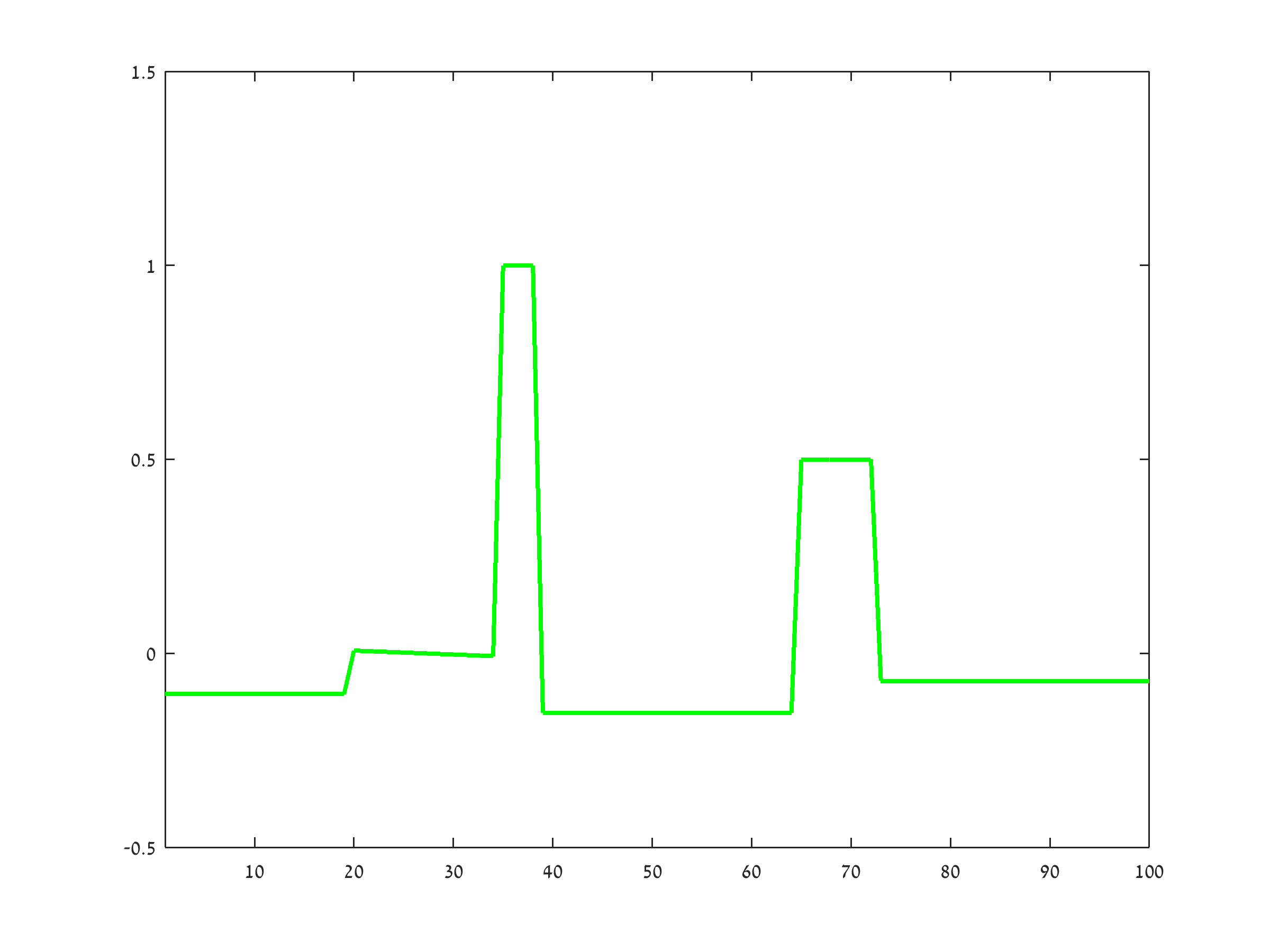} &
    \includegraphics[width=0.25\textwidth]{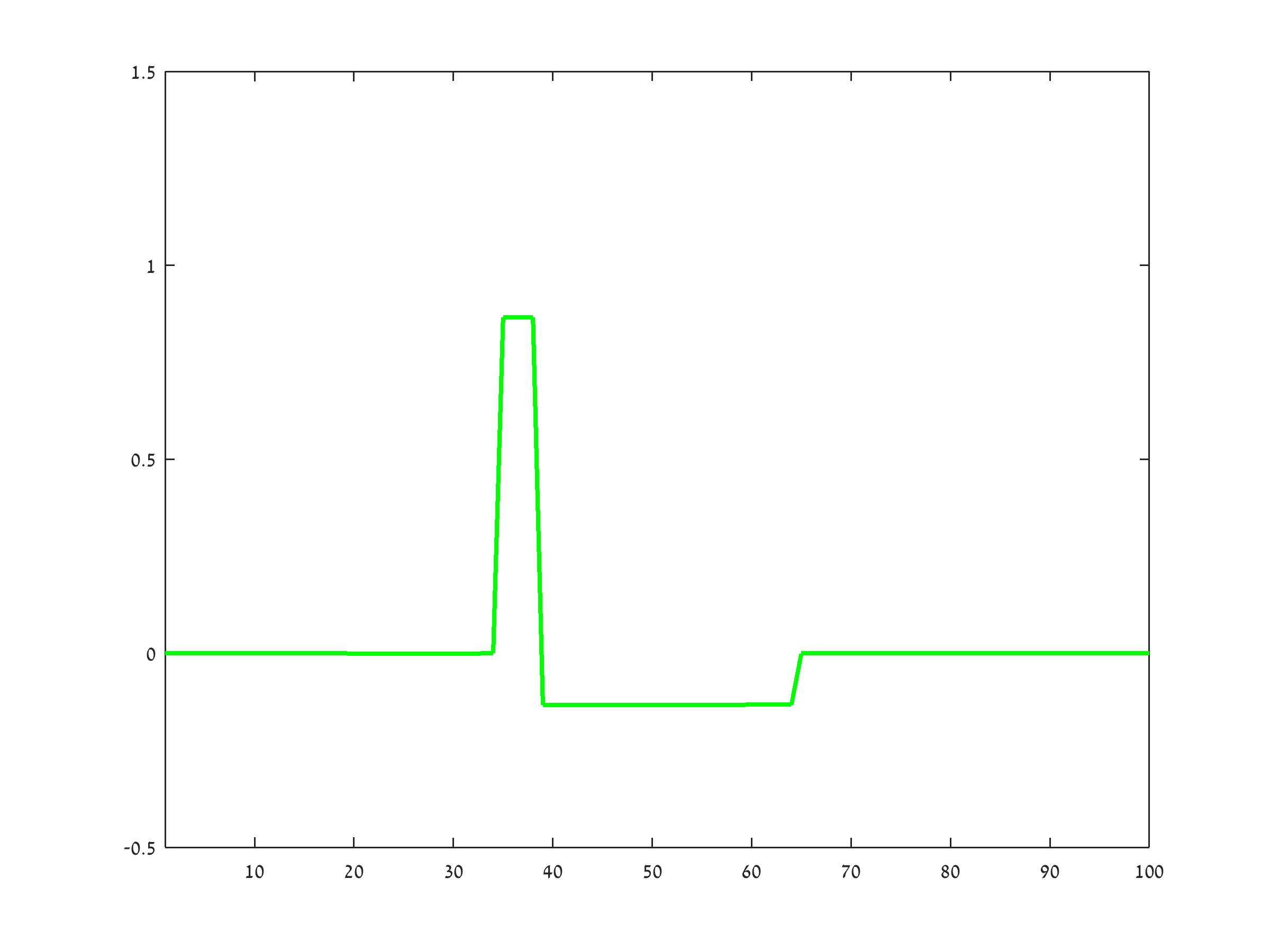} \\
    \includegraphics[width=0.25\textwidth]{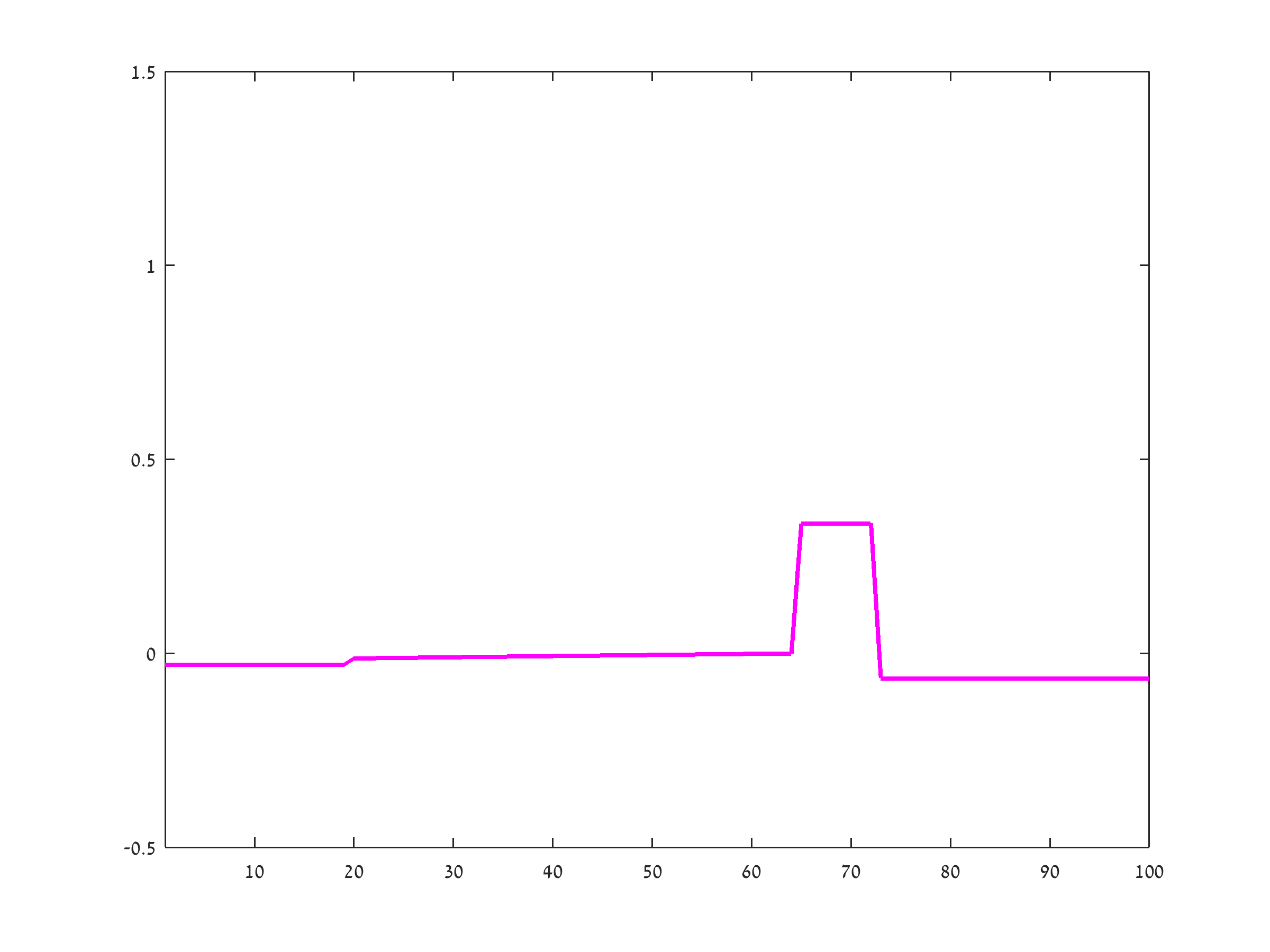} &
    \includegraphics[width=0.25\textwidth]{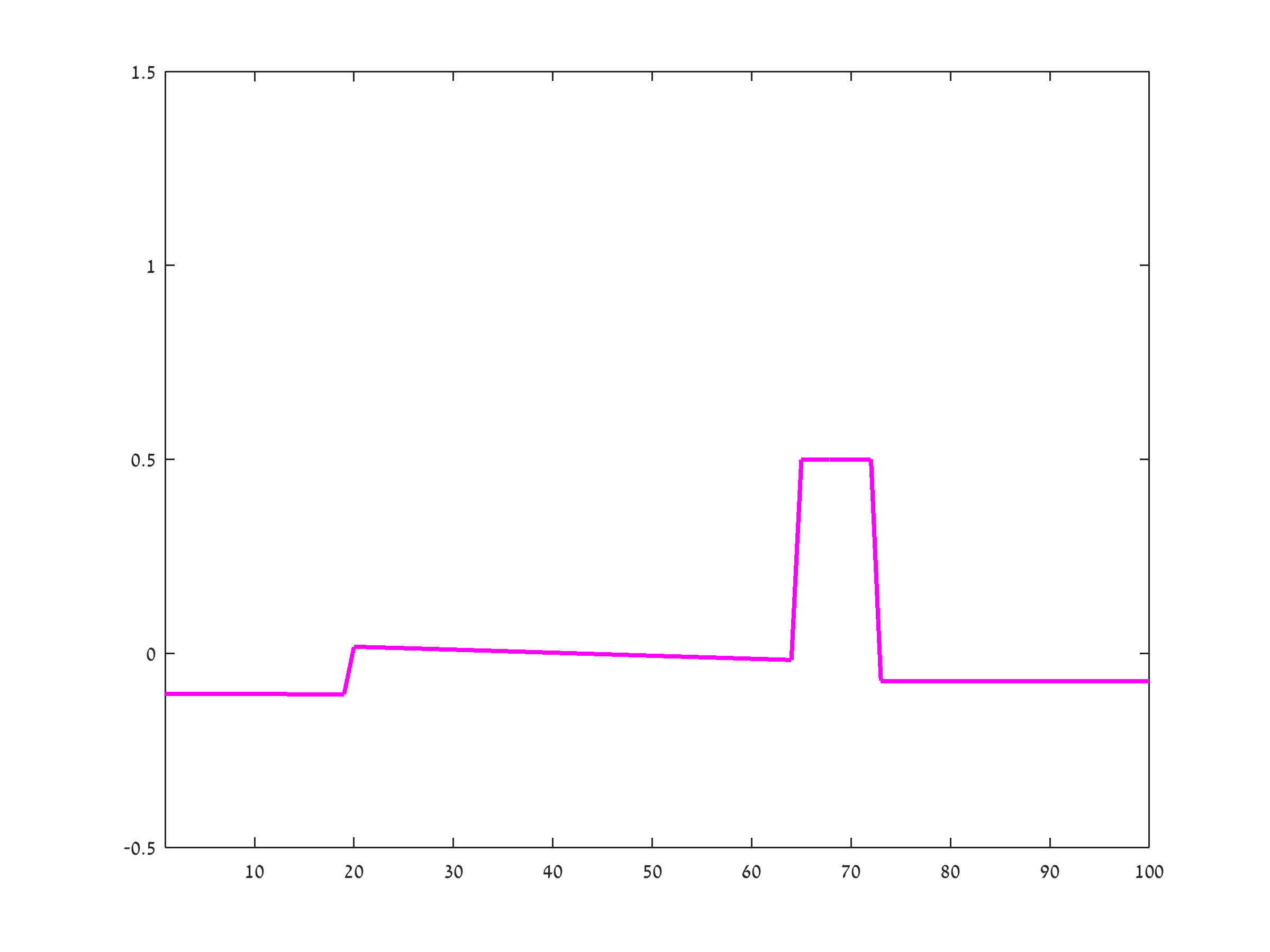} &
    \includegraphics[width=0.25\textwidth]{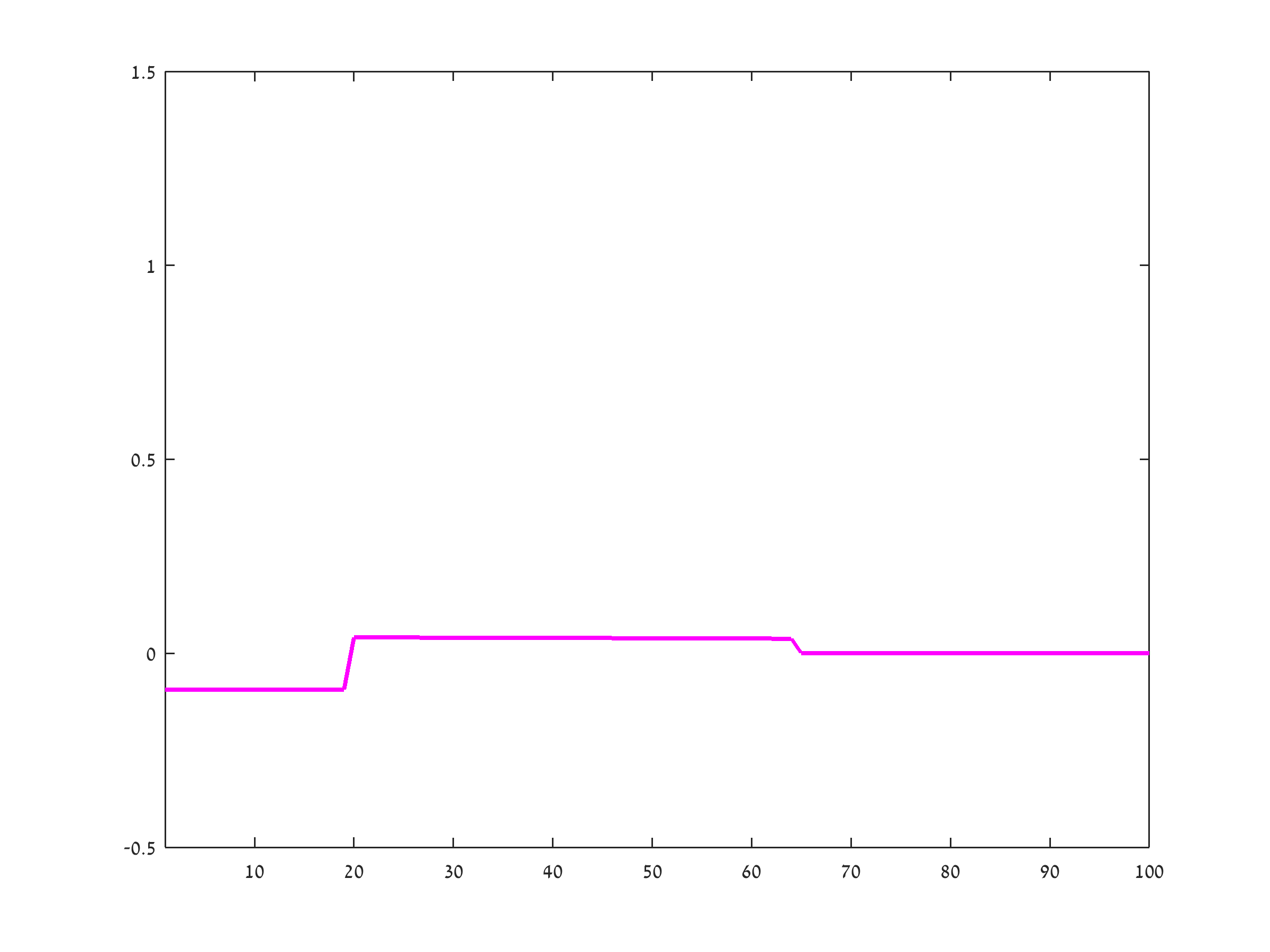} \\
    \includegraphics[width=0.25\textwidth]{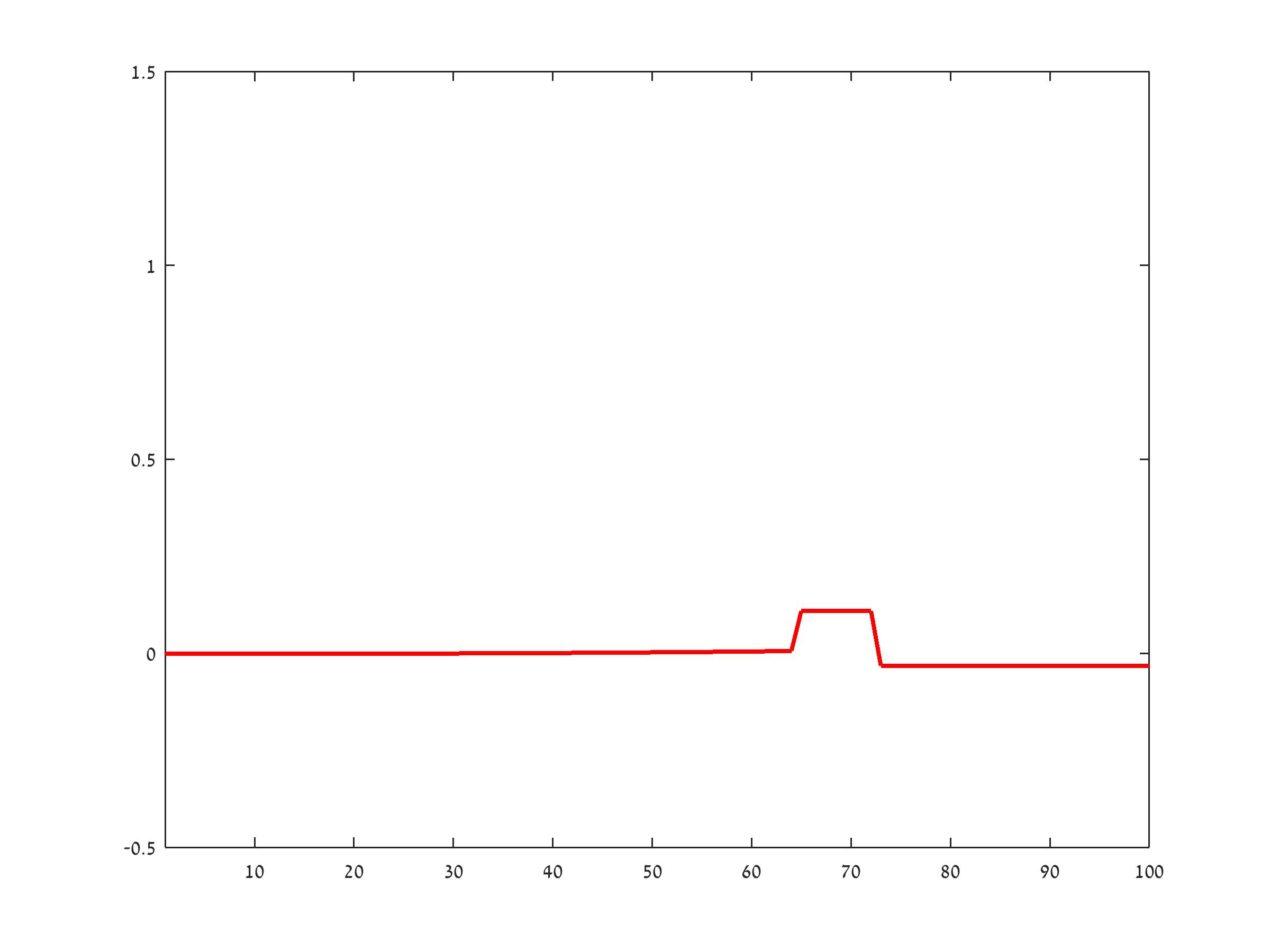} &
    \includegraphics[width=0.25\textwidth]{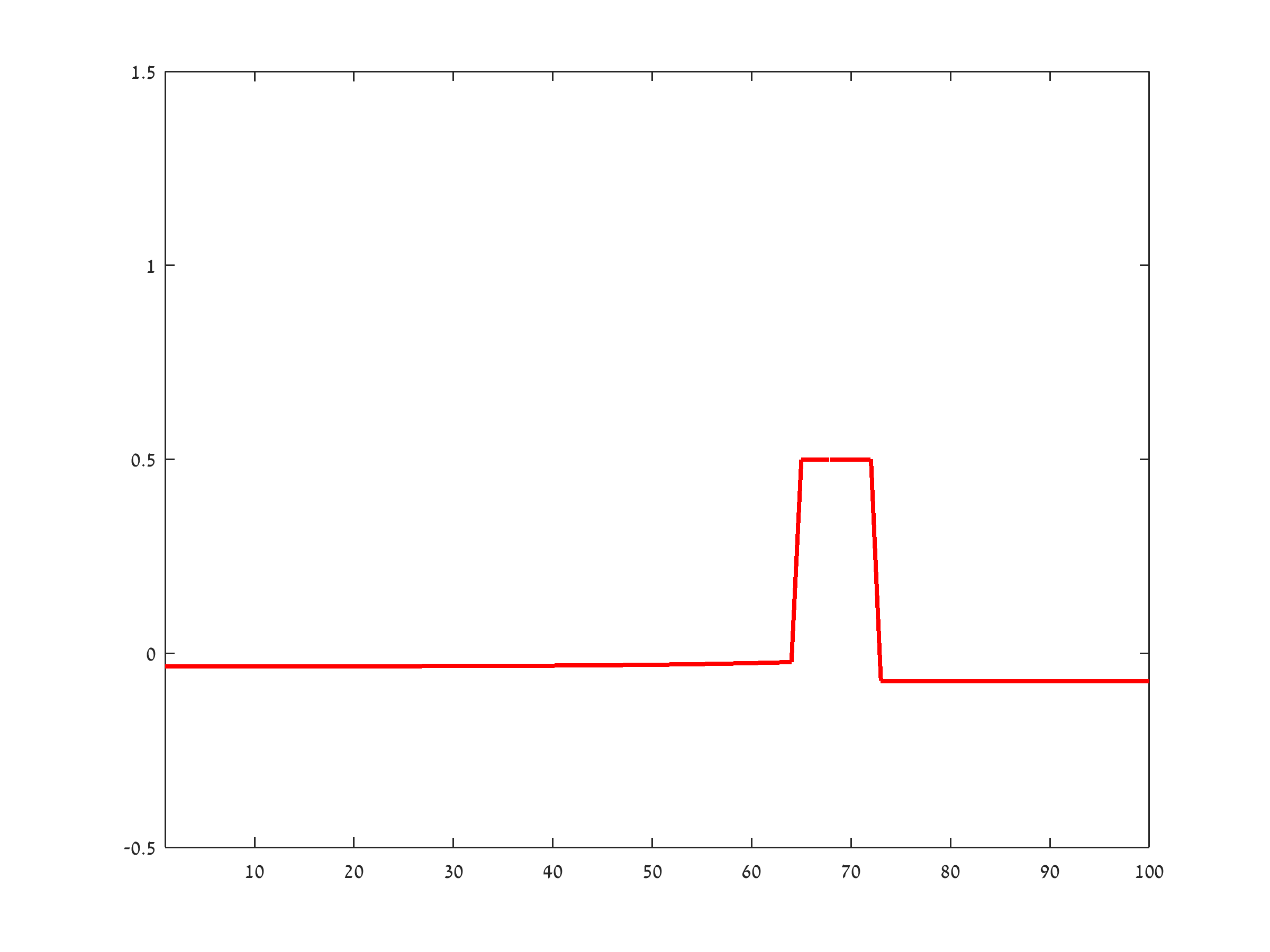} &
    \includegraphics[width=0.25\textwidth]{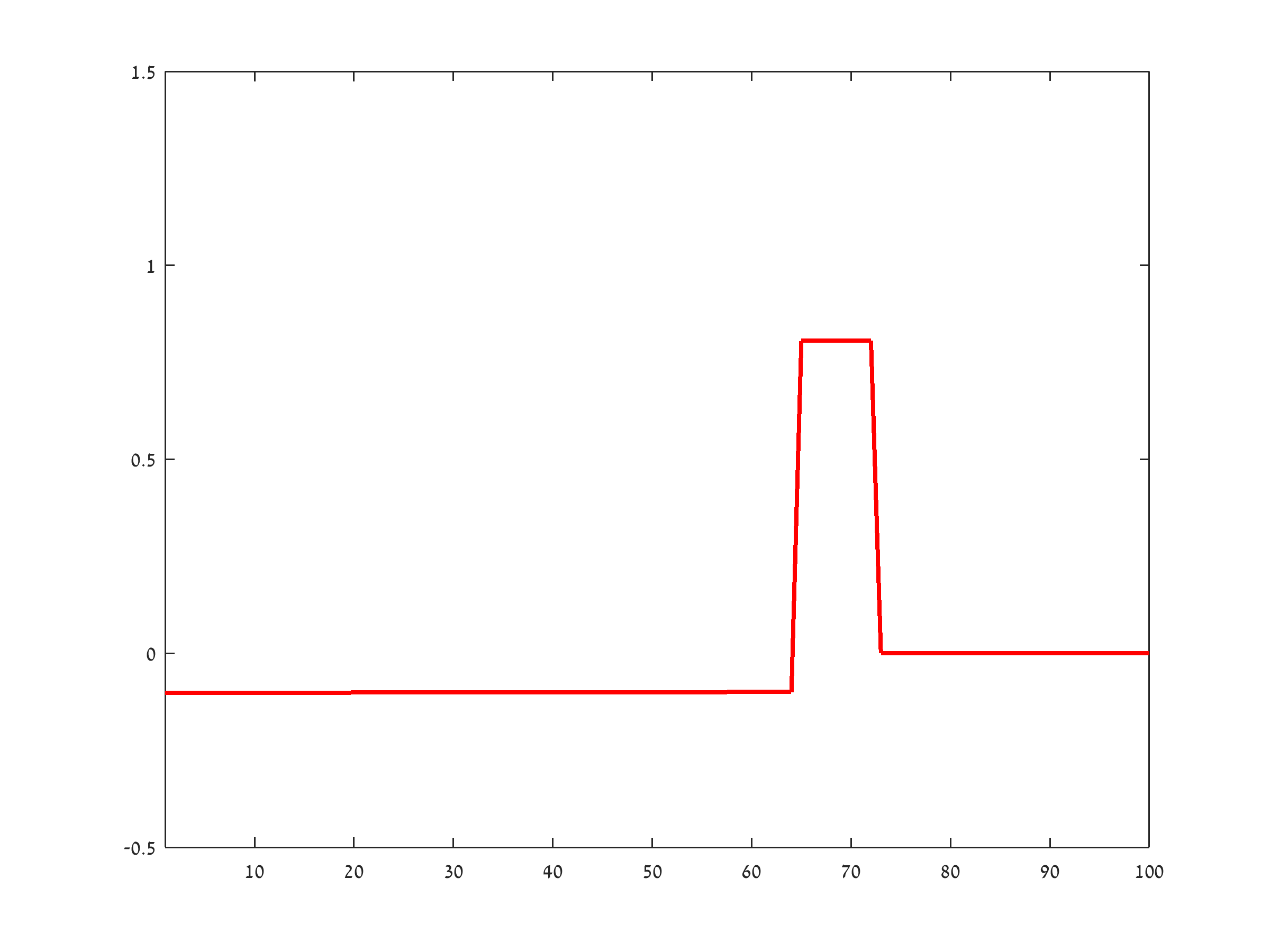} \\
    \includegraphics[width=0.25\textwidth]{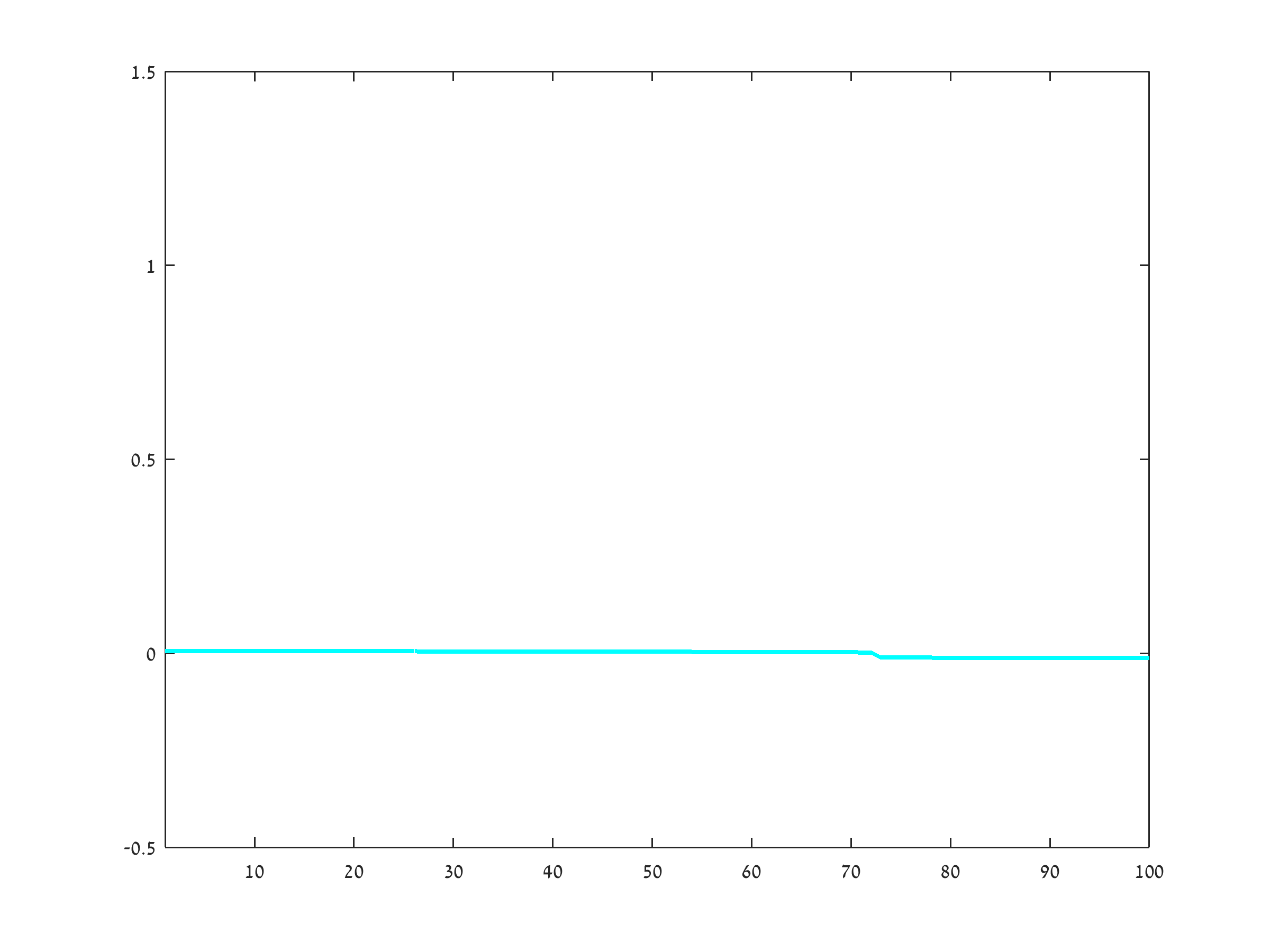} &
    \includegraphics[width=0.25\textwidth]{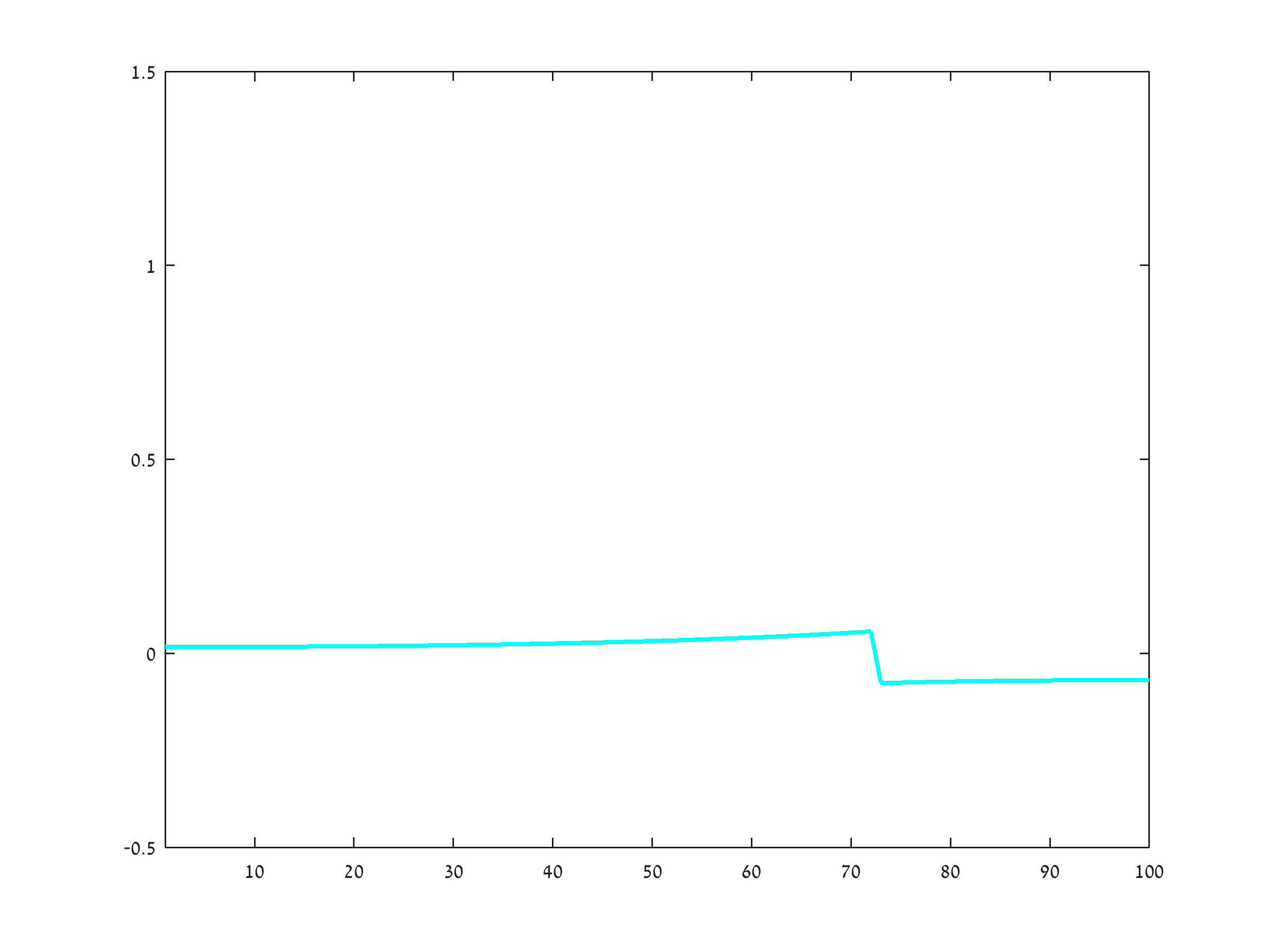} &
    \includegraphics[width=0.25\textwidth]{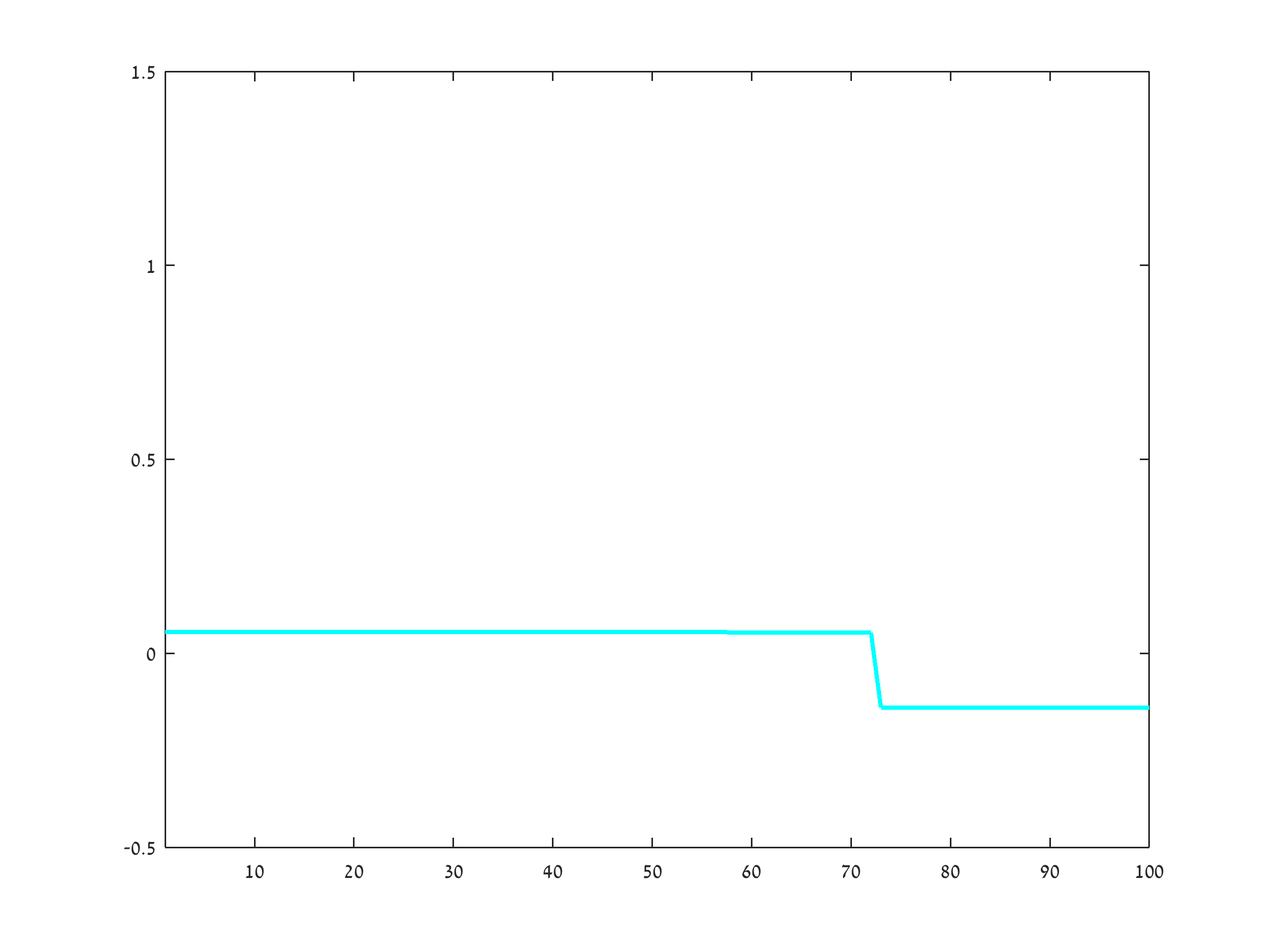} \\
	\end{tabular}
	\caption{\textbf{Decomposition example.} An illustration of Theorem \ref{thm:eigenfunctions}.
Top (from left): input signal $f$, spectrum $S_3^2(t)$ and $S_1(t)$.
From second to sixth row, $u_{GF}(t_i)$ (left), $p_{GF}(t_i)$ (center) and $\Phi(t_i)$ at time points
marked by circles in the spectrum plots. $\Phi(t_i)$ is an integration of $\phi_{GF}(t)$ between times
$t_i$ and $t_{i+1}$, visualized with different colors in the spectrum plots.
}
	\label{fig:decomp}
\end{center}
\end{figure}

\begin{figure}
\begin{center}
\begin{tabular}{cc}
$\textrm{corr}(p(t_i),p(t_j))$ & $\textrm{corr}(\Phi(t_i),\Phi(t_j))$\\
\includegraphics[height=30mm]{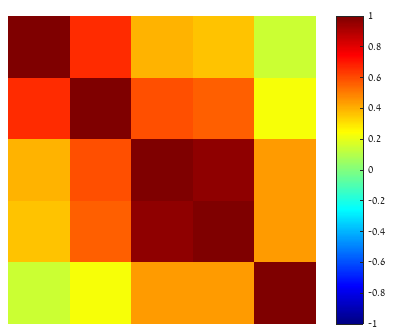}&
\includegraphics[height=30mm]{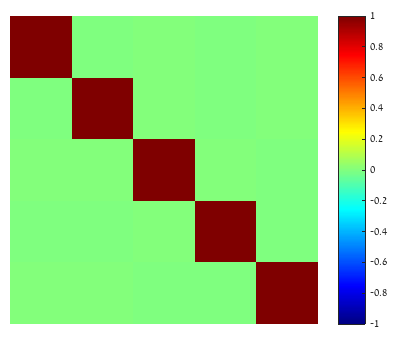}
\end{tabular}
\caption{Illustration of Thm. \ref{thm:orthogonality} (Orthogonal Decompositions). Correlation matrix of the $p$ (left) and the $\Phi$ elements in the example shown in Fig. \ref{fig:decomp}.
Whereas the $p$ elements are correlated the $\Phi$ elements are very close numerically to being orthogonal.}
\label{fig:correlation}
\end{center}
\end{figure}

\section{Numerical Results}
\label{sec:NumericalResults}
In this section we illustrate the qualitative properties of the new spectra definitions and show an example of a decomposition into eigenfunctions.

\subsection{Comparison of Spectrum Definitions}
In Fig. \ref{fig:spectrum} the 3 spectra definitions for $S_1$, $S_2$ and $S_3$, given in Eqs. \eqref{eq:S1}, \eqref{eq:S2} and \eqref{eq:S3}, respectively, are compared for a synthetic and for a natural image using isotropic total variation regularization.
The synthetic image (top left) consists of 3 (approximate) disks of different contrast comprising
of 3 peaks in the spectrum (``numerical deltas''). The spectra are computed for the gradient flow.
The respective spectra, using all definitions, are shown on the bottom left depicting $S_1(t)$ (blue),
$S_2^2(t)$ (red) and $S_3^2(t)$ (green). The peaks of $S_2$ and $S_3$ seem to be more concentrated at the singularities. On the right a natural image is processed (top) where one can observe the qualitative
distinction of $S_1$ from $S_2$ and $S_3$ which are almost identical, suggesting that the equivalence
of $S_2$ and $S_3$ under the (PS) and (MINSUB) conditions (see Prop. \ref{prop:equivalencespectra}), holds (at least approximately) also in a more general setting.

\subsection{A Decomposition Example}
In Fig. \ref{fig:decomp} the input signal $f$ (top left) consists of 3 flat peaks of different width.
The spectra $S_3^2(t)$ and $S_1(t)$ are shown on the top row (middle and right, respectively) and
a detailed visualization of the TV gradient flow is shown on rows 2-6 with the time-points $t_i$ shown
on the spectra plots by black circles. For $u_{GF}(t_i)$ and $p_{GF}(t_i)$ the actual functions are
plotted at those time points, whereas $\Phi(t_i)$ is a time integrated version, corresponding to the time
intervals depicted by 5 colors in the spectra plots (blue, green, magenta, red and cyan).
This is a 1D discrete case, so it meets the conditions of Theorems \ref{thm:eigenfunctions}
and \ref{thm:orthogonality}. Note that though the functions $p_{GF}$ are eigenfunctions they
are highly correlated to each other, as can be expected by the nature of a gradient flow. The functions $\Phi(t_i)$ are very close to orthogonal to each other and thus can be
considered as the decomposition of the signal into its orthogonal components.

This can be most clearly seen in Fig. \ref{fig:correlation} where the empirical correlation
$$ \textrm{corr}(u,v) :=  \frac{\langle u,v \rangle}{\|u\|\|v\|}$$
is computed for all pairs $(p_{GF}(t_i),p_{GF}(t_j))$, $i,j=\{1, 2..\, 5\}$ (left) and
all  pairs $(\Phi(t_i),\Phi(t_j))$ (color visualization of $5\times 5$ matrices). Note that $-1 \le \textrm{corr}(u,v) \le 1$ and that both $p_{GF}$ and $\Phi$ are with zero mean.
We can clearly see that all $p_{GF}$ functions have a strong positive correlation with each other, whereas the $\Phi$ functions are with a correlation which is very close to zero (depicted by the green color) for all $i\ne j$.

\section{Conclusions}
\label{sec:Conclusions}
We have analyzed the definitions of a nonlinear spectral decomposition based on absolutely one-homogeneous regularization functionals via scale space flows, variational methods, and inverse scale space flows in detail and demonstrated that such decompositions inherit many properties from the classical linear setting. We have proven that the spectrum of all three methods merely consists of finitely many delta peaks for polyhedral regularizations. Relations between the different methods, such as the equivalence of the variational and scale space method under the (MINSUB) assumption, as well as the equivalence of all three spectral decompositions under the (DDL1) assumption, were shown. The latter furthermore implies an orthogonal decomposition of the input data. A particularly interesting direction of research is the extension of eigenvector theory to a nonlinear setting. In this context we were able to show that all subgradients of the scale space flow with respect to a (DDL1) regularization are generalized eigenvectors. This implicitly establishes the existence of a basis of such eigenvectors for (DDL1) regularizations.

There are many open questions and directions of future research. In particular, we will investigate if the spectral decomposition with respect to the inverse scale space flow is equivalent to the other two decompositions under the assumption of (MINSUB) already, and the related search for regularizations that meet (MINSUB) but not (DDL1). Further questions are the consistency of filters, statements regarding the relation of the three different approaches beyond the assumptions we made, and the extension of our considerations to function spaces.

\section*{Acknowledgements}
The authors would like to thank Emanuele Rodola for interesting discussions on the spectral decomposition framework. 
MB acknowledges support by ERC via Grant EU FP 7 - ERC Consolidator Grant 615216 LifeInverse.
GG acknowledges support by the Israel Science Foundation (grant No. 718/15) and by the Magnet program of the OCS, Israel Ministry of Economy, in the framework of Omek Consortium.
MM and DC were supported by the ERC Starting Grant ConvexVision.

\bibliographystyle{siam}
\bibliography{refs}

\end{document}